\documentclass[reqno,11pt]{amsart}

\usepackage[utf8]{inputenc} 
\usepackage{fontenc,xcolor}    
\usepackage{hyperref}       
\usepackage{url}            
\usepackage{booktabs}       
\usepackage{nicefrac}       
\usepackage{microtype}      
\usepackage{graphicx}
\usepackage{subfig}
\usepackage{amsmath,amssymb}
\usepackage{enumitem}
\usepackage{bbm}
\usepackage{dsfont}
\usepackage{cancel}
\usepackage[sort]{cite}

\numberwithin{equation}{section}
\newtheorem{theorem}{Theorem}[section]
\newtheorem{lemma}[theorem]{Lemma}
\newtheorem{remark}{Remark}[section]

\newtheorem{definition}[theorem]{Definition}

\newtheoremstyle{nonitalic}
{3pt}
{3pt}
{\normalfont}
{}
{\bfseries}
{}
{.5em}
{}

\theoremstyle{nonitalic}

\newcounter{hypothesis}

\newtheorem{hyp}[hypothesis]{}

\DeclareMathOperator{\sign}{sign}

\DeclareMathOperator{\supp}{supp}

\DeclareMathOperator*{\esssup}{ess\,sup}

\DeclareMathAlphabet\mathbfcal{OMS}{cmsy}{b}{n}
\DeclareMathOperator{\DIV}{div}
\DeclareMathOperator{\BV}{BV}

\newcommand{\as}[1]{\left\vert#1\right\vert}
\newcommand{\magg}[1]{\left\vert#1\right\vert ^{2}}
\newcommand{\norm}[1]{\left\Vert#1\right\Vert}

\newcommand{\grad}{\nabla}
\newcommand{\p}{\partial}
\newcommand{\dx}{{\rm d}x}
\newcommand{\Rd}{\mathbb{R}^{d}}
\newcommand{\diff}{\mathop{}\!\mathrm{d}}

\usepackage{pdfrender}

\makeatletter
\let\normalrender\PdfRender@NormalColorHook
\let\PdfRender@NormalColorHook\@empty

\makeatother
\pdfrender{StrokeColor=black,TextRenderingMode=2,LineWidth=0.25pt}

\title[short title]{Well-posedness of aggregation-diffusion systems with irregular kernels}

\author{José A. Carrillo}
\author{Yurij Salmaniw}
\author{Jakub Skrzeczkowski}

\address{Mathematical Institute, University of Oxford, Woodstock Road, Oxford, OX2 6GG, United Kingdom}
\email{carrillo@maths.ox.ac.uk}
\email{yurij.salmaniw@maths.ox.ac.uk}
\email{jakub.skrzeczkowski@maths.ox.ac.uk}



\hypersetup{
pdftitle={Aggregation-diffusion equations with irregular kernels},
pdfsubject={},
pdfauthor={},
pdfkeywords={},
}

\begin{document}

\begin{abstract}
We consider aggregation-diffusion equations with merely bounded nonlocal interaction potential $K$. We are interested in establishing their well-posedness theory when the nonlocal interaction potential $K$ is neither differentiable nor positive (semi-)definite, thus preventing application of classical arguments. We prove the existence of weak solutions in two cases: if the mass of the initial data is sufficiently small, or if the interaction potential is symmetric and of bounded variation without any smallness assumption. The latter allows one to exploit the dissipation of the free energy in an optimal way, which is an entirely new approach. Remarkably, in both cases, under the additional condition that $\nabla K\ast K$ is in $L^2$, we can prove that the strong solution is unique. When $K$ is a characteristic function of a ball, we construct the classical unique solution. Under additional structural conditions we extend these results to the $n$-species system.
\end{abstract}

\maketitle



\section{Introduction}

Aggregation-diffusion equations and systems have been used extensively as a tool to model mean-field approximations of interacting agents or particles in a wide array of scientific disciplines, including fluid dynamics \cite{paterson1981first}, environmental science \cite{bear2013dynamics}, chemical engineering \cite{bird2007transport}, physics \cite{barends1997groundwater}, phase separation in materials science \cite{HP06,MR1453735}, cell-cell adhesion in biology \cite{MR3783102,MR3948738,ButtenschoenHillen2021,Buttenschoen2018SpaceJump,falco2022local}, crowd dynamics \cite{muntean2014collective,MR1698215}, biological aggregation \cite{MR2257718} and ecology \cite{cantrell2003spatial}. Their ubiquity can be explained by their direct connection with discrete dynamics described by stochastic differential equations \cite{Carrillo2020LongTime,kolokolnikov2013emergent, motsch2014heterophilious,pareschi2013interacting,MR3858403}, random walk processes \cite{potts2016territorial, pottslewis2016} and singular limits connecting them with higher-order PDEs \cite{elbar2023limit, elbar-skrzeczkowski, carrillo2023degenerate}.\\

In the present work, we consider the general scalar aggregation-diffusion equation
\begin{align}\label{eq:mainscalar}
\begin{cases}
        \frac{\p u}{\p t} = \grad \cdot \left( D \grad u + u \grad ( K * u ) \right), \\
        u(0,x) = u_0(x),
\end{cases}
\end{align}
and more generally the $n$-species counterpart
\begin{align}\label{maineq}
\begin{cases}
        \frac{\p u_i}{\p t} = \grad \cdot \left( D_i \grad u_i + u_i \sum_{j=1}^n \grad ( K_{ij} * u_j ) \right),  \\
        u_i(0,x) = u_{i0}(x) ,
\end{cases}
\end{align}
where $*$ denotes a spatial convolution
\begin{align}
    K * v (t,x) := \int_{\mathbb{R}^d} K(x-y)\, v (t,y) \diff y
\end{align}
for some appropriate kernel(s) $K(\cdot)$ and spatial dimension $d\geq1$. The coefficients $D_i >0$ with $i=1,\ldots,n$ describe the rate at which the $i^{th}$ population disperses via linear diffusion while the interaction term describes the attraction/repulsion of population $u_i$ to/from population $u_j$ depending on the characteristics of the kernel $K_{ij}$. \\

The primary purpose of this work is to establish a robust well-posedness theory (global existence \& uniqueness) for irregular interaction potentials, with the motivating prototypical form being the top-hat kernel in one dimension:
\begin{equation}\label{detectionkernelT}
K(x) =  \begin{cases} 
-\frac{\alpha}{2R} , \hspace{0.5cm}- R \leq x \leq R, \cr
0, \hspace{1.0cm} \text{otherwise}, 
\end{cases} 
\end{equation}
where $\alpha \in \mathbb{R}$ is the strength of attraction ($\alpha>0$) or repulsion ($\alpha < 0$).\\

Despite the lack of well-posedness theory for \eqref{eq:mainscalar} with the top-hat kernel, \eqref{eq:mainscalar} has been used extensively as a minimal example of nonlocal animal interactions that can be explored analytically from the perspective of pattern formation via linear stability analysis, see e.g., \cite{Fagan2017perceptual, potts2016territorial, pottslewis2016, pottslewis2019, giunta2022detecting}. In some cases, we obtain the top-hat kernel through derivation \cite{pottslewis2016}; in other cases, we use the top-hat kernel to explore the transition between platykurtic (no tails) to leptokurtic (fat tailed) kernels \cite{Fagan2017perceptual}. As demonstrated in \cite{pottslewis2019}, several animal movement models which incorporate spatial memory (e.g., of marks left on a landscape \cite{lewismurray1993, moorcroftlewis2006}, of locations of territorial interactions \cite{potts2016territorial}, or of memory of previously visited locations \cite{liu2023biological}) can be reduced to system \eqref{maineq} through an appropriate quasi-steady state approximation. System \eqref{maineq} can then be viewed as a minimal model describing interacting populations with perception and implicit spatial memory.\\

In this work we provide a well-posedness theory for a wide class of low-regularity kernels, including \eqref{detectionkernelT}. Our work solves an open problem posed in \cite{wangsalmaniw2022}, and answers in the affirmative a conjecture made in \cite{pottslewis2019}. Importantly, we note that \eqref{detectionkernelT} falls outside the scope of any recent works on the topic: \cite{jungel2022nonlocal} deals with positive definite kernels (note that \eqref{detectionkernelT} is not positive definite as its Fourier transform is $\frac{\sin(Rx)}{Rx}$) while \cite{Giunta2021local} deals with kernels which are twice continuously differentiable. From the regularity theory point of view, equation \eqref{eq:mainscalar} includes two terms of opposite effect: a regularizing linear diffusion and irregular advection term which does not fall into the classical Cauchy-Lipschitz theory nor the theory of renormalized solutions \cite{MR2096794, MR1022305}. Our work answers the question when the regularizing effect of linear diffusion is sufficiently strong that it can smooth out the advection driven by the velocity $\nabla K \ast u$.\\

We highlight some relevant efforts related to the current objectives. If $K \in W^{2,\infty} (\mathbb{T}^d)$ and $n=1$, the well-posedness of \eqref{eq:mainscalar} is fairly classical and we refer to \cite{Carrillo2020LongTime,chazelle2017wellposedness} and references therein. Concerning the case of the $n$-species systems, the authors in \cite{Giunta2021local} construct the unique mild solution under the assumption that $K_{ij} = \alpha_{ij} K$, $\alpha_{ij} \in \mathbb{R}$, $i,j = 1,\ldots,n$, with $K$ twice differentiable and $\grad K \in L^\infty (\mathbb{T}^d)$. They use a semigroup theory approach, applying a contraction mapping argument with Duhamel's formula. The obtained solution exists globally in time only for $d=1$. The regularity assumption on the kernel $K$ has been substantially relaxed in \cite{jungel2022nonlocal}, where weak solutions are obtained for the $n$-species case when $D_i = \sigma>0$, $i=1,\ldots,n$, and the kernels $K_{ij} \in L^s (\mathbb{T}^d)$ for some $s\geq1$ (depending on the dimension $d$) are positive-definite and satisfy the so-called \textit{detailed balance} condition (see Hypothesis \textbf{\ref{h5}}). Their approach uses entropy estimates, where the positive-definiteness is required to control the dissipation of the entropy. Finally, \eqref{eq:mainscalar} can be considered with singular kernels in the whole space. For instance, the particular important case with $K$ being the Newtonian potential leads to the celebrated and well-studied Patlak-Keller-Segel system \cite{MR2436186, MR2793895, MR2914242, MR2481820, MR2249579,MR2383208,MR2103197, MR2226917,MR2876403,MR3024094,MR3365970,MR3466844,MR3196188,MR2227752,MR2099126}.\\ 

Let us now discuss the sharpness of our results. First, regarding the existence of weak solutions, classical blow-up results for the Keller-Segel system \cite{MR2099126,MR2103197,LS19} indicate that some regularity of $K$ must be imposed to ensure global well-posedness. These examples rely on specific singularities of the kernel $K$. To the best of our knowledge, there are no known nonexistence or blow-up results for \eqref{eq:mainscalar} with an integrable kernel $K\in L^1(\Rd)$ where the nonexistence arises from a lack of derivatives rather than from a singularity of $K$. We also would like to mention \cite[Theorem 2.2]{MR2519677} which proves that for appropriate initial conditions, the solution to
$$
\frac{\p u}{\p t} + (-\Delta)^{\alpha/2}u + \nabla\cdot(u\,\nabla K\ast u) = 0
$$
with a fractional diffusion of order $\alpha$, blows up in finite time when the diffusion is too weak ($\alpha \in (0,1)$), even for fairly regular kernels such as $K(x) = e^{-|x|}$.\\

Regarding the uniqueness of solutions, there have recently been several counterexamples for the related advection–diffusion equation
\begin{equation}\label{eq:advection_diffusion_nonuniqueness}
\frac{\p u}{\p t} = \grad \cdot (u\, V) + \Delta u.
\end{equation}
In \cite{MR4940033}, the Authors, using the stochastic Lagrangian method, constructed an example of $V \in L^{\infty}(0,T; L^p(\mathbb{T}^2))$ with $p<2$ for which solutions to \eqref{eq:advection_diffusion_nonuniqueness} are not unique in the natural class $L^{\infty}((0,T)\times \mathbb{T}^2)\cap L^2(0,T;H^1(\mathbb{T}^2))$. Another construction was obtained by the convex integration method \cite{MR4138227, MR3884855}. For instance, in \cite{MR3884855}, the Authors constructed $V \in C(0,T; W^{1,p}(\mathbb{T}^d)\cap L^q(\mathbb{T}^d))$ for suitable $p$ and $q$ such that uniqueness fails in the class $C(0,T; L^{q'}(\mathbb{T}^d))$, where $q'$ is the H{\"o}lder conjugate of $q$. Let us emphasize two points. First, it does not seem immediate to adapt these constructions to the case $V = \nabla K \ast u$, and to the best of our knowledge, no nonuniqueness results are currently known for \eqref{eq:mainscalar}. Second, in the above counterexamples, uniqueness fails in broader function classes compared to our setting, where we assume $u \in L^{\infty}((0,T)\times \mathbb{R}^d)\cap L^2(0,T;H^1(\Rd))$ together with $\Delta u \in L^2((0,T)\times\Rd)$. In particular, we do not exclude the possibility of other solutions with lower regularity. Finally, we mention \cite{MR2525170,MR2068667} where the authors derive formulas for two different solutions to the Keller-Segel system in the space of measures (after the blow-up), suggesting nonuniqueness in the class of measure solutions.\\
  
Finally, we want to comment that the small-mass case (Theorem \ref{thm:existsmallmassscalar}) can be combined with approaches of \cite{MR2103197} and \cite{jungel2022nonlocal} to allow to consider even more general kernels. Namely, we can consider $K =K_1 + K_2 + K_3$ where $K_1 \in L^{\infty}(\Rd)$, $K_2$ is as in the Keller-Segel system (this can be done only in dimension $d=2$) while $K_3$ is positive-definite. The details are discussed in Appendix \ref{app:extension_general_kernel}. \\

Last but not least, it is natural to ask whether our results can be extended to the more general, nonlinear equation
$$
\frac{\p u}{\p t} = \grad \cdot \left( D \grad u^m + u \grad ( K * u ) \right)
$$
with $m>1$ (porous medium regime) or $m<1$ (fast diffusion regime). In the present paper, the linearity of the diffusion term plays a crucial role: it allows us to control $\nabla (K\ast u) = K\ast \nabla u$ solely through diffusion, without relying on higher regularity properties of $K$. Therefore, the nonlinear extension requires some new ideas and will be a topic of future research.\\

The remainder of the paper is organized as follows. In Section \ref{sect:mr}, we present the main assumptions imposed on the kernel $K$ and the main results. Then, in Section \ref{sec:scalarcase}, we focus on the scalar equation. We briefly introduce some preliminary notations and notions of a weak, strong, and classical solution for problems \eqref{eq:mainscalar}-\eqref{maineq}. We obtain apriori estimates for the scalar equation, improve these estimates with no further restriction, and then prove Theorems \ref{thm:existsmallmassscalar}-\ref{thm:existarbmassscalar}. Under additional conditions, we obtain further improved estimates which allow us to prove Theorems \ref{thm:scalarfinal}-\ref{thm:scalarclassicalsolution}. In Section \ref{sec:nspeciescase}, we generalize the approach to the general $n$-species system case: we first state the main theorems, and then follow a similar organisation to Section \ref{sec:scalarcase}. In Section \ref{sec:apps}, we discuss some relevant applications of our results to biological systems. Numerical simulations are provided for selected cases using a positivity-preserving finite-volume method \cite{MR3372289, MR4605931}.

\section{Main Results}\label{sect:mr}

\subsection{Hypotheses}

In the present work, we consider the well-posedness of system \eqref{maineq} on the whole space for any dimension $d \geq 1$ for merely bounded kernels, noting that these results hold true for the $d$-dimensional torus with no essential change to the proofs. 

Given a function $g(\cdot)$, we denote by $\tilde g (\cdot) := g(- \cdot)$ its reflection. Our hypotheses on the kernels $K_{ij}$ are as follows.

\begin{hyp}\label{h1}
     $K_{ij} \in L^1 ({\mathbb{R}^d}) \cap L^\infty ({\mathbb{R}^d})$ for all $i,j = 1,\ldots,n$;
\end{hyp}
\begin{hyp}\label{h2}
    there exists $C_{K_{ij}} > 0$ such that for all $\phi(\cdot)$ smooth there holds $\norm{\grad K_{ij} * \phi}_{L^1 ({\mathbb{R}^d})} \leq C_{K_{ij}} \norm{\phi}_{L^1 ({\mathbb{R}^d})}$ or equivalently, as proven below, $K_{ij} \in \BV({\mathbb{R}^d})$ with the total variation $\|\nabla K_{ij}\|_{\text{TV}}<\infty$ for each $i,j=1,\ldots,n$;  
\end{hyp}
\begin{hyp}\label{h3}
    $K_{ij} (\cdot)$ is symmetric about the origin for each $i,j=1,\ldots,n$;
\end{hyp}
\begin{hyp}\label{h4}
    $\grad ( \tilde K_{ij} * K_{ij} ) \in L^2 ({\mathbb{R}^d})$ for all $i,j = 1, \ldots, n$;
\end{hyp}
\begin{hyp}\label{h5}
     $K_{ij}$ are in {\rm detailed balance}, that is, there exists constants $\pi_i > 0$, $i=1,\ldots,n$, such that $\pi_i K_{ij} (\cdot) = \pi_j K_{ji} (\cdot)$ for all $i,j = 1,\ldots,n$;
\end{hyp}
\begin{hyp}\label{h6}
     $K_{ij}$ is compactly supported for each $i,j = 1,\ldots,n$.
\end{hyp}
We briefly comment on assumptions that are not standard.\\
\phantom{...}\\
\underline{Hypothesis \textbf{\ref{h2}}.} We recall from \cite{MR1857292} that $K \in \BV(\Rd)$ means that its distributional gradient $\nabla K$ is a bounded measure, i.e.
\begin{equation}\label{eq:TV_cond}
\|\nabla K \|_{TV} = \sup_{\varphi \in C^1(\Rd),\,  \|\varphi\|_{\infty}\leq 1} \int_{\Rd} K\, \DIV  \varphi \diff x < \infty.
\end{equation}
For example, when $K = \mathds{1}_{U}$ and $U \subset \Rd$ is a bounded set with smooth boundary, the condition is satisfied because
$$
\|\nabla K \|_{TV} := \sup_{\varphi \in C^1(\Rd),\, \|\varphi\|_{\infty}\leq 1} \int_{\partial U} \langle \varphi, {\bf{n}}\rangle \diff S \leq |\partial U| < \infty. 
$$
Let us check the equivalence between the statements in hypothesis \textbf{\ref{h2}}. Indeed, let $\psi \in L^{\infty}(\Rd)$ with $\|\psi\|_{L^{\infty}(\Rd)} \leq 1$ and assume \eqref{eq:TV_cond} holds. Then,
$$
\int_{\Rd} \grad K * \phi \, \psi \diff x = - \int_{\Rd} K \, \DIV \phi * \psi \diff x \leq  \|\nabla K \|_{\text{TV}}\, \| \phi * \psi\|_{L^{\infty}(\Rd)} \leq  \|\nabla K \|_{\text{TV}}\, \| \phi \|_{L^1(\Rd)}.
$$
Taking supremum over all $\psi \in L^{\infty}(\Rd)$ with $\|\psi\|_{L^{\infty}(\Rd)} \leq 1$ we get 
\begin{equation}\label{eq:assumption_H2_with_explicit_constant}
\norm{\grad K * \phi}_{L^1 ({\mathbb{R}^d})} \leq \| \nabla K\|_{TV}\, \norm{\phi}_{L^1 ({\mathbb{R}^d})}.
\end{equation}
Conversely, assume this condition holds. Let $\{\omega_{\varepsilon}\}_{\varepsilon>0}$ be a standard mollifier. We have
$$
\int_{\Rd} K\ast \omega_{\varepsilon} \DIV  \varphi \diff x =\int_{\Rd} \nabla K \ast \omega_{\varepsilon} \, \varphi \diff x  \leq C_{K} \, \|\varphi\|_{L^{\infty}(\Rd)} 
$$
so passing to the limit $\varepsilon \to 0$ and taking supremum over all $\varphi \in C^1(\Rd)$ with $\|\varphi\|_{L^{\infty}(\Rd)} \leq 1$ we deduce \eqref{eq:TV_cond}.\\

\underline{Hypothesis \textbf{\ref{h4}}} can be better viewed in the Fourier space. Indeed, with $K = K_{ij}$ and $\widehat{K}$ its Fourier transform, it is sufficient by Plancherel's theorem to check that there holds
\begin{equation}\label{eq:simpler_version_of_H4}
\int_{|\xi|>1} |\xi|^2 \, |\widehat{K}(\xi)|^4 \diff \xi < \infty, 
\end{equation}
where we ignore the set $\{|\xi|\leq 1\}$ as the Fourier transform of an integrable function is bounded. Condition \eqref{eq:simpler_version_of_H4} is satisfied for $K(x) = \mathds{1}_{Q}$ where $Q \subset \Rd$ is a bounded cube since one can easily compute $|\widehat{K}(\xi)| \approx {C}\,{
|\xi|^{-d}}$ for large $\xi$. Condition \eqref{eq:simpler_version_of_H4} is also satisfied by $K(x) = \mathds{1}_{B_r}$ where $B_r \subset \Rd$ is a ball of radius $r$ since $\widehat{K}(\xi) = r^{d/2}\,|\xi|^{-d/2}\, J_{d/2}(r\,|\xi|)$ where $J_{d/2}$ is a Bessel function of the first kind (see, e.g., \cite[Lemma 12.2]{wendland2004scattered}) and it is known that for large arguments $\xi$ we have $|J_{d/2}(\xi)| \approx C\,|\xi|^{-1/2}$ (see, e.g., \cite[Prop. 5.6]{wendland2004scattered}).\\

We also comment that hypothesis \textbf{\ref{h4}} is implied by \textbf{\ref{h1}} and \textbf{\ref{h2}}. Indeed, as above, we consider function $K \ast K \ast \omega_{\varepsilon}$ which is clearly differentiable. Let $\psi$ be such that $\|\psi\|_{L^2(\Rd)} \leq 1$. Then,
\begin{align*}
\left|\int_{\Rd} \nabla K \ast K \ast \omega_{\varepsilon} \, \psi \diff x\right| &= \left|\int_{\Rd} K \, \nabla K \ast \omega_{\varepsilon} \ast \psi \diff x\right| \leq \|K\|_{L^2(\Rd)} \, \| \nabla K \ast \omega_{\varepsilon} \ast \psi \|_{L^2(\Rd)} \\[2mm] &\leq  \|K\|_{L^2(\Rd)} \, \| \nabla K \ast \omega_{\varepsilon}\|_{L^1(\Rd)} \leq \|K\|_{L^2(\Rd)} \, \|\nabla K\|_{\text{TV}} 
\end{align*}
by \eqref{eq:assumption_H2_with_explicit_constant}. It follows that $\|\nabla K \ast K \ast \omega_{\varepsilon}\|_{L^2(\Rd)} \leq \|K\|_{L^2(\Rd)} \, \|\nabla K\|_{\text{TV}}$ so that passing to the limit $\varepsilon\to0$ we deduce that $\nabla K \ast K$ exists as a function in $L^2(\Rd)$.

\underline{Hypothesis \textup{\textbf{\ref{h5}}}}: in the scalar equation, this condition is superfluous. For $n\geq2$ interacting species, it is a moderate generalization of a simpler symmetry condition $K_{ij} = K_{ji}$ which ensures that the system has a gradient flow structure \cite{jungel2022nonlocal}. A detailed balance condition has also been used to study local cross-diffusion systems, see e.g., \cite{ChenDausJungel2018, DausDesvillettesDietert2019}. We explore this condition in more detail in Section \ref{sec:apps}.

We can roughly divide our results into the following:

\begin{itemize}
    \item[i)] existence of a global weak solution for small initial mass under hypothesis \textup{\textbf{\ref{h1}}} only, and for arbitrary initial mass under hypotheses \textup{\textbf{\ref{h1}}}-\textup{\textbf{\ref{h3}}};
    \item[ii)] global existence of a unique, classical solution under the additional hypothesis \textup{\textbf{\ref{h4}}} and \textup{\textbf{\ref{h6}}}; 
    \item[iii)] generalization of these results to treat the $n$-species system.
\end{itemize}

Throughout the manuscript we always assume that the initial data satisfies the mild condition
$$
0\lneqq u_{i0} \in L^1(\mathbb{R}^d),
$$ 
and has finite entropy and finite second moment (defined in \eqref{q:entropy} and \eqref{eq:second_moment}, respectively) for each $i=1,\ldots,n$. Formally, for a smooth, nonnegative solution $\mathbf{u} = (u_1, \ldots, u_2)$, this problem enjoys a conservation of mass $m_i$ for each component $u_i$:
\begin{align*}
    \frac{{\rm d}}{{\rm d}t} \norm{u_i(t,\cdot)}_{L^1 ({\mathbb{R}^d})} = 0 \Rightarrow \norm{u_i (t,\cdot)}_{L^1 ({\mathbb{R}^d})} = \norm{u_{i0}}_{L^1 ({\mathbb{R}^d})} =: m_i > 0,
\end{align*}
for all $t>0$. In the scalar case $n=1$, we simply write $m_1 = m$. 

\subsection{Statement of Main Results}

Our main results for the scalar equation are as follows.
\begin{theorem}[Existence of global weak solution for scalar equation with small mass]\label{thm:existsmallmassscalar}
    Assume \textup{\textbf{\ref{h1}}} holds and that 
    \begin{align}\label{const:c1}
            m \norm{K}_{L^\infty(\mathbb{R}^d)} < D,
    \end{align}
where $m := \norm{u_0}_{L^1(\mathbb{R}^d)}$. Then, there exists a global weak solution $u \geq 0$ in $Q_T$ solving problem \eqref{eq:mainscalar} in the sense of Definition \ref{def:weaksoln}. Moreover, if $u_0 \in L^p({\mathbb{R}^d})$ then $u \in L^\infty(0,T; L^p({\mathbb{R}^d}))$ for any $1 \leq p < \infty$, and if $u_0 \in L^2({\mathbb{R}^d})$ then $\grad u \in L^2(0,T; L^2({\mathbb{R}^d}))$. 
\end{theorem}

\begin{theorem}[Existence of global weak solution for scalar equation with arbitrary mass]\label{thm:existarbmassscalar}
    Assume \textup{\textbf{\ref{h1}}}-\textup{\textbf{\ref{h3}}} hold. Then, there exists a global weak solution $u\geq0$ in $Q_T$ solving problem \eqref{eq:mainscalar} in the sense of Definition \ref{def:weaksoln}. Moreover, if $u_0 \in L^p({\mathbb{R}^d})$ then $u \in L^\infty(0,T; L^p({\mathbb{R}^d}))$ for any $1 \leq p < \infty$, and if $u_0 \in L^2({\mathbb{R}^d})$ then $\grad u \in L^2(0,T; L^2({\mathbb{R}^d}))$.
\end{theorem}

\begin{theorem}[Existence of unique strong solution]\label{thm:scalarfinal}
Assume that the initial data satisfies $u_0 \in L^{\infty}(\mathbb{R}^d)$, $\grad u_0 \in L^2(\mathbb{R}^d)$ as well as $K\ast u_0 \in W^{2,p}(\Rd)$ for some $p>2$. Under the assumptions of Theorem \ref{thm:existarbmassscalar} or under the assumptions of Theorem \ref{thm:existsmallmassscalar} together with \textup{\textbf{\ref{h4}}}, the obtained global weak solution is the unique, global strong solution solving problem \eqref{eq:mainscalar} in the sense of Definition \ref{def:strongsolution}.
\end{theorem}

\begin{theorem}[Existence of unique classical solution]\label{thm:scalarclassicalsolution}
    Assume that the initial data satisfies $u_0 \in C^3(\mathbb{R}^d)\cap L^\infty(\mathbb{R}^d)$ as well as $\nabla u_0 \in L^2(\mathbb{R}^d) \cap L^p(\Rd)$ for some $p>d+2$, and suppose one of the following:
    \begin{enumerate}
        \item in addition to the hypotheses of Theorem \ref{thm:existsmallmassscalar}, \textbf{\textup{\ref{h2}}} and \textbf{\textup{\ref{h6}}} hold;
        \item in addition to the hypotheses of Theorem \ref{thm:existarbmassscalar}, \textbf{\textup{\ref{h6}}} holds;
    \end{enumerate}
Then, the unique global strong solution is the unique, global classical solution solving problem \eqref{eq:mainscalar} in the sense of Definition \ref{def:classicalsolution}. Moreover, the solution is strictly positive over $\mathbb{R}^d$ for all $t>0$.
\end{theorem}

We remind readers that \textbf{\ref{h1}}-\textbf{\ref{h2}} imply \textbf{\ref{h4}}, and so we do not include it in the statement of Theorem \ref{thm:scalarclassicalsolution}. Under the hypotheses of Theorem \ref{thm:existarbmassscalar}, the weak solution is in fact the global strong solution under additional regularity requirements of the initial data; under the hypotheses of Theorem \ref{thm:existsmallmassscalar} only, we need to introduce \textup{\textbf{\ref{h2}}} to bound $\Delta (K*u)$ in $L^p(Q_T)$ so that we may apply the $L^p$-theory of parabolic equations; otherwise, one requires higher integrability of $\Delta u$ a priori. The difference between the two cases in Theorem \ref{thm:scalarclassicalsolution} is the removal of the symmetry condition \textup{\textbf{\ref{h3}}} at the cost of the restriction \eqref{const:c1} on the size of the initial mass $m$. In either case, we introduce \textup{\textbf{\ref{h6}}} to obtain additional regularity and positivity via the maximum principle.

Let us briefly explain the proofs of Theorems \ref{thm:existsmallmassscalar}--\ref{thm:scalarfinal}. The existence results are obtained via compactness, departing from regularized equations. We now discuss the main apriori estimates leading to the proof of Theorem \ref{thm:existsmallmassscalar}. We first analyse the evolution of the entropy functional
\begin{align}\label{q:entropy}
    H[u] := \int_{\mathbb{R}^d} u \log u \, \dx.
\end{align}
More precisely, differentiating $H[u(t)]$ along solutions of \eqref{eq:mainscalar}, one can show that
\begin{align*}
\frac{{\rm d}}{{\rm d}t} H[u(t)] + D\, \int_{\Rd} \frac{|\nabla u|^2}{u} \diff x 
\leq m\, \|K\|_{L^{\infty} (\Rd)} \,  \int_{\Rd} \frac{|\nabla u|^2}{u} \diff x.
\end{align*}
Classical arguments using the evolution of the second moment allow us to estimate the entropy from below and control the entropy in bounded time intervals using the small mass condition.\\

\noindent Theorem \ref{thm:existarbmassscalar} is a more subtle consequence of the gradient flow structure and so we can only prove it for symmetric kernels (see \textbf{\ref{h3}}). More precisely, \eqref{eq:mainscalar} is the Wasserstein gradient flow of the free energy functional, i.e. the weighted sum of the entropy and the interaction energy $\mathcal{K}[u]$ of the system, given by 
\begin{align}\label{q:FreeEnergy}
    \mathcal{F}[u] := D \int_{\mathbb{R}^d} u \log u \, \diff x+ \frac12 \int_{\mathbb{R}^d} u (K*u) \, \diff x=: D H[u]+ \mathcal{K}[u].
\end{align}
The dissipation of $\mathcal{F}[u(t)]$ is given by
$$
\frac{{\rm d}}{{\rm d}t} \mathcal{F}[u(t)] = -\|f(t)\|_{L^2 (\Rd)}^2 \qquad \text{and} \qquad f:= \sqrt{u}\, \grad (D \log u + K  * u).
$$
Note that this allows one to estimate $\nabla u$ in $L^2(0,T; L^1(\mathbb{R}^d))$ since $D \nabla u = f\, \sqrt{u} - u\, \nabla K \ast u$. Indeed, the first term can be bounded using the dissipation of the free energy and the conservation of mass. The second term $u \nabla K \ast u$ can be split into small values of $u$, controlled by the bounded variation of $K$ in \textbf{\ref{h2}}, and large values of $u$, controlled by the uniform equi-integrability estimate from the entropy. We conclude that $\nabla \sqrt{u} \in L^2(Q_T)$, which is precisely the information obtained in the small-mass case.\\

Improved regularity is first based on the novel condition $\nabla K \ast K \in L^2(\mathbb{R}^d)$, cf. \textbf{\ref{h4}}. This allows one to obtain better estimates on the vector field $\nabla K \ast u$ by testing \eqref{eq:mainscalar} with $K \ast K \ast \Delta u$. The main point here is that while $K$ itself may not be differentiable, $K \ast K$ enjoys regularization by convolution and may be differentiable. Once we obtain bounds on $\nabla K \ast u$ in $L^{\infty}(0,T; L^2(\Rd))$ and $\Delta K \ast u$ in $L^2(Q_T)$, we can show $u$ is bounded (see Lemma \ref{lem:scalaresthigherorder}) and uniqueness can be established. The uniqueness result is obtained by usual energy methods with a modification that we study the evolution of the mixed local-nonlocal energy
$$
\mathcal{E}(u):= \int_{\Rd} u^2 \diff x + C\,\int_{\Rd} |K\ast u|^2 \diff x
$$
for some constant $C>0$. Combining two quantities allows us to prove that if $u$ and $v$ are regular solutions to \eqref{eq:mainscalar} then $\partial_t \mathcal{E}(u-v) \leq C(t)\, \mathcal{E}(u-v)$ which implies uniqueness. Under some additional technical requirements on the initial data, we can improve the integrability of $\grad u$ and ultimately apply Schauder estimates to obtain the classical differentiability of the solution.\\

Finally, the methods discussed above are flexible enough to be extended to the case of the $n$-species systems, see Theorems \ref{thm:existsmallmasssystem}-\ref{thm:systemfinal} in Section \ref{sec:nspeciescase}.

\section{The Scalar Equation}\label{sec:scalarcase}

In this section, we will prove well-posedness in the scalar case, Theorems \ref{thm:existsmallmassscalar}-\ref{thm:scalarfinal}. This will simplify the exposition of the general $n$-species case. We start by introducing some notations needed and the notion of weak solution we will work with directly in the general $n$-species case.

\subsection{Preliminaries}

We consider our problem over ${\mathbb{R}^d}$ for $d \geq 1$, and denote by $Q_T := {\mathbb{R}^d} \times (0,T)$ for some $T>0$ fixed. For $1 \leq p \leq \infty$, $L^p ({\mathbb{R}^d})$ denotes the usual Lebesgue space over the spatial domain ${\mathbb{R}^d}$, with $\norm{u}_{L^\infty({\mathbb{R}^d})}$ the essential supremum of $u$ over ${\mathbb{R}^d}$. For $1 \leq p,q < \infty$, $L^q ( 0,T; L^p ({\mathbb{R}^d}) )$ then denotes the typical Lebesgue space over $Q_T$ with norm
\begin{align}
    \norm{v}_{L^q (0,T; L^p ({\mathbb{R}^d}) )} := \left( \int_0 ^T \norm{v (t,\cdot)}_{L^p ({\mathbb{R}^d})} ^q {\rm d}t \right)^{1/q},
\end{align}
and when $1 \leq p \leq q=\infty$ 
\begin{align}
    \norm{v}_{L^\infty (0,T; L^p ({\mathbb{R}^d}) )} := \esssup_{t \in (0,T)} \norm{v(t,\cdot)}_{L^p({\mathbb{R}^d})}.
\end{align}

In cases where $p=q$, for notational brevity we simply write $L^p(Q_T) := L^p(0,T; L^p({\mathbb{R}^d}))$. We denote by $C^{1+\sigma/2,2+\sigma}_{loc} ([0,T] \times \mathbb{R}^d)$ the class a functions which are (locally) twice differentiable in space and once differentiable in time (in the classical sense), with its second order spatial derivatives and first order time derivative being H{\"o}lder continuous for some $\sigma \in (0,1)$. We understand a \textit{weak solution} to problems \eqref{eq:mainscalar}-\eqref{maineq} in the following sense.

\begin{definition}[weak solution]\label{def:weaksoln}
    We call $\mathbf{u} := ( u_1, u_2, \ldots, u_n)$ a \textup{weak solution} to problem \eqref{maineq} corresponding to the initial data $\mathbf{u_0} := (u_{10}, \ldots, u_{n0})$ if 
$$
{u_i} \in L^{\infty}(0,T; L^1(\mathbb{R}^d)), \quad \nabla \sqrt{u_i} \in L^{2}((0,T) \times \Rd),
$$ 
\begin{equation}\label{eq:reg_time_der_value_of_exponent_q}
    \tfrac{\p u_i}{\p t} \in L^{\frac{4}{3}}(0,T; W^{-1,q} (\mathbb{R}^d)) \mbox{ with } q = \begin{cases} 
    \frac{2d}{2d-1} &\mbox{ for } d\neq 2,\\
    \frac{6}{5} &\mbox{ for } d= 2,
    \end{cases}
\end{equation}
    for all $i = 1,\ldots,n$, and for all test functions $\phi_i \in L^4 (0,T; W^{1,q^\prime} ({\mathbb{R}^d}))$, $i=1,\ldots,n$, there holds
\begin{align}\label{defn:weaksolution}
    \int_0^T \left< (u_i)_t , \phi_i \right> \diff t + D_i \iint_{Q_T} \grad u_i \cdot \grad \phi_i \diff x \diff t = - \iint_{Q_T} u_i \sum_{j=1}^n ( K_{ij} * \grad u_j ) \cdot \grad  \phi_i \diff x \diff t,
\end{align}
were $\left< \cdot, \cdot\right>$ denotes the dual pairing between $W^{-1,q}({\mathbb{R}^d})$ and $W^{1,q^\prime} ({\mathbb{R}^d})$ with $1/q + 1/q^\prime = 1$, and the initial data is satisfied in the sense of $W^{-1,q}(\mathbb{R}^d)$. We call $\mathbf{u}$ a \textup{global weak solution} if \eqref{defn:weaksolution} holds for all $T>0$.
\end{definition}

The following lemma shows where the exponent $q$ comes from.

\begin{lemma}\label{lem:exponent_q_weak_formulation_joint_space}
Let $u\in L^{\infty}(0,T; {L^1(\mathbb{R}^d)})$, $\nabla \sqrt{u} \in L^{2}((0,T) \times \Rd)$, and $K \in L^\infty(\mathbb{R}^d)$. Then, $u\, \nabla K\ast u$ and $\nabla u$ belong to the same space $L^{\frac{4}{3}}(0,T; L^{q}(\Rd))$ with $q$ given by \eqref{eq:reg_time_der_value_of_exponent_q} and with the norm being controlled only by the norm of $K$ in $L^{\infty}(\Rd)$ and by the norms of $u$ and $\nabla \sqrt{u}$ in $L^{\infty}(0,T; {L^1(\mathbb{R}^d)})$ and $L^{2}((0,T) \times \Rd)$, respectively.  
\end{lemma}
\begin{proof}
First, assume $d\neq2$. We analyse two terms separately.\\

\underline{Term $u\, \nabla K\ast u$.} By the Gagliardo-Nirenberg inequality, $\sqrt{u} \in L^2(0,T; L^{\frac{2d}{d-2}}(\Rd))$, and so $u$ is in $L^1(0,T; L^{\frac{d}{d-2}}(\Rd))$. Interpolating with $L^{\infty}(0,T; L^1(\Rd))$ we obtain $u \in L^{\frac{1}{\theta}}(0,T; L^{\frac{d}{d-2\theta}}(\Rd))$ for all $\theta \in [0,1]$. Since $K \in L^{\infty}(\Rd)$ and $\nabla u = 2\,\nabla \sqrt{u}\,\sqrt{u}$ is in $L^2(0,T; L^1(\Rd))$, we deduce that $\nabla K \ast u \in L^2(0,T; L^{\infty}(\Rd))$. Finally, by H{\"o}lder's inequality we obtain that $u \,\nabla K\ast u \in L^{\frac{2}{2\theta+1}}(0,T; L^{\frac{d}{d-2\theta}}(\Rd))$. The last step requires $\theta \leq \frac{1}{2}$ so that $\frac{2}{2\theta + 1}\geq 1$.\\

\underline{Term $\nabla u$.} From the computations above, $\sqrt{u} \in L^{\frac{2}{\theta'}}(0,T; L^{\frac{2d}{d-2\theta'}}(\Rd))$ for any $\theta'\in[0,1]$. By writing $\nabla u = 2\,\nabla \sqrt{u}\,\sqrt{u}$ and using H{\"o}lder's inequality we obtain $\nabla u \in L^{\frac{2}{\theta'+1}}(0,T; L^{\frac{d}{d-\theta'}}(\Rd))$ for any $\theta'\in[0,1]$.\\

We now choose $\theta = \frac{1}{4}$ and $\theta'=2\theta$ so that $\nabla u$ and $u \nabla K\ast u$ belong to the same space $L^{\frac{4}{3}}(0,T; L^{\frac{2d}{2d-1}}(\Rd))$. For $d=2$, the regularity $\sqrt{u} \in L^2(0,T; L^{\infty}(\mathbb{R}^2))$ is not attained and we only know $\sqrt{u} \in L^2(0,T; L^{\frac{2}{\delta}}(\mathbb{R}^2))$ for all $\delta \in (0,1]$. Repeating the calculations above we obtain that $\nabla u$ and $u\nabla K\ast u$ belong to the same space $L^{\frac{4}{3}}(0,T; L^{\frac{4}{3+\delta}}(\mathbb{R}^2))$. Choosing $\delta = \frac{1}{3}$ we obtain $L^{\frac{4}{3}}(0,T; L^{\frac{6}{5}}(\mathbb{R}^2))$ as desired.  
\end{proof}

In fact, for many kernels we obtain a \textit{strong} solution to problems \eqref{eq:mainscalar}-\eqref{maineq} in the following sense.
\begin{definition}[strong solution]\label{def:strongsolution}
    We call $\mathbf{u} := ( u_1, u_2, \ldots, u_n)$ a \textup{strong solution} to problem \eqref{eq:mainscalar} corresponding to the initial data $\mathbf{u_0} := (u_{10}, \ldots, u_{n0})$ if it is a bounded weak solution with
    $$
\grad u \in L^\infty(0,T; L^2(\mathbb{R}^d)); \quad \Delta u \in L^2(Q_T),
    $$
    and the initial data is satisfied in the sense of $L^2(\mathbb{R}^d)$. We call the strong solution \textup{global} if it is a strong solution for all $T>0$.
\end{definition}

Finally, we refer to a \textit{classical} solution to problems \eqref{eq:mainscalar}-\eqref{maineq} in the following sense.

\begin{definition}[classical solution]\label{def:classicalsolution}
    We call $\mathbf{u} := ( u_1, u_2, \ldots, u_n)$ a \textup{classical solution} to problem \eqref{eq:mainscalar} corresponding to the initial data $\mathbf{u_0} := (u_{10}, \ldots, u_{n0})$ if
    $$
u_i \in C^{1+\sigma/2, 2+\sigma}_{loc} ([0,T] \times \mathbb{R}^d),
    $$
    for some $\sigma \in (0,1)$. We call the classical solution \textup{global} if it is a classical solution for all $T>0$. 
\end{definition}

We now establish some apriori estimates for small mass and arbitrary mass regimes for problem \eqref{eq:mainscalar}.

\subsection{Apriori estimates}\label{sec:apriorismallmass}

We first establish some apriori estimates under the assumption that we have a classical solution $u(t,x)$ solving problem \eqref{eq:mainscalar}. In addition to the entropy $H[u]$, the interaction energy $\mathcal{K}[u]$, and the free energy functional $\mathcal{F}[u]$ defined in \eqref{q:entropy}-\eqref{q:FreeEnergy}, we also denote the second moment by:
\begin{equation}\label{eq:second_moment}
I[u(t)] = \int_{\mathbb{R}^d} u \as{x}^2 \diff x\,.
\end{equation}

With these estimates, subsequent improvements to these preliminary estimates will follow. For the regime of small mass, we have the following.

\begin{lemma}[Apriori estimates under \textup{\textbf{\ref{h1}}} and small-mass condition \eqref{const:c1}]\label{lem:iscalar}
     Fix $T>0$. Assume \textup{\textbf{\ref{h1}}} and suppose \eqref{const:c1} holds. Then for any smooth, positive solution $u$ to problem \eqref{eq:mainscalar} there holds
    \begin{align}
                  \sup_{t \in (0,T)}I[u(t)] = \sup_{t\in(0,T)} \int_{\mathbb{R}^d} u \magg{x} \diff x&\leq C; \label{est:scalar1.2prime} \\
        \norm{\grad \sqrt{u}}_{L^2(Q_T)}^2 &\leq C ; \label{est:scalar1} \\
        \norm{\sqrt{u}\, \grad K * u}_{L^2 (Q_T)}^2 &\leq  C\label{est:scalar2},\\
        \norm{\grad u}_{L^2(0,T; L^1(\mathbb{R}^d))}   &\leq C;         \label{est:scalar1.123_lembefore} 
    \end{align}
    where $C = C(D,T,c_{1,1}^{-1}, \norm{K}_{L^\infty(\mathbb{R}^d)},d,m,I[u_0],H[u_0])$, $m = \norm{u_0}_{L^1(\mathbb{R}^d)}$ and $c_{1,1} = D - m \norm{K}_{L^\infty(\mathbb{R}^d)}$.
\end{lemma}

\begin{lemma}[Apriori estimates under \textup{\textbf{\ref{h1}}}-\textup{\textbf{\ref{h3}}}]\label{lem:iscalararbmass}
    Fix $T>0$. Assume \textup{\textbf{\ref{h1}}}-\textup{\textbf{\ref{h3}}}. Then for any smooth, positive solution $u$ to problem \eqref{eq:mainscalar} there holds
    \begin{align}
     \sup_{t\in(0,T)} \int_{\mathbb{R}^d} u \magg{x} \diff x&\leq C ; \label{est:scalar1.2} \\
        \sup_{t \in (0,T)} \int_{\mathbb{R}^d} u \as{\log u} \diff x+ \tfrac{1}{2} \norm{f}_{L^2(Q_T)}^2 &\leq C; \label{est:scalar2.2} \\
         \norm{\grad u}_{L^2(0,T; L^1(\mathbb{R}^d))}   &\leq C;         \label{est:scalar1.123}           \\
                   \norm{\sqrt{u}\, \grad K * u}_{L^2(Q_T)} &\leq C; \label{est:scalar1.4} \\
        \norm{\grad \sqrt{u}}_{L^2(Q_T)} &\leq C ; \label{est:scalar1.3}
    \end{align}
    where $C = C(D,T,\norm{K}_{L^\infty(\mathbb{R}^d)},d,m,I[u_0],H[u_0])$ and $f:= \sqrt{u}\, \grad (D \log u + K  * u)$ is the dissipation.
\end{lemma}

We begin with the proof of Lemma \ref{lem:iscalar}.
\begin{proof}[Proof of Lemma \ref{lem:iscalar}]
 Consider the entropy functional $H[u(t)] = \int_{\mathbb{R}^d} u \log u \, \dx$. Taking the derivative with respect to time, using the conservation of mass, and integrating by parts we have
\begin{align}\label{est:dwdtscalar}
    \frac{{\rm d}}{{\rm d}t}H[u(t)]  &= - \int_{\mathbb{R}^d} \frac{\grad u}{u} \left( D\, \grad u + u \grad ( K * u ) \right) \diff x\nonumber \\
    &= - D \int_{\mathbb{R}^d} \frac{\magg{\grad u}}{u } \diff x-  \int_{\mathbb{R}^d} \grad u \, \grad (K * u ) \diff x.
\end{align}
We first control the second term of \eqref{est:dwdtscalar} using H{\"o}lder's inequality and Young's convolution inequality:
\begin{align}\label{est:graduKscalar}
    \as{\int_{\mathbb{R}^d} \grad u \grad (K * u) \dx}  &\leq \norm{K * \grad u (t,\cdot)}_{L^\infty ({\mathbb{R}^d})} \norm{\grad u (t,\cdot)}_{L^1 ({\mathbb{R}^d})} \nonumber \\
    &\leq \norm{K}_{L^\infty ({\mathbb{R}^d})} \norm{\grad u (t,\cdot)}_{L^1 ({\mathbb{R}^d})} ^2 .
\end{align}
Furthermore, we can estimate $\grad u$ via H{\" o}lder's inequality and the conservation of mass
\begin{align}\label{est:graduscalar}
    \norm{\grad u (t,\cdot)}_{L^1 ({\mathbb{R}^d})}^2 &= \left( \int_{\mathbb{R}^d} \sqrt{u}\, \frac{\as{\grad u}}{\sqrt{u}}  \diff x\right)^2 \leq \norm{u(t,\cdot)}_{L^1 ({\mathbb{R}^d})} \int_{\mathbb{R}^d} \frac{\magg{\grad u}}{u} \diff x = m  \int_{\mathbb{R}^d} \frac{\magg{\grad u}}{u} \diff x.
\end{align}
Combining estimates \eqref{est:graduKscalar}-\eqref{est:graduscalar}, \eqref{est:dwdtscalar} becomes
\begin{align}
    \frac{{\rm d}}{{\rm d}t} H[u(t)] + 4 c_{1,1} \norm{\grad \sqrt{u} (t,\cdot)}_{L^2 ({\mathbb{R}^d})}^2 \leq 0.
\end{align}
Integrating both sides from $0$ to $t$ yields 
\begin{align}\label{est:scalarprime1.1}
    H[u(t)] + 4 c_{1,1}\norm{\grad \sqrt{u}}_{L^2 (Q_t)}^2 \leq H[u_0] .
\end{align}
We also notice that by H{\"o}lder's inequality, Young's convolution inequality and \eqref{est:graduscalar} that
\begin{align}\label{est:iscalar9.11}
   \norm{\sqrt{u}\, \grad (K * u) (t,\cdot)}_{L^2 ({\mathbb{R}^d})}^2 = \int_{\mathbb{R}^d} u \magg{K * \grad u} \diff x&\leq m \norm{K * \grad u}_{L^\infty ({\mathbb{R}^d})}^2 \nonumber \\
    &\leq m \norm{K}_{L^\infty ({\mathbb{R}^d})}^2 \norm{\grad u (t,\cdot)}_{L^1 ({\mathbb{R}^d})}^2 \nonumber \\
    &\leq 4 m^2 \norm{K}_{L^\infty ({\mathbb{R}^d})}^2 \norm{\grad \sqrt{u} (t,\cdot)}_{L^2 ({\mathbb{R}^d})}^2
\end{align}
We seek to combine these estimates to obtain an estimate on $\norm{\grad \sqrt{u}}_{L^2 (Q_T)}^2$, but we must first control the negative part of $u \log u$. This can be achieved using the second moment $I[u(t)]$ by considering the set $ \underline{\Omega} := \{ (t,x) : u \leq 1\}$ and splitting it for two subsets $\{ u \leq e^{-\as{x}^2/2}\}$, $\{e^{-\as{x}^2/2} \leq u \leq 1 \}$:
\begin{align}\label{eq:arbitrarymass2.12prime} 
    - \int_{\mathbb{R}^d} \chi_{\underline{\Omega} }u \log u \diff x&\leq \sup_{\xi \in (0,1)} \as{\sqrt{\xi} \log \xi}\int_{\mathbb{R}^d} e^{-\as{x}^2/2} \diff x- \int_{\{e^{-\as{x}^2/2} \leq u \leq 1 \}} u \log u \diff x\nonumber \\
    &\leq C_1 + \int_{\mathbb{R}^d} u \as{x}^2  \diff x= C_1 + I[u(t)] ,
\end{align}
where $C_1>0$ depends only on the dimension $d$. Hence, we compute the time derivative of $I[u(t)]$, integrate by parts, apply Young's inequality with $\varepsilon$ and use estimate \eqref{est:iscalar9.11} to find
\begin{align}\label{est:scalarprime1.3}
    \frac{{\rm d}}{{\rm d}t} I[u(t)] &= - 2\int_{\mathbb{R}^d} x \cdot ( D \grad u + u\, K* \grad u ) \diff x\nonumber \\
    &\leq \varepsilon^{-1} I[u(t)] + \varepsilon ( 8D^2 \norm{\grad \sqrt{u} (t,\cdot)}_{L^2(\mathbb{R}^d)}^2 + 2\norm{\sqrt{u} \, \grad K*u}_{L^2(\mathbb{R}^d)}^2 )  \nonumber \\
    &\leq \varepsilon^{-1} I[u(t)] + 8\, \varepsilon ( D^2 + m^2 \norm{K}_{L^\infty (\mathbb{R}^d)}^2 ) \norm{\grad \sqrt{u} (t,\cdot)}_{L^2(\mathbb{R}^d)}^2 ,
\end{align}
where $\varepsilon>0$ is to be chosen. Integrating this result from $0$ to $t$ yields
\begin{align}\label{est:scalarprime1.5}
    I[u(t)] \leq \varepsilon^{-1} \int_0 ^t I[u(s)] {\rm d} s + 8\, \varepsilon (D^2 + m^2 \norm{K}_{L^\infty(\mathbb{R}^d)}^2 ) \norm{\grad \sqrt{u}}_{L^2(Q_T)}^2 +  I[u_0].
\end{align}
Writing
\begin{equation}\label{eq:splitting_ulogu}
\int_{\mathbb{R}^d} u \as{\log u} \diff x= \int_{\mathbb{R}^d} u \log u \diff x- 2 \int_{\mathbb{R}^d} \chi_{\underline\Omega} u \log u \diff x,
\end{equation} 
we may combine estimates \eqref{est:scalarprime1.1}, \eqref{eq:arbitrarymass2.12prime}, \eqref{est:scalarprime1.5} and \eqref{eq:splitting_ulogu} to obtain
\begin{align}\label{est:scalarprime1.7}
    \int_{\mathbb{R}^d} u \as{\log u} \diff x&\,+ 2 c_{1,1} \norm{\grad \sqrt{u}}_{L^2(Q_t)}^2 +I[u(t)] \nonumber \\
    = &\,H[u(t)] - 2 \int_{\mathbb{R}^d} \chi_{\underline\Omega} u \log u \diff x+ 2 c_{1,1} \norm{\grad \sqrt{u}}_{L^2(Q_t)}^2 +I[u(t)] \nonumber \\
    \leq &\, H[u_0] -  2 c_{1,1} \norm{\grad \sqrt{u}}_{L^2(Q_t)}^2 + 2 C_1 + 3 I[u(t)] \nonumber \\
    \leq &\,  H[u_0] + 2 C_1 + 3 \varepsilon^{-1} \int_0 ^t I[u(s)]{\rm d}s  \nonumber \\
    &\, + (24 \varepsilon ( D^2 + 4 m^2 \norm{K}_{L^\infty (\mathbb{R}^d)}^2 )  -  2 c_{1,1} )\norm{\grad \sqrt{u}}_{L^2(Q_t)}^2 + 3 I[u_0].
\end{align}
Choosing $\varepsilon$ sufficiently small, the last of \eqref{est:scalarprime1.7} is nonpositive, and so we conclude that
\begin{align}\label{est:scalarprime1.9}
    \int_{\mathbb{R}^d} u \as{\log u} \diff x+ 2 c_{1,1} \norm{\grad \sqrt{u}}_{L^2(Q_t)}^2 +I[u(t)] &\leq H[u_0] + 2C_1 + 3 \varepsilon^{-1} \int_0 ^t I[u(s)] {\rm d}s + 3 I[u_0],
\end{align}
whence Gr{\"o}nwall's lemma yields estimate \eqref{est:scalar1.2prime}. With the bound on $I[u(t)]$, we immediately obtain estimate \eqref{est:scalar1}, from which estimate \eqref{est:scalar2} follows from estimate \eqref{est:iscalar9.11}. Finally, \eqref{est:scalar1.123_lembefore} follows by writing $\nabla u = 2\, (\nabla \sqrt{u})\, \sqrt{u}$ and using the conservation of mass.
\end{proof}

Next we prove Lemma \ref{lem:iscalararbmass}.
\begin{proof}[Proof of Lemma \ref{lem:iscalararbmass}]
    We begin with estimates on $\norm{u(t,\cdot)}_{L^1(\mathbb{R}^d)}$, the second moment $I[u(t)]$, and $\int_{\mathbb{R}^d}u \as{\log u} \dx$ in $L^\infty (0,T)$. From the conservation of mass, we already have that $\norm{u}_{L^\infty (0,T; L^1({\mathbb{R}^d}))} = m$. We then write \eqref{eq:mainscalar} as a gradient flow:
$$
u_t - \grad \cdot ( u \grad( D \log u + K   * u )  ) = u_t -  \grad \cdot (\sqrt{u} f ) = 0,
$$
where $f$ is the dissipation defined in the statement of the Lemma. By direct computation and hypothesis \textup{\textbf{\ref{h3}}} there holds
\begin{align}
    \frac{{\rm d}}{{\rm d}t} \mathcal{F}[u(t)] + \int_{\mathbb{R}^d} \as{f}^2 \diff x= 0.
\end{align}
Integrating both sides from $0 \to t$, applying Young's convolution inequality and using the conservation of mass yields
\begin{align}\label{eq:arbitrarymass2.11}
   D H[u(t)] + \norm{f}_{L^2(Q_t)}^2 &\leq D H[u_0] + m^2 \norm{K}_{L^\infty ({\mathbb{R}^d})}. 
\end{align}
As in \eqref{eq:arbitrarymass2.12prime}, we may control the negative part of $u \log u$ in terms of the second moment $I[u(t)]$ which can be estimated by using the dissipation $f$ as follows
\begin{align*}
     \frac{{\rm d}}{{\rm d}t} I[u(t)] &= - 2 \int_{\mathbb{R}^d} \sqrt{u} f  \cdot x \diff x\leq  \varepsilon^{-1} \int_{\mathbb{R}^d} u \as{x}^2 \diff x+ \varepsilon\int_{\mathbb{R}^d} \as{f}^2 \diff x.
\end{align*}
After integration from $0$ to $t$ we obtain
\begin{align}\label{eq:arbitrarymass2.13}
    I[u(t)] &\leq \varepsilon^{-1}\int_0^t I[u(s)] {\rm d}s + \varepsilon \norm{f}_{L^2(Q_t)}^2 + I[u_0] ,
\end{align}
where we fix $\varepsilon = \tfrac{1}{2(2D+1)}$. Using \eqref{eq:arbitrarymass2.12prime} and \eqref{eq:splitting_ulogu} as in the proof of Lemma \ref{lem:iscalar}, estimates \eqref{eq:arbitrarymass2.11}-\eqref{eq:arbitrarymass2.13} paired with Gr{\"o}nwall's lemma gives estimate \eqref{est:scalar1.2}, from which estimate \eqref{est:scalar2.2} follows.\\

Now we obtain an estimate on $\grad u$ in $L^2(0,T; L^1({\mathbb{R}^d}))$. To this end, we write $D\grad u = \sqrt{u} f - u \grad K * u$ so that we may deduce from H{\"o}lder's inequality and the conservation of mass that
\begin{align}\label{eq:arbitrarymass2.1}
   D  \norm{\grad u}_{L^2 (0,T; L^1 ({\mathbb{R}^d}))} &\leq \norm{ \sqrt{u} f}_{L^2 (0,T; L^1 ({\mathbb{R}^d}))} + \norm{u \grad K  * u}_{L^2 (0,T; L^1 ({\mathbb{R}^d}))} \nonumber \\
    &\leq \sqrt{m} \norm{f}_{L^2(Q_T)} + \norm{u \grad K  * u}_{L^2 (0,T; L^1 ({\mathbb{R}^d}))}.
\end{align}
We consider separately two sets: $\underline{{\Omega}} :=\{ (t,x) : u \leq \ell\}$ and $\overline{{\Omega}}:= \{ (t,x) : u > \ell\}$ for some $\ell > 1$ to be chosen sufficiently large. We then estimate as follows:
\begin{align}\label{eq:arbitrarymass2.2}
    &\norm{u \grad K * u}_{L^2 (0,T; L^1 ({\mathbb{R}^d}))} = \norm{\chi_{\underline{\Omega}} u \grad K * u}_{L^2 (0,T; L^1 ({\mathbb{R}^d}))} + \norm{ \chi_{\overline{\Omega}} u \grad K * u}_{L^2 (0,T; L^1 ({\mathbb{R}^d}))} \nonumber \\
    &\leq \ell \norm{\grad K * u}_{L^2 (0,T; L^1 ({\mathbb{R}^d}))} + \norm{\frac{u \as{\log u}}{\log \ell} \grad K * u}_{L^2 (0,T; L^1 ({\mathbb{R}^d}))} \nonumber \\
    &\leq C_K \ell \norm{u}_{L^2 (0,T; L^1 ({\mathbb{R}^d}))} + \frac{1}{\log \ell } \norm{u \as{\log u}}_{L^\infty (0,T; L^1 ({\mathbb{R}^d}))} \norm{\grad K *  u}_{L^2 (0,T; L^\infty ({\mathbb{R}^d}))} \nonumber \\
    &\leq C_K \ell \norm{u}_{L^2 (0,T; L^1 ({\mathbb{R}^d}))} + \frac{1}{\log \ell } \norm{u \as{\log u}}_{L^\infty (0,T; L^1 ({\mathbb{R}^d}))} \norm{K}_{L^\infty({\mathbb{R}^d})} \norm{ \grad u}_{L^2 (0,T; L^1 ({\mathbb{R}^d}))} \nonumber \\
    &\leq C_K \ell m \sqrt{T} + \frac{C_2}{\log \ell} \norm{\grad u}_{L^2 (0,T; L^1 ({\mathbb{R}^d}))} .
\end{align}
The first inequality uses the properties of the sets $\underline{\Omega}$, $\overline{\Omega}$; the second inequality uses hypothesis \textup{\textbf{\ref{h2}}} and H{\"o}lder's inequality; the third inequality uses Young's convolution inequality; the fourth inequality uses hypothesis \textup{\textbf{\ref{h1}}} and the previously obtained bounds on $u$ and $u \as{\log u}$ in $L^\infty (0,T; L^1({\mathbb{R}^d}))$. Combining \eqref{eq:arbitrarymass2.1}-\eqref{eq:arbitrarymass2.2} we have shown that
\begin{align}\label{eq:arbitrarymass2.223}
   D  \norm{\grad u}_{L^2 (0,T; L^1 ({\mathbb{R}^d}))} \leq \sqrt{m} \norm{f}_{L^2(Q_T)} + C_K \ell m \sqrt{T} + \frac{C_2}{\log \ell} \norm{\grad u}_{L^2 (0,T; L^1 ({\mathbb{R}^d}))},
\end{align}
whence choosing $\ell$ sufficiently large yields a uniform estimate on $\grad u$ in $L^2(0,T; L^1(\mathbb{R}^d))$. Consequently, H{\"o}lder's inequality, Young's convolution inequality and hypothesis \textup{\textbf{\ref{h1}}} yields
\begin{align}\label{eq:arbitrarymass2.3}
    \norm{\sqrt{u} \grad K * u}_{L^2(Q_T)} &\leq \norm{\sqrt{u}}_{L^\infty (0,T; L^2 ({\mathbb{R}^d}))}^{1/2} \norm{K}_{L^\infty({\mathbb{R}^d})} \norm{ \grad u}_{L^2 (0,T; L^1 ({\mathbb{R}^d}))} \nonumber \\
    &= \sqrt{m} \norm{K}_{L^\infty({\mathbb{R}^d})} \norm{ \grad u}_{L^2 (0,T; L^1 ({\mathbb{R}^d}))} 
\end{align}
and so \eqref{eq:arbitrarymass2.3} paired with the uniform estimate obtained from \eqref{eq:arbitrarymass2.223} yields estimate \eqref{est:scalar1.4}.

Finally, as $2 D \grad \sqrt{u} = f - \sqrt{u} \grad K * u$, estimate \eqref{eq:arbitrarymass2.3} and the bound on the dissipation $f$ in $L^2(Q_T)$ shows that $\grad \sqrt{u}$ is bounded in $L^2(Q_T)$ and estimate \eqref{est:scalar1.3} follows.
\end{proof}

\subsection{Improved estimates}

In this section we improve the estimates obtained in Lemmas \ref{lem:iscalar}-\ref{lem:iscalararbmass}. As the the following result uses only the estimates on $\grad \sqrt{u}$ and $\sqrt{u} \grad K * u$, the same argument applies assuming the conditions of either Lemma \ref{lem:iscalar} or Lemma \ref{lem:iscalararbmass} are met.

First, we obtain estimates to $u \in L^\infty (0,T; L^p ({\mathbb{R}^d}))$ for any $2 \leq p < \infty$ and on $\grad u$ in $L^2 (Q_T)$ with no additional assumptions required.

\begin{lemma}[Improved estimates with no further assumptions]\label{lem:iscalar2}
    Assume the conditions of Lemma \ref{lem:iscalar} (Lemma \ref{lem:iscalararbmass}) are satisfied. 
    Then, for any $p \in (1,\infty)$, if $u_0 \in L^p(\Rd)$ then
        \begin{equation}\label{est:scalar3}
        \sup_{t \in (0,T)}\norm{u(t,\cdot)}_{ L^p({\mathbb{R}^d})} \leq e^{C \norm{K}_{L^\infty (\mathbb{R}^d)}^2 p / D} \norm{u_0}_{L^p ({\mathbb{R}^d})}, 
        \end{equation}
        where $C$ is given in Lemma \ref{lem:iscalar} (Lemma  \ref{lem:iscalararbmass}). Moreover, if $u_0 \in L^2(\Rd)$ then
        \begin{equation}\label{est:scalar4}
        \norm{\grad u}_{L^2 (Q_T)} \leq   D^{-1}\, \tilde C \norm{u_0}_{L^2 ({\mathbb{R}^d})}, 
        \end{equation}
where $\tilde C = \tilde C( C )$.
\end{lemma}
\begin{proof}
 We consider the quantity $ \tfrac{1}{p}\norm{u(t,\cdot)}_{L^p({\mathbb{R}^d})}^p$ for $p \in (1,\infty)$. Taking the derivative with respect to time and integrating by parts yields
    \begin{align}\label{est:iscalar2.1}
        \frac{1}{p} \frac{{\rm d}}{{\rm d}t} \int_{\mathbb{R}^d} u^p \diff x&= - D (p-1) \int_{\mathbb{R}^d} \magg{\grad u} u^{p-2} \diff x-(p-1) \int_{\mathbb{R}^d} u^{p-1} \grad u \cdot K * \grad u \diff x\nonumber \\
        &= - \tfrac{4 D (p-1)}{p^2} \int_{\mathbb{R}^d} \magg{\grad u^{p/2}} \dx-(p-1) \int_{\mathbb{R}^d} u^{p-1} \grad u \cdot K * \grad u \diff x.
    \end{align}
We use the first term on the right hand side of \eqref{est:iscalar2.1} to control the second term using Cauchy's inequality with $\varepsilon$ and H{\"o}lder's inequality:
\begin{align}\label{est:iscalar2.2}
    (p-1) \int_{\mathbb{R}^d} u^{p-1} \grad u K * \grad u \diff x&= \tfrac{2(p-1)}{p} \int_{\mathbb{R}^d} \grad u^{p/2} u^{p/2} K * \grad u \diff x\nonumber \\
    &\leq \tfrac{\varepsilon (p-1)}{p} \int_{\mathbb{R}^d} \magg{\grad u^{p/2}} \diff x+ \tfrac{(p-1)}{\varepsilon p} \int_{\mathbb{R}^d} u^p \as{K * \grad u}^2 \diff x\nonumber \\
    &\leq \tfrac{\varepsilon (p-1)}{p} \int_{\mathbb{R}^d} \magg{\grad u^{p/2}} \diff x\nonumber \\
    &\quad + \tfrac{(p-1)}{\varepsilon p} \norm{(K * \grad u)(t,\cdot)}_{L^\infty ({\mathbb{R}^d})}^2 \int_{\mathbb{R}^d} u^p \diff x.
\end{align}
Choosing $\varepsilon = 2D / p$ we combine estimate \eqref{est:iscalar2.2} with \eqref{est:iscalar2.1} to find
\begin{align}\label{est:iscalar2.4}
    \frac{1}{p} \frac{{\rm d}}{{\rm d}t} \int_{\mathbb{R}^d} u^p \diff x+ \tfrac{2 D (p-1)}{p^2} \int_{\mathbb{R}^d}\magg{\grad u^{p/2}} \diff x&\leq \tfrac{(p-1)}{2D} \norm{(K*\grad u)(t,\cdot)}_{L^\infty ({\mathbb{R}^d})}^2 \int_{\mathbb{R}^d} u^p \diff x.
\end{align}
Gr{\"o}nwall's lemma implies that
\begin{align}\label{est:iscalar2.6}
    \norm{u(t,\cdot)}_{L^p({\mathbb{R}^d})} &\leq e^{C_3(T) p/D} \norm{u_0}_{L^p ({\mathbb{R}^d})},
\end{align}
where $C_3(T) = \norm{K * \grad u}_{L^2(0,T; L^\infty(\mathbb{R}^d))}^2 \leq C \norm{K}_{L^\infty (\mathbb{R}^d)}^2$ by estimate \eqref{est:scalar1.123_lembefore} (or \eqref{est:scalar1.123}). Taking the supremum over $t \in (0,T)$ yields estimate \eqref{est:scalar3}. Note carefully that this estimate does not allow one to take $p \to \infty$ as the coefficient on the right hand side depends critically on $p$.

Returning to estimate \eqref{est:iscalar2.4}, we fix $p=2$ and integrate both sides from $0$ to $T$ to obtain
\begin{align}
    \int_{\mathbb{R}^d} u^2 \diff x+ D \iint_{Q_T}  \magg{\grad u} \diff x{\rm d}t &\leq D^{-1} C \norm{K}_{L^\infty(\mathbb{R}^d)}^2 \sup_{t \in (0,T)} \norm{u(t,\cdot)}_{L^2 ({\mathbb{R}^d})}^2 + \norm{u_0}_{L^2 ({\mathbb{R}^d})}^2 .
\end{align}
Using estimate \eqref{est:iscalar2.6} with $p=2$ and rearranging the result yields estimate \eqref{est:scalar4}.
\end{proof}

\subsection{Existence of weak solutions (proof of Theorems \ref{thm:existsmallmassscalar}-\ref{thm:existarbmassscalar})}

We are now prepared to prove Theorems \ref{thm:existsmallmassscalar}-\ref{thm:existarbmassscalar}.
\begin{proof}[Proof of Theorems \ref{thm:existsmallmassscalar}-\ref{thm:existarbmassscalar}]
    We present the details for Theorem \ref{thm:existsmallmassscalar} only; the conclusion of Theorem \ref{thm:existarbmassscalar} follows in an identical fashion by using the estimates obtained in Lemma \ref{lem:iscalararbmass}, which are the same as those obtained in Lemma \ref{lem:iscalar} for the small mass regime.

    To start, we consider approximate solutions for any $\varepsilon>0$ given by
    \begin{align}\label{approximatescalar}
    \frac{\partial u^\varepsilon}{\partial t} = D \Delta u^\varepsilon + \grad \cdot (u^\varepsilon \grad K * {\omega}^\varepsilon * u^\varepsilon),
\end{align}
satisfying the initial data $u^\varepsilon (0,x) = u_0 (x)$ for each $\varepsilon>0$, where ${\omega}^\varepsilon$ denotes the standard mollifier. We seek to pass the limit as $\varepsilon \to 0^+$. First note that a unique, positive classical solution $u^\varepsilon(t,x)$ solving \eqref{approximatescalar} with $u(\cdot,x) = u_0(x)$ exists for each $\varepsilon>0$ by the methods used in the proof of, e.g., \cite[Theorem 2.2]{Carrillo2020LongTime}. Moreover, by Young's convolution inequality we have that
$$
\norm{K * {\omega}^\varepsilon}_{L^\infty({\mathbb{R}^d})} \leq \norm{K}_{L^\infty({\mathbb{R}^d})} \norm{{\omega}^\varepsilon}_{L^1({\mathbb{R}^d})} = \norm{K}_{L^\infty({\mathbb{R}^d})} ,
$$
and so if the small mass condition \eqref{const:c1} is satisfied, it is also satisfied for the regularized problem for all $\varepsilon\geq0$. Finally, for each $t>0$, the solution $u^{\varepsilon}$ is strictly positive by the strong maximum principle as in \eqref{eq:strong_maximum_principle}.

By Lemma \ref{lem:iscalar}, we have the following uniform bounds on the approximate solutions $u^\varepsilon$:
\begin{enumerate}
    \item[(A)]$\{ \grad \sqrt{u^\varepsilon} \}_{\varepsilon>0}$ is bounded in $L^2 (Q_T)$; 
    \item[(B)] $\{ \sqrt{u^\varepsilon} \grad K * {\omega}^\varepsilon * u^\varepsilon \}_{\varepsilon>0}$ is bounded in $L^2 (Q_T)$; 
    \item[(C)] $\{ \sup_{t\in(0,T)} \int_{\mathbb{R}^d} u^\varepsilon \magg{x} \diff x\}_{\varepsilon>0}$ is bounded.
\end{enumerate}

(A) implies that $2 \grad u^\varepsilon =\sqrt{u^\varepsilon} \grad \sqrt{ u^\varepsilon}$ belongs to $L^2(0,T; L^1({\mathbb{R}^d}))$ uniformly in $\varepsilon$. Moreover, $\{ \frac{\p u^\varepsilon}{\p t} \}_{\varepsilon>0}$ is bounded in $L^{\frac{4}{3}}(0,T; W^{-1,q}(\mathbb{R}^d) )$ for $q = \begin{cases} 
    \frac{2d}{2d-1} &\mbox{ for } d\neq 2,\\
    \frac{6}{5} &\mbox{ for } d= 2,
    \end{cases}$. Indeed, given any $\phi \in L^4(0,T; W^{1,q^\prime} (\mathbb{R}^d))$ H{\"o}lder's inequality gives
\begin{align*}
    \as{\int_0 ^T \int_{\mathbb{R}^d} \frac{\p u^\varepsilon}{\p t} \phi \diff x \diff t} &\leq D \norm{\grad u^\varepsilon}_{L^\frac{4}{3}(0,T; L^q(\mathbb{R}^d))} \norm{\grad \phi }_{L^4(0,T; L^{q^\prime} (\mathbb{R}^d))} \nonumber \\
    & + \norm{u^\varepsilon\, \grad K*\omega^\varepsilon *u^\varepsilon}_{L^\frac{4}{3}(0,T; L^q(\mathbb{R}^d))} \norm{\grad \phi}_{L^4(0,T; L^{q^\prime} (\mathbb{R}^d))} \nonumber \\
    &\leq C_5\,(1+D) \norm{\phi}_{L^{4} (0,T; W^{1, q^\prime} (\mathbb{R}^d))} ,
\end{align*}
where $C_5$ is a constant given by Lemma \ref{lem:exponent_q_weak_formulation_joint_space} which depends only on $\|K\ast \omega^{\varepsilon}\|_{L^{\infty}(\Rd)}$ and the norms of $u^{\varepsilon}$ and $\nabla \sqrt{u^{\varepsilon}}$ in $L^{\infty}(0,T; {L^1(\mathbb{R}^d)})$ and $L^{2}(Q_T)$, respectively.

Consequently, (A) paired with the conservation of mass implies that $u^\varepsilon \in L^2(0,T; W^{1,1} (\mathbb{R}^d))$) uniformly in $\varepsilon$. We then conclude by the Aubin-Lions lemma with
$$
W:= \left\{ v \in L^2 (0,T; W^{1,1}_{loc} (\mathbb{R}^d)) : \tfrac{\p v}{\p t} \in L^\frac{4}{3} (0,T; W^{-1,q}_{loc} (\mathbb{R}^d))\right\}
$$
that there exists a subsequence (still denoted by $\varepsilon$) such that
\begin{align}\label{conv:scalarstrong_L1}
u^\varepsilon \to u \quad \text{ strongly in }\quad L^1_{loc}(Q_T)
\end{align}
as $\varepsilon \to 0^+$. This implies that, perhaps for another subsequence,
\begin{align}\label{conv:scalarstrong}
    \sqrt{u^\varepsilon} \to \sqrt{u} \quad \text{ strongly in }\quad L^2_{loc}(Q_T)
\end{align}
as $\varepsilon \to 0^+$ by the Vitali convergence theorem. By (C) it is easy to see that given any $\delta>0$, for $R$ chosen sufficiently large there holds $\norm{u^\varepsilon}_{L^1(0,T;L^1(\mathbb{R}^d \setminus B_R))} < \delta$, independent of $\varepsilon$, and so $\{ u^\varepsilon \}_{\varepsilon>0}$ has compact closure in $L^1(Q_T)$. In particular, the strong convergence obtained in \eqref{conv:scalarstrong_L1}-\eqref{conv:scalarstrong} holds globally. 

Next, (B) implies that there exists $\xi \in L^2(Q_T)$ such that
\begin{align}\label{conv:scalarweak}
    \sqrt{u^\varepsilon} \grad K * {\omega}^\varepsilon * u^\varepsilon \to \xi \quad\text{ weakly in }\quad L^2(Q_T)
\end{align}
as $\varepsilon \to 0^+$. We claim that $\xi = \sqrt{u} \grad K*u$. To identify the limit, we may identify it in $L^1(0,T;L^2({\mathbb{R}^d}))$. Notice first that $\nabla K * {\omega}^\varepsilon *  u^\varepsilon$ is bounded in $L^2(0,T; L^\infty ({\mathbb{R}^d}))$ since $\grad u^\varepsilon \in L^2(0,T; L^1( \mathbb{R}^d))$, and so has a weak$^*$ limit in the same space. Hence, there holds
$$
\nabla K * {\omega}^\varepsilon * u^\varepsilon \to \nabla K *  u \quad\text{ weak*  in }\quad L^2(0,T; L^\infty({\mathbb{R}^d}))
$$
as $\varepsilon \to 0^+$. Furthermore, the limit can be identified as $\nabla K *  u = K * \nabla u$ because of the $L^\frac{4}{3}(0,T; L^q(\mathbb{R}^d))$ estimate on $\nabla u^{\varepsilon}$. Then, as $\sqrt{u^\varepsilon}$ converges strongly in $L^2(Q_T)$, given any $\phi \in L^\infty (0,T; L^2({\mathbb{R}^d}))$ there holds
\begin{align}
    &\as{\int_0 ^T \int_{\mathbb{R}^d} \left( \sqrt{u^\varepsilon} \grad K * {\omega}^\varepsilon * u^\varepsilon  - \sqrt{u} K * \grad u  \right) \phi \diff x \diff t} \leq \nonumber \\
    & \qquad \qquad \leq  \as{ \int_0 ^T \int_{\mathbb{R}^d} \left( [ \sqrt{u^\varepsilon} - \sqrt{u}\, ] \grad K * {\omega}^\varepsilon * u^\varepsilon ] \right) \phi \diff x \diff t} \nonumber \\
    &  \qquad \qquad \quad + \as{\int_0 ^T\int_{\mathbb{R}^d} \sqrt{u} \phi  \left( \grad K * {\omega}^\varepsilon * u^\varepsilon - K * \grad u  \right) \diff x \diff t} \nonumber \\
    & \qquad \qquad\leq \norm{\sqrt{u^\varepsilon} - \sqrt{u}}_{L^2(0,T;L^2(\mathbb{R}^d))} \norm{\grad K * {\omega}^\varepsilon * u^\varepsilon}_{L^2(0,T; L^\infty({\mathbb{R}^d}))} \norm{\phi }_{L^\infty (0,T; L^2({\mathbb{R}^d}))} \nonumber \\
    & \qquad \qquad \quad + \as{\int_0 ^T \int_{\mathbb{R}^d} \sqrt{u} \phi ( \grad K*{\omega}^\varepsilon * u^\varepsilon - K * \grad u ) \diff x \diff t} \to 0
\end{align}
as $\varepsilon\to0^+$. Notice that the second term converges to zero as $\varepsilon \to 0^+$ since $\sqrt{u} \phi \in L^2(0,T; L^1({\mathbb{R}^d}))$, the predual of $L^2(0,T; L^\infty({\mathbb{R}^d}))$. Hence, \eqref{conv:scalarstrong}-\eqref{conv:scalarweak} hold and $\xi = \sqrt{u} \grad (K*u)$ as asserted.

These convergence results allow us to conclude that in the limit as $\varepsilon \to 0^+$ there holds
$$
\int_0 ^T \left< \tfrac{\p u}{\p t}, \phi \right> \diff t + D \int_0 ^T \int_{\mathbb{R}^d} \grad u \cdot \grad \phi \, \diff x \diff t = - \int_0 ^T \int_\mathbb{R}^d u \grad (K*u) \cdot \grad \phi \, \diff x \diff t,
$$
for any $\phi$ smooth and compactly supported. By (A)-(B), the $L^\frac{4}{3}(0,T; L^q(\mathbb{R}^d))$-regularity of $\nabla u$ and $u\,\nabla (K\ast u)$, the $L^\frac{4}{3}(0,T; W^{-1,q}(\mathbb{R}^d))$-regularity of $\tfrac{\p u}{\p t}$ and a density argument, this weak formulation holds for any $\phi \in L^4(0,T; W^{1,q^\prime} (\mathbb{R}^d))$. Finally, by writing $u(t,\cdot) - u_0 = \int_0^t \tfrac{\p u}{\p s} \diff s$, we see that $u(t,\cdot)\to u_0(\cdot)$ in $W^{-1,q} (\mathbb{R}^d)$ and indeed the initial data is satisfied in the sense of $W^{-1,q} (\mathbb{R}^d)$. As $T>0$ was arbitrary, the solution obtained is in fact global.

To conclude, we recall that by Lemma \ref{lem:iscalar2} we have that $u \in L^\infty(0,T; L^p({\mathbb{R}^d}))$ for all $1 \leq p < \infty$ whenever $u_0\in L^p(\Rd)$ and $\grad u \in L^2(Q_T)$ if $u_0 \in L^2(\Rd)$.
\end{proof}

\subsection{Higher order estimates}

Under the additional hypothesis \textup{\textbf{\ref{h4}}} and further conditions on the initial data we can obtain better estimates on $\grad u$, $\Delta u$, and a uniform estimate on $u$ over $Q_T$.

\begin{lemma}[Higher order estimates]\label{lem:scalaresthigherorder}
Assume that $u_0 \in L^{\infty}(\Rd)$, $\grad u_0 \in L^2(\mathbb{R}^d)$ and $K*u_0 \in W^{2,p}(\mathbb{R}^d)$ for some $p>2$. Suppose hypothesis \textup{\textbf{\ref{h4}}} holds in addition to the conditions of Lemma \ref{lem:iscalar} (resp. Lemma \ref{lem:iscalararbmass}). Then, there exists a constant $\tilde C$  so that
\begin{align}
 \norm{\grad u }_{L^\infty (0,T; L^2 ({\mathbb{R}^d}))} &\leq \tilde C ; \label{est:scalar5} \\
    \norm{\Delta u}_{L^2 (Q_T)} &\leq \tilde C ;\label{est:scalar6} \\
         \norm{u }_{L^\infty (Q_T)} &\leq \tilde C, \label{est:scalarunif}
\end{align}
where $\tilde C = \tilde C(C, \norm{\grad (\tilde K*K)}_{L^2(\mathbb{R}^d)})$, and $C$ is given in Lemma \ref{lem:iscalar} (resp. Lemma \ref{lem:iscalararbmass}). 
 \end{lemma}
 \begin{proof}
     Before estimating $u$ and its derivatives, we first obtain improved estimates on $K * \grad u$ and $K* \Delta u$. To this end, recall that $\tilde K(\cdot) := K(-\cdot)$ and multiply \eqref{eq:mainscalar} by $\tilde K * K * \Delta u$, integrate over ${\mathbb{R}^d}$, and integrate by parts to find
     \begin{align}\label{est:iscalar3.1}
         \frac{1}{2} \frac{{\rm d}}{{\rm d}t} \int_{\mathbb{R}^d} \magg{\grad (K*u)} \diff x&= - D \int_{\mathbb{R}^d} \magg{\Delta (K*u)} \diff x + \int_{\mathbb{R}^d}  \tilde K * K * \Delta u \grad \cdot ( u \grad (K*u) )\nonumber \\
         &= - D \int_{\mathbb{R}^d} \magg{\Delta (K*u)} \diff x + \int_{\mathbb{R}^d} \Delta ( K * u) K * ( \grad u \cdot \grad ( K *  u )) \diff x\nonumber \\
         &\quad + \int_{\mathbb{R}^d} u\, \Delta ( K * u ) (\tilde K * \Delta ( K * u ) ) \diff x.
     \end{align}

     We first estimate the second term on the right hand side of \eqref{est:iscalar3.1} via Cauchy's inequality with  $\varepsilon = D/2$, followed by Young's convolution inequality and H{\"o}lder's inequality:
     \begin{align}\label{est:iscalar3.2}
          \int_{\mathbb{R}^d} \Delta ( K * u) &\, K * ( \grad u \cdot K * \grad u ) \diff x\nonumber \\
          &\leq \tfrac{D}{4} \norm{\Delta (K*u)}_{L^2 ({\mathbb{R}^d})}^2 + D^{-1} \norm{K*( \grad u \cdot \grad (K*u)) }_{L^2(\mathbb{R}^d)}^2 \nonumber \\
          &\leq \tfrac{D}{4}\norm{\Delta (K*u)}_{L^2 ({\mathbb{R}^d})}^2 + D^{-1}\norm{K}_{L^2 ({\mathbb{R}^d})} ^2 \norm{\grad u \cdot \grad (K * u)}_{L^1 ({\mathbb{R}^d})}^2 \nonumber \\
          &\leq \tfrac{D}{4}\norm{\Delta (K*u)}_{L^2 ({\mathbb{R}^d})}^2 + D^{-1} \norm{K}_{L^2 ({\mathbb{R}^d})}^2 \norm{\grad u(t,\cdot)}_{L^2 ({\mathbb{R}^d})}^2 \norm{\grad (K*u)}_{L^2 ({\mathbb{R}^d})}^2  \nonumber \\
          &=: \tfrac{D}{4}\norm{\Delta (K*u)}_{L^2 ({\mathbb{R}^d})}^2 + D^{-1} a^2(t) \norm{\grad (K*u)}_{L^2 ({\mathbb{R}^d})}^2 ,
     \end{align}
     where $a(t) := \norm{K}_{L^2 ({\mathbb{R}^d})} \norm{\grad u(t,\cdot)}_{L^2 ({\mathbb{R}^d})}$. Notice that $a^2(t) \in L^1 (0,T)$ for any fixed $T>0$ by Lemma \ref{lem:iscalar2}. 

     Next, we estimate the third term appearing in \eqref{est:iscalar3.1} via two applications of H{\"o}lder's inequality followed by Cauchy's inequality with $\varepsilon = D/2$:
     \begin{align}
         \int_{\mathbb{R}^d} u\, \Delta ( K * u ) (\tilde K * \Delta ( K * u ) ) \diff x&\leq \norm{u(t,\cdot)}_{L^2 ({\mathbb{R}^d})} \norm{\Delta (K*u) (\tilde K* \Delta (K*u) )}_{L^2 ({\mathbb{R}^d})} \nonumber \\
         &\leq \norm{u(t,\cdot)}_{L^2 ({\mathbb{R}^d})} \norm{\tilde K* (\Delta (K*u) )}_{L^{\infty} ({\mathbb{R}^d})} \norm{\Delta (K*u)}_{L^{2} ({\mathbb{R}^d})} \nonumber \\
         &\leq \frac{D}{4} \norm{\Delta (K*u)}_{L^{2} ({\mathbb{R}^d})}^2 \nonumber \\
         &\quad +  D^{-1} \norm{u(t,\cdot)}_{L^2 ({\mathbb{R}^d})}^2 \norm{\tilde K* (\Delta (K*u) )}_{L^{\infty} ({\mathbb{R}^d})} ^2,
     \end{align}

     We must ensure that the second product belongs to $L^1 (0,T)$. By Lemma \ref{lem:iscalar2}, we know that $u \in L^\infty (0,T; L^2 ({\mathbb{R}^d}))$. For the remaining term, we use the fact that
     \begin{align*}
         \tilde K* \Delta (K*u) = \grad (\tilde K * K) * \grad u .
     \end{align*}
     Hence, we may estimate via Young's convolution inequality
     \begin{align}
          \norm{\tilde K* (\Delta (K*u) )}_{L^{\infty} ({\mathbb{R}^d})} ^2 &\leq \norm{\grad (\tilde K*K)}_{L^2 ({\mathbb{R}^d})}^2 \norm{\grad u(t,\cdot)}_{L^2 ({\mathbb{R}^d})}^2, 
     \end{align}
     from which we conclude that
     \begin{align}\label{est:iscalar3.3}
         \int_{\mathbb{R}^d} u\, \Delta ( K * u ) (\tilde K * \Delta ( K * u ) ) \diff x&\leq \frac{D}{4} \norm{\Delta (K*u)}_{L^2 ({\mathbb{R}^d})}^2 \nonumber \\
         &+ D^{-1} b^2(t)
     \end{align}
     where
     \begin{align*}
         b(t) := \norm{\grad (\tilde K*K)}_{L^2 ({\mathbb{R}^d})} \sup_{t \in (0,T)}\norm{u(t,\cdot)}_{L^2 ({\mathbb{R}^d})} \norm{\grad u(t,\cdot)}_{L^2 ({\mathbb{R}^d})},
     \end{align*}
     and $b^2(t) \in L^1(0,T)$ by hypothesis \textup{\textbf{\ref{h4}}} and by estimate \eqref{est:scalar4} of Lemma \ref{lem:iscalar2}.
     
    Combining estimates \eqref{est:iscalar3.2} and \eqref{est:iscalar3.3} with \eqref{est:iscalar3.1}, we finally obtain
     \begin{align}\label{est:iscalar3.4}
         \frac{{\rm d}}{{\rm d} t} \int_{\mathbb{R}^d} \magg{\grad (K*u)} \diff x+ D \int_{\mathbb{R}^d} \magg{\Delta (K*u)} \diff x&\leq D^{-1} \left( a^2 (t) \norm{\grad (K*u))}_{L^2 ({\mathbb{R}^d})}^2  + b^2 (t)\right).
     \end{align}
     By Gronwall's lemma we conclude that
     \begin{align}\label{est:0.001}
         \sup_{t \in (0,T)} \norm{\grad (K*u)(t,\cdot)}_{L^2 ({\mathbb{R}^d})} ^2 &\leq D^{-1} A(T) \left( \norm{\grad (K * u_0)}_{L^2 ({\mathbb{R}^d})}^2 + \int_0 ^T b^2 (t) {\rm d}t \right),
     \end{align}
     where $A(t) := e^{\int_0 ^t a^2 (s) {\rm d}s}$. Hence, $\grad (K*u) \in L^\infty(0,T; L^2(\Omega))$. 

     Returning to estimate \eqref{est:iscalar3.4}, we extract the second term on the left hand side and use estimate \eqref{est:0.001} to conclude that in fact
     \begin{align}
         \norm{\Delta (K*u)}_{L^2 (Q_T)} &\leq 2 D^{-1} \left( \sup_{t \in (0,T)} \norm{\grad (K*u) (t,\cdot)}_{L^2 ({\mathbb{R}^d})} \norm{a(\cdot)}_{L^2 (0,T)} + \norm{b(\cdot)}_{L^2 (0,T)}  \right),
     \end{align}
     and so $\Delta (K*u) \in L^2(Q_T)$.

     We now improve the estimate on $\Delta (K * u)$ from $L^2 (Q_T)$ to $L^2 (0,T; L^p ({\mathbb{R}^d}))$ for some $p>2$. To this end, notice first that if $u$ solves the heat equation
     $$
\frac{\p u}{\p t} - D \Delta u = g \in L^2 (0,T; L^1 ({\mathbb{R}^d})) ,
     $$
     for some function $g(t,x)$, then by the linearity of the convolution we have also that $K * u$ solves
     $$
\frac{\p}{\p t} ( K*u) - D \Delta (K*u) = K * g.
     $$
     In our case, using the improved estimates on $K * \grad u$, $K* \Delta u$, we have
     $$
     g = \grad u\, K* \grad u + u \, K* \Delta u \, \in L^2 (0,T; L^1 ({\mathbb{R}^d})),
     $$
     so that $K*g \in L^2 (0,T; L^q ({\mathbb{R}^d}))$ for any $q \in [1,\infty]$ since $K \in L^1 ({\mathbb{R}^d}) \cap L^\infty ({\mathbb{R}^d})$. Hence, by the maximal regularity of the heat equation (see, e.g., \cite[Theorem 3.1]{HieberPruss1997}), we find that $\Delta (K*u) \in L^2 (0,T; L^p ({\mathbb{R}^d}))$ where $p$ is an exponent such that $K * u_0 \in W^{2,p}(\Rd)$.

     We are now ready to obtain estimates \eqref{est:scalar5}-\eqref{est:scalar6}. Multiplying \eqref{eq:mainscalar} by $\Delta u$ and integrating over ${\mathbb{R}^d}$ yields
     \begin{align}\label{est:scalar8}
         \frac{{\rm d}}{{\rm d}t} \int_{\mathbb{R}^d} \magg{\grad u} \diff x+ D \int_{\mathbb{R}^d} \magg{\Delta u} \diff x\leq \int_{\mathbb{R}^d} \as{u\, \Delta u (K*\Delta u)} \diff x+ \int_{\mathbb{R}^d} \as{\grad u\, \Delta u (K * \grad u)} \diff x.
     \end{align}
We estimate the right hand side via Cauchy's inequality. For the first term we have
     \begin{align}
         \int_{\mathbb{R}^d} \as{u\, \Delta u (K*\Delta u)} \diff x&\leq \frac{D}{4} \int_{\mathbb{R}^d} \magg{\Delta u} \diff x+ D^{-1} \int_{\mathbb{R}^d} \magg{u} \magg{K * \Delta u} \diff x\nonumber \\
         &\leq \frac{D}{4} \int_{\mathbb{R}^d} \magg{\Delta u} \diff x+ D^{-1} \norm{u(t,\cdot)}_{L^{q} ({\mathbb{R}^d})}^2 \norm{K * \Delta u (t,\cdot)}_{L^p ({\mathbb{R}^d})}^2 \nonumber \\
         &\leq \frac{D}{4} \int_{\mathbb{R}^d} \magg{\Delta u} \diff x+ D^{-1} \norm{u}_{L^\infty (0,T; L^q ({\mathbb{R}^d}))}^2 \norm{K * \Delta u (t,\cdot)}_{L^p ({\mathbb{R}^d})}^2 \nonumber \\
         &\leq \frac{D}{4} \int_{\mathbb{R}^d} \magg{\Delta u} \diff x+ D^{-1} e^{8 C(T) / D} \norm{u_0}_{L^q ({\mathbb{R}^d})}^2 \norm{K * \Delta u (t,\cdot)}_{L^p ({\mathbb{R}^d})}^2 ,
     \end{align}
     where we have used H{\"o}lder's inequality and that $u \in L^\infty (0,T; L^q ({\mathbb{R}^d}))$ for any exponent $q := 2p/(p-2) \in (1,\infty)$. From our improved estimates on $\Delta (K*u)$ we have already shown that $\norm{K * \Delta u (t,\cdot)}_{L^p ({\mathbb{R}^d})}^2  \in L^1 (0,T)$ for some $p>2$. 

     For the second term, we again use Cauchy's inequality
     \begin{align*}
         \int_{\mathbb{R}^d} \as{\grad u \Delta u (K * \grad u)} \diff x&\leq \frac{D}{4} \int_{\mathbb{R}^d} \magg{\Delta u} \diff x+ D^{-1} \int_{\mathbb{R}^d} \magg{K * \grad u} \magg{\grad u} \diff x\nonumber \\
         &\leq \frac{D}{4} \int_{\mathbb{R}^d} \magg{\Delta u} \diff x+ D^{-1} \norm{K * \grad u (t,\cdot)}_{L^\infty ({\mathbb{R}^d})}^2 \int_{\mathbb{R}^d} \magg{\grad u} \dx,
     \end{align*}
     where $\norm{K * \grad u (t,\cdot)}_{L^\infty ({\mathbb{R}^d})} \in L^1 (0,T)$ since $\grad u \in L^2 (Q_T)$ and $K \in L^1(\mathbb{R}^d) \cap L^\infty ({\mathbb{R}^d})$. Hence, \eqref{est:scalar8} becomes
     \begin{align}
                  \frac{{\rm d}}{{\rm d}t} \int_{\mathbb{R}^d} \magg{\grad u} \diff x+ \frac{D}{2} \int_{\mathbb{R}^d} \magg{\Delta u} \diff x\leq C_1 (t) + C_2 (t) \int_{\mathbb{R}^d} \magg{\grad u} \diff x,
     \end{align}
     for functions $C_1(t)$, $C_2(t) \in L^1 (0,T)$. Gr{\" o}nwall's lemma implies $\Delta u \in L^2 (0,T; L^2 ({\mathbb{R}^d}))$ and $\grad u \in L^\infty (0,T; L^2 ({\mathbb{R}^d}))$ from which estimates \eqref{est:scalar5}-\eqref{est:scalar6} follow.

     We now obtain a uniform bound on the solution $u(t,x)$ via an Alikakos type argument \cite{Alikakos1979}, which arises as a consequence of $\grad u \in L^\infty(0,T;L^2 ({\mathbb{R}^d})) \Rightarrow K * \grad u \in L^\infty (Q_T)$. Set $C_0 := \norm{K * \grad u}_{L^\infty (Q_T)}$. To begin, we return to the right hand side of estimate \eqref{est:iscalar2.1} and apply Cauchy's inequality
     \begin{align*}
         -(p-1) \int_{\mathbb{R}^d} u^{p-1} \grad u K * \grad u \diff x&= - \tfrac{2(p-1)}{p} \int_{\mathbb{R}^d} u^{p/2} \grad ( u^{p/2} ) K * \grad u \diff x\\
         &\leq \tfrac{2(p-1)}{p} C_0 \int_{\mathbb{R}^d} u^{p/2} \as{\grad u^{p/2}} \diff x\\
         &\leq C_0 \int_{\mathbb{R}^d} \left( \varepsilon \magg{\grad u^{p/2}} + \frac{1}{\varepsilon} u^p \right) \dx,
     \end{align*}
     where $\varepsilon>0$ is to be chosen. We then recall the Nash's inequality: there exists $C_N>0$ such that for $g \in L^1(\mathbb{R}^d)\cap W^{1,2}(\mathbb{R}^d)$ there holds
     $$
     \|g\|^2_{L^2(\Rd)} \leq C_N \, \|g\|^{4/(d+2)}_{L^1(\Rd)}\, \|\nabla g\|^{2d/(d+2)}_{L^2(\Rd)} \leq \delta \, C_N^{(d+2)/d} \,  \|\nabla g\|^{2}_{L^2(\Rd)} + \delta^{-d/2}\|g\|_{L^1(\Rd)}^2,  
     $$
     where the second step is a consequence of the Young's inequality with conjugate exponents $\frac{d+2}{2}$ and $\frac{d+2}{d}$, and $\delta>0$ is to be chosen. We apply it with $g = u^{p/2}$ and we combine it with estimate \eqref{est:iscalar2.1} to obtain
         \begin{align}\label{est:iscalar2.1.22}
        \frac{1}{p} \frac{{\rm d}}{{\rm d}t} \int_{\mathbb{R}^d} u^p \dx
        \leq &- \tfrac{4 D (p-1)}{p^2} \|\nabla u^{p/2}\|^{2}_{L^2(\Rd)}+ \varepsilon \, C_0 \|\nabla u^{p/2}\|^{2}_{L^2(\Rd)} \nonumber \\
        &+ \frac{1}{\varepsilon} C_0 \left(\delta \, C_N^{(d+2)/d} \,  \|\nabla u^{p/2}\|^{2}_{L^2(\Rd)} + \delta^{-d/2}\|u^{p/2}\|_{L^1(\Rd)}^2  \right) .
    \end{align}
    We now set $\delta \, C_N^{(d+2)/d} = \varepsilon^2$ and choose $\varepsilon = \tfrac{2 D (p-1)}{C_0 p^2}$ so that
     $$
 2 C_0 \varepsilon = \frac{4 D (p-1)}{p^2} \Rightarrow \varepsilon = \frac{2 D (p-1)}{C_0 p^2}.
    $$
    It follows that the gradient term vanishes and \eqref{est:iscalar2.1.22} becomes
    \begin{align}
        \frac{1}{p} \frac{{\rm d}}{{\rm d}t} \int_{\mathbb{R}^d} u^p \diff x&\leq C_6\,  \varepsilon^{-d-1} \norm{u^{p/2} (t,\cdot)}_{L^1 ({\mathbb{R}^d})}^2 \leq  C_7\, p^{d+1} \norm{u^{p/2} (t,\cdot)}_{L^1 ({\mathbb{R}^d})}^2,
    \end{align}
    where $ C_6$, $C_7$ depend on $C_N$, $C_0$, $D$ but do not depend on $p$. Integrating both sides from $0 \to t$ yields
    \begin{align}
        \norm{u(t,\cdot)}_{L^p ({\mathbb{R}^d})} ^p \leq \norm{u_0}_{L^p ({\mathbb{R}^d})}^p +  C_7\, p^{d+2} T \norm{u}_{L^\infty (0,T; L^{p/2} ({\mathbb{R}^d})}^p .
    \end{align}
    Without loss of generality, for $T$ fixed we may assume that 
    $$
    \norm{u_0}_{L^p ({\mathbb{R}^d})}^p \leq C_7\, p^{d+2}\, T \norm{u}_{L^\infty (0,T; L^{p/2} ({\mathbb{R}^d}))}^p,
    $$
    for all $p \geq p^*$ for some $p^* \geq 1$. Otherwise, there exists a sequence $\{ p_k\}_{k \geq 0}$ such that $p_k \to +\infty$ as $k \to +\infty$ so that
    $$
\norm{u_0}_{L^{p_k} ({\mathbb{R}^d})}^{p_k} \geq  C_7\, {p_k}^{d+2}\, T \norm{u}_{L^\infty (0,T; L^{p_k /2} ({\mathbb{R}^d}))}^{p_k}
    $$
    from which the boundedness of $u$ follows. Set $C_8 = (2 C_7 T)^{1/(d+2)}$ and  $M(p) := \norm{u}_{L^\infty (0,T; L^p ({\mathbb{R}^d}))}.$
    We then have that
    $$
        M(p) \leq ( C_8\, p )^{
        (d+2)/p} M(p/2) , \quad \forall p \geq p^*.
    $$
    Take $i \geq 1$ so that $2^i \geq p^*$. For any fixed $k$ we proceed iteratively to obtain
    \begin{align}
        M( 2^k) \, &\leq \, ( C_8 2^{k} )^{(d+2)/{2^k}} M(2^{k-1}) \, \leq \, ( C_8 2^{k} )^{(d+2)/{2^k}}  ( C_8 2^{(k-1)} )^{(d+2)/{2^{k-1}}} M(2^{k-2}) \nonumber \\
        &\ldots \leq M(2^{i-1}) \prod_{j=i}^k ( C_8 2^{j})^{(d+2)/2^{j}} .
    \end{align}
    It follows that $M(2^k)$ is bounded independent of $k$ so long as
    $$
\lim_{k \to \infty} \prod_{j=i}^k ( C_8 2^{j})^{(d+2)/2^{j}} < \infty \iff \sum_{j=i}^\infty \frac{(d+2) \log (C_8 2^{j})}{2^j} < \infty ,
    $$
which readily follows from the fact that $2^j \gg (d+2)\,j$. Hence, $u$ is uniformly bounded in $Q_T$ for any $T>0$ fixed and estimate \eqref{est:scalarunif} follows.
 \end{proof}

 \begin{lemma}[Improved integrability of gradient]\label{lem:scalarLpgradient}
     Suppose $\grad u_0 \in L^p(\mathbb{R}^d)$ for some $p \in [2,\infty)$. Under the conditions of Lemma \ref{lem:scalaresthigherorder} there exists a constant $\tilde C_0$ such that there holds
     \begin{align}\label{est:graduLp}
         \norm{\grad u}_{L^\infty (0,T; L^p(\mathbb{R}^d))} \leq \tilde C_0,
     \end{align}
     where $\tilde C_0 = \tilde C_0 (\tilde C, p, \norm{\grad u_0}_{L^p (\mathbb{R}^d)})$, and $\tilde C$ is given in Lemma \ref{lem:scalaresthigherorder}.
 \end{lemma}
 \begin{proof}
     Under the conditions of Lemma \ref{lem:scalaresthigherorder}, we have that
     $$
u \in L^\infty (Q_T);\quad \grad u \in L^\infty(0,T; L^2(\Omega)); \quad \Delta u \in L^2(Q_T).
     $$
     Assume that $u$ is a smooth solution to problem \eqref{eq:mainscalar}. For any $p \in [2,\infty)$, we multiply the equation for $u$ by
    $$
    -\mbox{div}\left(\as{\grad u}^{p-2} \nabla u \right) = -(p-1)\, \Delta u\, |\nabla u|^{p-2}
    $$
    and integrate over $\mathbb{R}^d$ to obtain
     \begin{align}\label{eq:lemLpgrad1.1}
     \int_{\mathbb{R}^d} \as{\grad u}^{p-2}  \nabla u \, \nabla u_t \diff x =  -(p-1) \int_{\mathbb{R}^d} \as{\grad u}^{p-2}\, \Delta u\,  ( D \Delta u + \grad u K * \grad u + u K* \Delta u) \diff x .
     \end{align}
Clearly, the left hand side of \eqref{eq:lemLpgrad1.1} equals $\frac{1}{p} \int_{\mathbb{R}^d} \as{\grad u}^p \diff x$ and so rearrangement of \eqref{eq:lemLpgrad1.1} yields
\begin{align}
    \frac{1}{p\,(p-1)} \frac{{\rm d}}{{\rm d}t} \int_{\mathbb{R}^d} \as{\grad u}^p \dx \leq& - D \int_{\mathbb{R}^d} \as{\grad u}^{p-2} \as{\Delta u}^2 \dx + \int_{\mathbb{R}^d} \as{\grad u}^{p-1} \as{\Delta u} \as{K* \grad u} \dx \nonumber \\
    &+ \int_{\mathbb{R}^d} \as{\grad u}^{p-2} u \as{\Delta u} \as{K* \Delta u} \nonumber \\
    =:& - D \int_{\mathbb{R}^d} \as{\grad u}^{p-2} \as{\Delta u}^2 \dx + I_1 + I_2 .
\end{align}
We estimate $I_1$, $I_2$ separately. For $I_1$, Cauchy's inequality yields
\begin{align}\label{eq:lemLpgrad1.3}
    I_1 \leq &\, \frac{D}{2} \int_{\mathbb{R}^d} \as{\grad u}^{p-2} \as{\Delta u}^2 \dx + \frac{1}{2D} \int_{\mathbb{R}^d} \as{\grad u}^p \as{K*\grad u}^2 \dx \nonumber \\
    \leq& \frac{D}{2} \int_{\mathbb{R}^d} \as{\grad u}^{p-2} \as{\Delta u}^2 \dx + C_1 \int_{\mathbb{R}^d} \as{\grad u}^p \dx ,
\end{align}
where
$$
C_1 := \frac{\norm{K}_{L^2(\mathbb{R^d})}^2 \norm{\grad u(t,\cdot)}_{L^2(\mathbb{R}^d)}^2}{2D} < \infty,
$$
since $\grad u \in L^\infty(0,T; L^2(\mathbb{R}^d))$. For $I_2$, we first estimate via Cauchy's inequality as

     \begin{align}\label{eq:lemLpgrad1.5}
         I_2  \leq \frac{D}{2} \int_{\mathbb{R}^d} \as{\grad u}^{p-2} \as{\Delta u}^2 \dx + C_2 \int_{\mathbb{R}^d} \as{\grad u}^{p-2} \dx ,
     \end{align}
     where
     $$
C_2(t) := \frac{\norm{u}_{L^\infty(Q_T)}^2 \norm{K}_{L^2(\mathbb{R}^d)}^2 \norm{\Delta u(t,\cdot)}_{L^2(\mathbb{R}^d)}^2}{2D } \in L^1(0,T),
     $$
     since $u \in L^\infty(Q_T)$ and $\Delta u \in L^2(Q_T)$. 

     To estimate $\int \as{\grad u}^{p-2} \dx$, we first suppose that $p > 4$ (the case $p=4$ is direct) and apply Young's inequality for products with $p^\prime = (p-2)/(p-4)$, $q^\prime = (p-2)/2$:
     \begin{align}\label{eq:lemLpgrad1.7}
         \int_{\mathbb{R}^d} \as{\grad u}^{p-2} \dx =& \int_{\mathbb{R}^d} \as{\grad u}^{p(p-4)/(p-2)} \as{\grad u}^{4/(p-2)} \dx \nonumber \\
         \leq& \int_{\mathbb{R}^d} ( \as{\grad u}^{p} + \as{\grad u}^2 ) \dx .
     \end{align}
     Combining estimates \eqref{eq:lemLpgrad1.3}-\eqref{eq:lemLpgrad1.7} with estimate \eqref{eq:lemLpgrad1.1} leaves
     \begin{align}
         \frac{1}{p\,(p-1)} \frac{{\rm d}}{{\rm d}t} \int_{\mathbb{R}^d} \as{\grad u}^p \dx \leq& \left( C_1 + C_2(t) \right) \int_{\mathbb{R}^d} \as{\grad u}^p \dx + C_2(t) \norm{\grad u(t,\cdot)}_{L^2(\mathbb{R}^d)}^2.
     \end{align}
     Since $C_2(t) \in L^1(0,T)$, and since 
     $$
     \grad u \in L^\infty(0,T; L^2(\mathbb{R}^d)) \Rightarrow C_2(t) \norm{\grad u(t,\cdot)}_{L^2(\mathbb{R}^d)}^2 \in L^1(0,T),
     $$
     we obtain the estimate \eqref{est:graduLp} via Gr{\"o}nwall's lemma. The case when $p \in (2,4)$ is obtained via interpolation between the cases $p=2$ and $p=4$.
 \end{proof}

 \subsection{Uniqueness}

The higher regularity results from the previous subsection allow us to show the following uniqueness result.
 \begin{theorem}\label{thm:scalarunique}
     Suppose hypothesis \textup{\textbf{\ref{h4}}} holds. Then any nonnegative solution $u$ solving problem \eqref{eq:mainscalar} belonging to the class
     \begin{align*}
         u &\in L^\infty (Q_T) \cap L^\infty (0,T; L^1 ({\mathbb{R}^d})); \\
         \grad u &\in L^\infty (0,T; L^2 ({\mathbb{R}^d}));\\
         \Delta u &\in L^2 (Q_T)
     \end{align*}
     is unique.
 \end{theorem}
 \begin{proof}
\normalfont     Assume there are two nonnegative solutions $u$ and $v$. We write
\begin{align}\label{eq:unique1.0}
(u-v)_t - D \Delta (u-v) = \grad \cdot ( (u-v) \grad K*u) + \grad \cdot ( v \grad K*(u-v) ) .
\end{align}
     Set $w = u-v$. Testing \eqref{eq:unique1.0} against $w$, integrating over ${\mathbb{R}^d}$ and integration by parts gives
     \begin{align}\label{eq:unique1}
         \frac{1}{2} \frac{{\rm d}}{{\rm d}t} \int_{\mathbb{R}^d} w^2 \diff x+ D \int_{\mathbb{R}^d} \magg{\grad w} \diff x&= - \int_{\mathbb{R}^d} w \grad (K*u) \grad w \diff x\nonumber \\
         &\quad - \int_{\mathbb{R}^d} v \grad (K*w) \grad w \diff x\nonumber \\
         = - \frac{1}{2} \int_{\mathbb{R}^d} \grad w^2 \grad (K*u) \diff x& - \int_{\mathbb{R}^d} v \grad (K*w) \grad w \diff x.
     \end{align}
     We estimate each term separately. For the first term, we integrate by parts and use that $\Delta u \in L^2 (0,T; L^2 ({\mathbb{R}^d}) ) \Rightarrow K* \Delta u \in L^2 (0,T; L^\infty({\mathbb{R}^d}))$ to obtain
     \begin{align*}
         - \frac{1}{2} \int_{\mathbb{R}^d} \grad w^2 \grad (K*u) \diff x= \frac{1}{2} \int_{\mathbb{R}^d} w^2 K * \Delta u \diff x\leq \frac{1}{2} \norm{K * \Delta u (t,\cdot)}_{L^\infty ({\mathbb{R}^d})} \int_{\mathbb{R}^d} w^2 \diff x.
     \end{align*}
     For the second term, we apply Cauchy's inequality with $\varepsilon = D \norm{v}_{L^\infty (Q_T)}^{-1}$ to obtain
     \begin{align*}
         - \int_{\mathbb{R}^d} v \grad (K*w) \grad w \diff x&\leq \frac{1}{2} \int_{\mathbb{R}^d} \left( D \magg{\grad w} + \frac{\norm{v}_{L^\infty (Q_T)}^2}{D}\magg{K* \grad w} \right) \diff x.
     \end{align*}
   Combining these estimates, \eqref{eq:unique1} becomes
     \begin{align}\label{est:uniqueness1}
         \frac{1}{2} \frac{{\rm d}}{{\rm d}t} \int_{\mathbb{R}^d} w^2 \diff x+ \frac{D}{2} \int_{\mathbb{R}^d} \magg{\grad w} \diff x&\leq C_1 (t) \int_{\mathbb{R}^d} w^2 \diff x+ C_2 \int_{\mathbb{R}^d} \magg{K*\grad w} \diff x,
     \end{align}
     where $C_1(t) = \norm{K * \Delta u (t,\cdot)}_{L^\infty ({\mathbb{R}^d})} \in L^1 (0,T)$ and $C_2 = D^{-1}\norm{v}_{L^\infty (Q_T)}^2 $.

     Next, we test \eqref{eq:unique1.0} against $\tilde K* K* w$, integrate over ${\mathbb{R}^d}$ and integrate by parts to find
     \begin{align}\label{eq:unique1.2}
         \frac{1}{2} \frac{{\rm d}}{{\rm d}t} \int_{\mathbb{R}^d} \magg{K*w} \diff x+ D \int_{\mathbb{R}^d} \magg{\grad K*w} \diff x&= - \int_{\mathbb{R}^d} w \grad (K*u) \grad ( \tilde K * K * w ) \diff x\nonumber \\
         &\quad - \int_{\mathbb{R}^d} v \grad (K*w) \grad (\tilde K*K*w) \diff x.
     \end{align}
     Since $\grad u \in L^\infty (0,T; L^2 ({\mathbb{R}^d}))$, there holds $\grad (K*u) \in L^\infty (0,T; L^2({\mathbb{R}^d}))$ so long as $K \in L^1 ({\mathbb{R}^d})$. Moreover,
     $$
\norm{\grad (\tilde K*K*w) (t,\cdot)}_{L^\infty ({\mathbb{R}^d})} \leq \norm{\grad (\tilde K*K)}_{L^2 ({\mathbb{R}^d})} \norm{w(t,\cdot)}_{L^2 ({\mathbb{R}^d})} = C_3 \norm{w(t,\cdot)}_{L^2 ({\mathbb{R}^d})},
     $$
     where $C_3 = \norm{\grad (\tilde K*K)}_{L^2 ({\mathbb{R}^d})} < \infty$ under hypothesis \textup{\textbf{\ref{h4}}}. Thus, the first term of \eqref{eq:unique1.2} can be estimated as
     \begin{align}
         &\as{\int_{\mathbb{R}^d} w \grad (K*u) \grad ( \tilde K * K * w ) \dx} \nonumber  \\
         &\quad\leq \norm{w(t,\cdot)}_{L^2 ({\mathbb{R}^d})} \norm{\grad (K*u)(t,\cdot)}_{L^2 ({\mathbb{R}^d})} \norm{\grad (\tilde K*K*w)(t,\cdot)}_{L^\infty({\mathbb{R}^d})}  \nonumber \\
         &\quad \leq C_3 \norm{\grad (K*u)}_{L^\infty(0,T; L^2 ({\mathbb{R}^d}) )} \norm{w(t,\cdot)}_{L^2 ({\mathbb{R}^d})}^2 .
     \end{align}
     For the second term of \eqref{eq:unique1.2}, we use that $v \in L^\infty (Q_T ) \cap L^\infty (0,T; L^1({\mathbb{R}^d})) \Rightarrow v \in L^\infty (0,T; L^2({\mathbb{R}^d}))$ and apply Cauchy's inequality:
     \begin{align}
         &\as{\int_{\mathbb{R}^d} v \grad (K*w) \grad (\tilde K*K*w) \dx} \nonumber \\
         &\quad\leq \norm{v (t,\cdot)}_{L^2 ({\mathbb{R}^d})} \norm{\grad K * w(t, \cdot)}_{L^2 ({\mathbb{R}^d})} \norm{\grad \tilde K*K*w(t, \cdot)}_{L^\infty({\mathbb{R}^d})} \nonumber \\
         &\quad\leq C_3 \norm{v}_{L^\infty (0,T;L^2({\mathbb{R}^d}))}\left( \varepsilon \norm{\grad K*w(t, \cdot)}_{L^2 ({\mathbb{R}^d})}^2 + \frac{1}{\varepsilon} \norm{w(t, \cdot)}_{L^2 ({\mathbb{R}^d})}^2  \right).
     \end{align}
     Choosing $\varepsilon = \frac{D}{2 C_3 \norm{v}_{L^\infty (0,T; L^2({\mathbb{R}^d}))} }$, we conclude that
     \begin{align}\label{est:uniqueness2}
         \frac{1}{2} \frac{{\rm d}}{{\rm d}t} \int_{\mathbb{R}^d} \magg{K*w} \diff x+ \frac{D}{2} \int_{\mathbb{R}^d} \magg{\grad K*w} \diff x&\leq C_4 \norm{w(t, \cdot)}_{L^2 ({\mathbb{R}^d})}^2,
     \end{align}
     where $C_4 < \infty$. 

     We now combine estimates \eqref{est:uniqueness1} and \eqref{est:uniqueness2}. To this end, we multiply estimate \eqref{est:uniqueness1} by $\frac{D}{4 C_2}$ and add the result to \eqref{est:uniqueness2}:
     \begin{align}
                   &\frac{1}{2} \frac{{\rm d}}{{\rm d}t} \left( \frac{D}{4 C_3} \int_{\mathbb{R}^d} w^2 \diff x+ \int_{\mathbb{R}^d} \magg{K*w} \diff x\right) + \frac{D^2}{8 C_3} \int_{\mathbb{R}^d} \magg{\grad w} \diff x+ \frac{D}{4} \int_{\mathbb{R}^d} \magg{\grad K*w} \diff x\nonumber \\
                   &\quad \leq ( C_1 (t) + C_4 ) \int_{\mathbb{R}^d} w^2 \diff x\leq ( C_1 (t) + C_4 ) \int_{\mathbb{R}^d} \left( w^2 + \magg{K * w} \right) \dx.
     \end{align}
     Gr{\"o}nwall's lemma implies that $w \equiv 0$ a.e. in $Q_T$, and uniqueness follows.
 \end{proof}

 \subsection{Existence of unique strong and classical solutions (proof of Theorems \ref{thm:scalarfinal}-\ref{thm:scalarclassicalsolution})}

 We are now ready to prove Theorems \ref{thm:scalarfinal}-\ref{thm:scalarclassicalsolution}.

\begin{proof}[Proof of Theorem \ref{thm:scalarfinal}]
    By Lemma \ref{lem:scalaresthigherorder}, the global weak solution obtained in either Theorem \ref{thm:existsmallmassscalar} (under the addition of hypothesis \textup{\textbf{\ref{h4}}}) or Theorem \ref{thm:existarbmassscalar} belongs to the class
         \begin{align*}
         u &\in L^\infty (Q_T) \cap L^\infty (0,T; L^1 ({\mathbb{R}^d})); \\
         \grad u &\in L^\infty (0,T; L^2 ({\mathbb{R}^d}));\\
         \Delta u &\in L^2 (Q_T).
     \end{align*}
     By Theorem \ref{thm:scalarunique} the solution is unique. Directly from \eqref{eq:mainscalar}, the spatial regularity of $u$ implies that $\frac{\p u}{\p t} \in L^2(Q_T)$. Consequently, $u \in C(0,T; L^2(\mathbb{R}^d))$ and so we conclude that $\norm{u(t,\cdot) - u_0(\cdot)}_{L^2(\mathbb{R}^d)} \to 0$ as $t \to 0^+$.
\end{proof}

\begin{proof}[Proof of Theorem \ref{thm:scalarclassicalsolution}]
    Working with the approximate solutions $u^\varepsilon$ as in \eqref{approximatescalar}, we expand the right hand side
    $$
\frac{\p u^\varepsilon}{\p t} - D \Delta u^\varepsilon = \grad u^\varepsilon \cdot \grad (K * \omega ^\varepsilon * u^\varepsilon) + u^\varepsilon \Delta (K * \omega^\varepsilon * u^\varepsilon) .
    $$
    By Lemma \ref{lem:scalaresthigherorder}, $\{ u^\varepsilon \}_{\varepsilon>0}$ is bounded in $L^\infty(Q_T)$; moreover, by Lemma \ref{lem:scalarLpgradient}, $\{ \grad u^\varepsilon \}_{\varepsilon>0}$ is bounded in $L^\infty(0,T; L^p (\mathbb{R}^d))$ for some $p > d+2$. Therefore, $\{ \grad u^\varepsilon \cdot \grad (K * \omega^\varepsilon * u^\varepsilon) \}_{\varepsilon>0}$ is bounded in $L^\infty(0,T; L^p (\mathbb{R}^d))$. Similarly, writing $\Delta (K * \omega ^\varepsilon * u^\varepsilon) = \grad (K* \omega^\varepsilon) * \grad u^\varepsilon$, Hypothesis \textbf{\ref{h2}} allows one to conclude that $\{ u^\varepsilon \grad (K* \omega^\varepsilon) * \grad u^\varepsilon \}_{\varepsilon>0}$ is also bounded in $L^\infty(0,T; L^p (\mathbb{R}^d))$. 

    Hence, $L^p$-estimates for parabolic equations (see, e.g., \cite[Theorem 6]{Wang2003}) imply that for any smooth, compact set $\Omega^\prime \subset \mathbb{R}^d$, $\{ u^\varepsilon \}_{\varepsilon>0}$ is bounded in $W^{1,2}_p ((0,T) \times \Omega^\prime)$, the space of functions twice weakly differentiable in space, once weakly differentiable in time, with derivatives belonging to $L^p((0,T) \times \Omega^\prime)$. Therefore, as $p > d+2$, the Sobolev embedding (see, e.g., \cite[Theorem 1.4.1]{wu2006elliptic}) implies that in fact $\{ u^\varepsilon\}_{\varepsilon>0}$ is bounded in $C^{\sigma, 1+\sigma} ([0,T] \times \overline{\Omega^\prime})$ for some $\sigma \in (0,1)$. Since $\Omega^\prime$ is arbitrary, we have that $u^\varepsilon \in C^{\sigma, 1+\sigma}_{loc} (Q_T)$ uniformly in $\varepsilon$. 

    We now set $w_i^\varepsilon := \tfrac{\p u^\varepsilon}{\p x_i}$ for each $i=1,\ldots,d$. Working with each component, we again expand the right hand side of the equation solved by $w_i^\varepsilon$:
  \begin{align}\label{thm:classical1.001}
\frac{\p w^\varepsilon_i}{\p t} - D \Delta w^\varepsilon_i =&\, \grad w^\varepsilon_i \cdot \grad (K * \omega ^\varepsilon * u^\varepsilon) + \grad u^\varepsilon \cdot \grad (K * \omega^\varepsilon *w_i ^\varepsilon) \nonumber \\ 
&\,+ w^\varepsilon_i \Delta (K * \omega^\varepsilon * u^\varepsilon) + u^\varepsilon \Delta (K * \omega^\varepsilon * w^\varepsilon_i).
\end{align}
Since $u^\varepsilon$ is bounded in $W^{1,2}_{p,loc} (Q_T)$, we have that $\grad w_i ^\varepsilon \in L^p_{loc}(Q_T)$ component-wise, for each $i=1,\ldots,d$. Moreover, by the classical differentiability obtained, $ w_i ^\varepsilon$ is locally uniformly bounded over $Q_T$ for each $i=1,\ldots,d$.
Therefore, the first and third terms on the right hand side of \eqref{thm:classical1.001} are uniformly bounded in $L^p_{loc}(Q_T)$. Furthermore, by hypotheses \textbf{\ref{h2}}, \textbf{\ref{h6}}, and the boundedness of $u^\varepsilon$, the second and fourth terms on the right hand side of \eqref{thm:classical1.001} are uniformly bounded in $L^p_{loc}(Q_T)$. To understand why we use the compact support of $K$, we look closely at the second term of \eqref{thm:classical1.001} in $L^p (B_R(0))$, the ball of radius $R>0$ centred at $0$. Since $\grad u^\varepsilon = (w_1^\varepsilon, \ldots, w_{d}^\varepsilon)$ is locally uniformly bounded, we need only control for $x \in B_R(0)$
\begin{align}
\as{(\grad K * \omega^\varepsilon) * w_i^\varepsilon (x) } \leq & \int_{\mathbb{R}^d} \as{\grad (K * \omega^\varepsilon)(y)} \as{w_i^\varepsilon(x+y)} {\rm d}y \nonumber \\
= & \int_{\supp { K } + B_1(0) } \as{\grad (K * \omega^\varepsilon)(y)} \as{w_i^\varepsilon(x+y)} {\rm d}y ,
\end{align}
and since $x \in B_R(0)$, we have $x+y \in B_{R+1}(0) + \supp{K}$. We may then use the local uniform boundedness of $w_i^\varepsilon$ to control $(\grad K * \omega^\varepsilon) * w_i^\varepsilon (t,x)$ in $L^p((0,T) \times B_R(0))$ for any $R>0$. The fourth term can be treated analogously.

Using once more $L^p$-estimates for parabolic equations and the Sobolev embedding, we conclude that $w_i ^\varepsilon = \tfrac{\p u^\varepsilon}{\p x_i} \in C^{\sigma, 1+\sigma}([0,T] \times \overline{\Omega^\prime})$ for some $\sigma \in (0,1)$, for any smooth, compact set $\Omega^\prime \subset \mathbb{R}^d$. Consequently, we conclude that $\{ u^\varepsilon \}_{\varepsilon>0}$ is bounded in $C^{\sigma, 2+\sigma}_{loc} (Q_T)$ for some $\sigma \in (0,1)$. In particular, $\grad u^\varepsilon \cdot \grad (K * \omega^\varepsilon * u^\varepsilon) + u^\varepsilon K * \omega^\varepsilon * \Delta u^\varepsilon \in C^{\sigma} _{loc} (Q_T)$ for some $\sigma \in (0,1)$.

Returning once more to the right hand side of the equation satisfied by $u^\varepsilon$, we have shown that $u_t ^\varepsilon - D \Delta u^\varepsilon \in C^{\sigma}_{loc} (Q_T)$. By Schauder estimates for parabolic equations (see, e.g., \cite[Theorem 5.1.8]{MR1329547}), $\{ u^\varepsilon \}_{\varepsilon>0}$ is bounded in $C^{1 + \sigma/2,2+\sigma}_{loc} (Q_T)$ for some $\sigma \in (0,1)$. Consequently, there exists a subsequence (still labelled by $\varepsilon$) such that $u^\varepsilon \to u \in C^{1+\sigma^\prime /2,2+\sigma^\prime} _{loc} (Q_T)$ as $\varepsilon\to0^+$ for any $\sigma^\prime < \sigma$, and $u$ is the classical solution to problem \eqref{eq:mainscalar}. By the assumed regularity of the initial data, we conclude that $u(t,x) \to u_0 (x)$ in $C(\mathbb{R}^d)$ as $t \to 0^+$.

We conclude by showing that the solution is strictly positive over $\mathbb{R}^d$ for all $t>0$. The approach is standard once we have the boundedness of $\Delta (K*u)$. Consider the auxiliary function $v := e^{\lambda t} u$ for $\lambda \in \mathbb{R}$ to be chosen later. Then $v$ is nonnegative, nonconstant, and satisfies
\begin{equation}\label{eq:strong_maximum_principle}
\frac{\p v}{\p t} - D \Delta v - \grad v \cdot \grad (K*u) - v [\Delta(K*u) + \lambda] = 0.
\end{equation}
By the smoothness of $u$ and Hypothesis \textbf{\ref{h6}}, we have that $K * \Delta u \in L^\infty((0,T) \times \Omega^\prime)$ for any $\Omega^\prime \subset \mathbb{R}^d$ compact, and any $T>0$ fixed. Suppose $\Omega^\prime$ is connected and that there exists a point $(t_0,x_0) \in (0,T) \times \Omega^\prime$ such that $v(t_0,x_0) = 0$. Choosing $\lambda$ large enough so that $\lambda - \as{K * \Delta u} \geq 0$ over $(0,T) \times \Omega^\prime$, we find that $v_t - D \Delta v - \grad (K*u) \cdot \grad v \geq 0$ over $(0,T) \times \Omega^\prime$. By the strong maximum principle, we conclude that $v$ is constant over $(0,t_0)\times\Omega^\prime$, a contradiction. 

Since $\Omega^\prime$, $T$ were arbitrary, we conclude that $v$ cannot attain a minimum (of zero) in any compact subset of $\mathbb{R}^d$ and must be positive everywhere in $(0,T) \times \mathbb{R}^d$.
\end{proof}

\section{The General \texorpdfstring{$n$}{}-Species System}\label{sec:nspeciescase}
In this section we establish the following theorems generalizing the scalar cases to the $n$-species system.

\begin{theorem}[Existence of global weak solution for $n$-species system with small mass]\label{thm:existsmallmasssystem}
    Assume \textup{\textbf{\ref{h1}}} holds and that $c_{i,n} > 0$ for each $i=1,\ldots,n$, where $c_{i,n}$ is defined as
    \begin{align}\label{const:ci}
    c_{i,n} := D_i - \frac{1}{2} \sum_{j=1}^n (m_i \norm{K_{ij}}_{L^\infty({\mathbb{R}^d})} + m_j \norm{K_{ji}}_{L^\infty({\mathbb{R}^d})} )
\end{align}
Then, there exists a global weak solution $\mathbf{u} \geq 0$ in $Q_T$ solving problem \eqref{maineq} in the sense of Definition~\ref{def:weaksoln}. Moreover, for all $i=1,\ldots,n$, if ${u}_{i0} \in L^p(\Rd)$, then $u_i \in L^\infty(0,T; L^p({\mathbb{R}^d}))$ for any $1 \leq p < \infty$, and if ${u}_{i0} \in L^2(\Rd)$, then $\grad u_i \in L^2(0,T; L^2({\mathbb{R}^d}))$. 
\end{theorem}

It is not difficult to see that when $n=1$, $c_{1,1}>0$ is equivalent to condition \eqref{const:c1} for the single species case. 

\begin{theorem}[Existence of global weak solution for $n$-species system with arbitrary mass]\label{thm:existarbmasssystem}
     Assume \textup{\textbf{\ref{h1}}}-\textup{\textbf{\ref{h3}}} and \textup{\textbf{\ref{h5}}} hold. Then for any initial mass there exists a global weak solution $\mathbf{u} \geq 0$ in $Q_T$ solving problem \eqref{maineq} in the sense of Definition \eqref{def:weaksoln}. Moreover, for all $i=1,\ldots,n$, if ${u}_{i0} \in L^p(\Rd)$, then $u_i \in L^\infty(0,T; L^p({\mathbb{R}^d}))$ for any $1 \leq p < \infty$, and if ${u}_{i0} \in L^2(\Rd)$, then $\grad u_i \in L^2(0,T; L^2({\mathbb{R}^d}))$.
\end{theorem}

\begin{theorem}[Existence of unique strong solution for $n$-species system]\label{thm:systemfinal}
Assume each component of the initial data satisfies $u_{i0}\in L^{\infty}(\Rd)$, $\grad u_{i0} \in L^2(\mathbb{R}^d)$ and $K_{ij}*u_{i0} \in W^{2,p}(\Rd)$ for some $p>2$, for each $i,j=1,\ldots,n$. Under the assumptions of Theorem \ref{thm:existarbmasssystem}, or under the assumptions of Theorem \ref{thm:existsmallmasssystem} together with \textup{\textbf{\ref{h4}}}, the obtained global weak solution is the unique, global strong solution solving problem \eqref{maineq} in the sense of Definition \ref{def:strongsolution}.
\end{theorem}

\begin{theorem}[Existence of unique classical solution for $n$-species system]\label{thm:systemfinal2}
Assume that each component of the initial data satisfies $u_{i0} \in C^3(\mathbb{R}^d)\cap L^\infty(\mathbb{R}^d)$ as well as $\nabla u_{i0} \in L^2(\mathbb{R}^d) \cap L^p(\Rd)$ for some $p>d+2$, and suppose one of the following:
    \begin{enumerate}
        \item in addition to the hypotheses of Theorem \ref{thm:existsmallmasssystem}, \textbf{\textup{\ref{h2}}} and \textbf{\textup{\ref{h6}}} hold;
        \item in addition to the hypotheses of Theorem \ref{thm:existarbmasssystem}, \textbf{\textup{\ref{h6}}} holds;
    \end{enumerate}
Then, the obtained unique global strong solution is the unique, global classical solution solving system \eqref{maineq} in the sense of Definition \ref{def:classicalsolution}. Moreover, each component $u_i$ is strictly positive over $\mathbb{R}^d$ for all $t>0$.
\end{theorem}

As in the single species case, we first treat the small mass case under hypothesis \textup{\textbf{\ref{h1}}}, and then treat the arbitrary mass case under the additional hypotheses \textup{\textbf{\ref{h2}}}-\textup{\textbf{\ref{h3}}} and \textup{\textbf{\ref{h5}}}. For strong solutions in the small mass case, the addition of hypotheses \textup{\textbf{\ref{h4}}is required as in the scalar case.} As much of the technical arguments are the same as in the scalar case, we present the key details when a careful manipulation of the indices of $K_{ij}$ is required. 

\subsection{Apriori estimates}

As in the scalar case, we first establish some a priori estimates under the assumption that we have a positive smooth solution $\mathbf{u} := (u_1, \ldots, u_n)$ solving problem \eqref{maineq}. To this end, fix $n \geq 2$ and define the following $n$-population counterparts to the entropy, total interaction energy, and free energy of the scalar system. 

We introduce the entropy (sometimes referred to as the weighted \textit{Shannon entropy} \cite{jungel2022nonlocal}):
$$
{\bf H} [ \mathbf{u}(t)] := \sum_{i=1}^n  D_i \pi_i H [ u_i(t) ].
$$
Then the total interaction energy is given by
$$
{\mathbfcal {K}}[\mathbf{u}(t)] := \frac12 \sum_{i,j=1}^n \int_{\mathbb{R}^d} \pi_i u_i (K_{ij}*u_j) \, \diff x,
$$
and so the free energy for a system in detailed balance (hypothesis \textup{\textbf{\ref{h5}}}) is given as
$$
{\mathbfcal{F}}[\mathbf{u}(t)] := \sum_{i=1}^n \int_{\mathbb{R}^d} D_i \pi_i  u_i \log u_i \, \diff x+\frac12 \sum_{i,j=1}^n \int_{\mathbb{R}^d} \pi_i u_i (K_{ij}*u_j) \, \diff x= {\bf H} [\mathbf{u}(t)] + {\mathbfcal{K}} [\mathbf{u}(t)] 
$$
In the small mass regime, we use the following auxiliary functional, the sum of the individual entropies of the species:
$$
{\bf H}_S [ \mathbf{u}(t)] := \sum_{i=1}^n  H [ u_i(t) ].
$$

We first prove the following, an $n$-species analogue to Lemma \ref{lem:iscalar}.
\begin{lemma}[A priori estimates under \textup{\textbf{\ref{h1}}} and small-mass condition]\label{lem:isystem}
    Fix $T>0$. Assume \textup{\textbf{\ref{h1}}} and suppose $c_{i,n} >0$ for each $i=1,\ldots,n$ where $c_{i,n}$ is as defined in the statement of Theorem \ref{thm:existsmallmasssystem}. Then for any smooth, positive solution $\mathbf{u}$ to problem \eqref{maineq} there holds
    \begin{align}
    \sup_{t \in (0,T)} I[u_i(t)] &\leq C; \label{est:isystem1.001} \\
        \norm{\grad \sqrt{u_i}}_{L^2(Q_T)}^2 &\leq C ;\label{est:isystem1.1} \\
        \norm{\sqrt{u_i} \grad K_{ij} * u_j}_{L^2 (Q_T)} &\leq C, \label{est:isystem1.2}
    \end{align}
    for each $i,j=1,\ldots,n$, where $C = C( D_i, T, c_{i,n}^{-1}, \norm{K_{ij}}_{L^\infty(\mathbb{R}^d)}, d, m_i, I[u_{i0}], H[u_{i0}]) \text{ for all } i,j=1,\ldots,n)$.
\end{lemma}

We also prove the following, an $n$-species analogue to Lemma \ref{lem:iscalararbmass}.

\begin{lemma}[Apriori estimates under \textup{\textbf{\ref{h1}}}-\textup{\textbf{\ref{h3}}}, \textup{\textbf{\ref{h5}}}]\label{lem:isystemarbmass}
 Fix $T>0$. Assume \textup{\textbf{\ref{h1}}}-\textup{\textbf{\ref{h3}}} and \textup{\textbf{\ref{h5}}}. Then for any smooth, positive solution $\mathbf{u}$ solving problem \eqref{maineq} there holds
    \begin{align}
       \sup_{t\in(0,T)}\int_{\mathbb{R}^d} u_i \as{x}^2 \diff x\leq C; \label{est:arbmasssys1.01} \\
       \sup_{t \in (0,T)} \int_{\mathbb{R}^d} u_i \as{\log u_i} \diff x+ \frac{1}{2} \norm{f_i}_{L^2(Q_T)}^2 \leq C; \label{est:arbmasssys1.02} \\
       \norm{\grad u_i}_{L^2(0,T; L^1(\mathbb{R}^d))} \leq C; \label{est:arbmasssys1.3.0} \\
               \norm{\sqrt{u_i} \grad K_{ij} * u_j}_{L^2(Q_T)} \leq C;\label{est:arbmass1.03} \\
        \norm{\grad \sqrt{u_i}}_{L^2(Q_T)} \leq C \label{est:arbmass1.04},
    \end{align}
    where $f_i = \sqrt{u_i} ( D_i \log u_i + \sum_{j=1}^n K_{ij} * u_j )$ for each $i=1,\ldots,n$ and 
    $
    C = C( D_i, T, \norm{K_{ij}}_{L^\infty(\mathbb{R}^d)},$ $ \pi_i, d, m_i, I[u_{i0}], H[u_{i0}])  \text{ for all } i,j=1,\ldots,n)$.
\end{lemma}

\begin{proof}[Proof of Lemma \ref{lem:isystem}]
    We first consider the quantity ${\bf H}_S[ \mathbf{u}(t)]$. Taking the derivative of ${\bf H}_S[ \mathbf{u}(t)]$ with respect to time, using the conservation of mass, and integrating by parts yields
\begin{align}\label{est:dwdt}
    \frac{{\rm d}}{{\rm d}t} {\bf H}_S [\mathbf{u}(t)] &= - \sum_{i=1}^n \int_{\mathbb{R}^d} \frac{\grad u_i}{u_i} \left( D_i \grad u_i + u_i \sum_{j=1}^n \grad ( K_{ij} * u_j ) \right) \diff x\nonumber \\
    &= - 4 \sum_{i=1}^n D_i \norm{\grad \sqrt{u_i} (t,\cdot)}_{L^2(\mathbb{R}^d)}^2 - \sum_{i,j=1}^n \int_{\mathbb{R}^d} \grad u_i \grad (K_{ij} * u_j ) \diff x.
\end{align}
By using the same estimates of the scalar case, \eqref{est:graduKscalar} and \eqref{est:graduscalar}, we obtain
\begin{align}\label{est:graduK}
    \as{\int_{\mathbb{R}^d} \grad u_i \grad (K_{ij} * u_j) \dx} &\leq 4 \, \norm{K_{ij}}_{L^\infty (\mathbb{R}^d)}\, \sqrt{m_i} \, 
    \sqrt{m_j} \, \norm{\grad \sqrt{u_i} (t,\cdot)}_{L^2(\mathbb{R}^d)} \, \norm{\grad \sqrt{u_j} (t,\cdot)}_{L^2(\mathbb{R}^d)} \nonumber  \\  
    &\leq 2\norm{K_{ij}}_{L^\infty (\mathbb{R}^d)} \left( m_i  \norm{\grad \sqrt{u_i} (t,\cdot)}_{L^2(\mathbb{R}^d)} ^2 + m_j  \norm{\grad \sqrt{u_j} (t,\cdot)}_{L^2(\mathbb{R}^d)} ^2 \right)
\end{align}
Summing over $i,j = 1, \ldots, n$, and carefully reindexing sums, \eqref{est:dwdt} becomes
\begin{align}\label{est:dwdt2}
     \frac{{\rm d}}{{\rm d}t} {\bf H}_S [\mathbf{u}(t)] &\leq 
    - 4 \sum_{i=1}^n \left[ D_i - \frac{m_i}{2} \sum_{j=1}^n \left(\norm{K_{ij}}_{L^\infty (\mathbb{R}^d)} + \norm{K_{ji}}_{L^\infty (\mathbb{R}^d)} \right)  \right] \norm{\grad \sqrt {u_i} (t,\cdot)}_{L^2(\mathbb{R}^d)}^2 \nonumber \\
        &= - 4\sum_{i=1}^n c_{i,n} \norm{\grad \sqrt{u_i} (t,\cdot)}_{L^2 ({\mathbb{R}^d})}^2 .
\end{align}

Integrating both sides from $0$ to $t$ yields
\begin{align}\label{est:isystem1.01.0}
    {\bf H}_S[\mathbf{u}(t)] + 4 \sum_{i=1}^n c_{i,n} \norm{\grad \sqrt{u_i}}_{L^2 (Q_t)}^2 \leq {\bf H}_S [ \mathbf{u}_0] .
\end{align}
Similarly as in \eqref{est:iscalar9.11}, we obtain
\begin{align}\label{est:isystem1.01.1}
    \sum_{i,j=1}^n \norm{\sqrt{u_i} \, \grad (K_{ij} * u_j) (t,\cdot)}_{L^2(\mathbb{R}^d)} ^2 &\leq 4 \sum_{i,j=1}^n \norm{K_{ij}}_{L^\infty (\mathbb{R}^d)}^2 m_i \, m_j \norm{\grad \sqrt{u_j} (t,\cdot)}_{L^2(\mathbb{R}^d)}^2 \nonumber \\
    &\leq C_2 \sum_{i=1}^n \norm{\grad \sqrt{u_i} (t,\cdot)}_{L^2(\mathbb{R}^d)}^2,
\end{align}
where $C_2 = C_2 (\norm{K_{ij}}_{L^\infty(\mathbb{R}^d)}, m_i)$ for all $i,j=1,\ldots,n$.

As in the proof of Lemma \ref{lem:iscalar}, we can control the second moment of all species as
\begin{align}\label{est:isystem1.0003}
   \sum_{i=1}^n \frac{{\rm d}}{{\rm d}t} I[u_i(t)] \leq &\, \varepsilon^{-1} \sum_{i=1}^n I[u_i(t)] + 4 \varepsilon \sum_{i=1}^n D_i ^2 \norm{\grad \sqrt{u_i} (t,\cdot)}_{L^2(\mathbb{R}^d)}^2 \nonumber \\
   &\, + \varepsilon n \sum_{i,j=1}^n \norm{\sqrt{u_i} \grad (K_{ij} * u_j) (t,\cdot)}_{L^2(\mathbb{R}^d)}^2 \nonumber \\
   \leq &\, \varepsilon^{-1} I[u_i(t)] + \varepsilon ( 4 \overline{D}^2 + C_2 n ) \sum_{i=1}^n \norm{\grad \sqrt{u_i} (t,\cdot)}_{L^2(\mathbb{R}^d)}^2,
\end{align}
where $\overline{D} = \max_{i} D_i $. Thus, combining estimates \eqref{est:isystem1.01.0}-\eqref{est:isystem1.0003}, as done in \eqref{est:scalarprime1.7}-\eqref{est:scalarprime1.9} for the scalar case, we find that for $\varepsilon$ sufficiently small there holds
\begin{align}\label{est:isystem1.0005}
    \sum_{i=1}^n &\left( \int_{\mathbb{R}^d} u_i \as{\log u_i} \, \diff x+ 2 c_{i,n} \norm{\grad \sqrt{u_i} }_{L^2(Q_t)}^2 + I[u_i(t)] \right) \leq 2 n C_1 \nonumber \\
    &\, + {\bf H}_S[\mathbf{u}_0] + 3 \varepsilon^{-1}  \sum_{i=1}^n\int_0 ^t I[u_i(t)] {\rm d}s ) 
\end{align}
where the control of the negative part of $u_i \log u_i$ for each $i$ is done as in \eqref{eq:arbitrarymass2.12prime}. Gr{\"o}nwall's lemma yields estimate \eqref{est:isystem1.001}, from which estimates \eqref{est:isystem1.1}-\eqref{est:isystem1.2} follow.
\end{proof}

\begin{proof}[Proof of Lemma \ref{lem:isystemarbmass}]
    As in the scalar case, the main idea is to write the dissipation for the $i^{\text{th}}$ component:
    $$
f_i := \sqrt{u_i} \, \grad \left( D_i \log u_i + \sum_{j=1}^n K_{ij} * u_j \right).
    $$
    Then, $(u_i)_t = \grad \cdot (\sqrt{u_i} f_i)$ for each component $u_i$. Under the detailed balance condition hypothesis \textup{\textbf{\ref{h5}}} there holds 
        $$
\sum_{i,j=1}^n \int_{\mathbb{R}^d} \pi_i u_i K_{ij} * (u_j)_t \diff x= \sum_{i,j=1}^n \int_{\mathbb{R}^d} \pi_i (u_i)_t K_{ij} * u_j \dx
    $$
    and so direct computation gives
    \begin{align}\label{est:arbmasssys1.11}
            \frac{{\rm d} }{{\rm d}t} {\mathbfcal{F}}[\mathbf{u}(t)]  + \sum_{i=1}^n \int_{\mathbb{R}^d} \pi_i \as{f_i}^2 \diff x= 0\,.
    \end{align}
Integrating from $0$ to $t$, and using the conservation of mass for each component then yields
\begin{align}
   {\mathbfcal{F}}[\mathbf{u}(t)] + \sum_{i=1}^n \pi_i \norm{f_i}_{L^2(Q_t)}^2 \leq \sum_{i,j=1}^n \pi_i \norm{K_{ij}}_{L^\infty(\mathbb{R}^d)} m_i m_j .
\end{align}
As it has been the case, we now control the negative part of $u_i \log u_i$ in terms of the second moment. It is not difficult to see that
\begin{align}
    \sum_{i=1}^n I[u_i(t)] \leq \varepsilon^{-1} \int_0 ^t I[u_i(s)] {\rm d} s + \varepsilon \sum_{i=1}^n \norm{f_i}_{L^2(Q_t)}^2 + \sum_{i=1}^n I[u_{i0}] ,
\end{align}
where $\varepsilon$ is to be chosen sufficiently small. Using the same procedure as in \eqref{est:isystem1.0003}-\eqref{est:isystem1.0005} yields estimates \eqref{est:arbmasssys1.01}-\eqref{est:arbmasssys1.02} for each $i=1,\ldots,n$. 

Next, we write $D_i \grad u_i = \sqrt{u_i} f_i - u_i \sum_{j=1}^n \grad (K_{ij} * u_j)$ to see that
\begin{align}\label{est:isystem1.0007}
    \sum_{i=1}^n D_i \norm{\grad u_i}_{L^2(0,T; L^1(\mathbb{R}^d))} \leq \sum_{i=1}^n \sqrt{m_i} \norm{f_i}_{L^2(Q_T)} + \sum_{i,j=1}^n \norm{u_i \grad (K_{ij} * u_j )}_{L^2(0,T; L^1(\mathbb{R}^d))}.
\end{align}
Then, as in \eqref{eq:arbitrarymass2.2} we consider for each $i$ the sets $\underline{\Omega}_i = \{ (t,x) : u_i \leq \ell \}$ and $\overline{\Omega}_i = \{ (t,x) : u_i > \ell \}$ to conclude that for any $\ell>1$ there holds
\begin{align}\label{est:isystem1.0009}
    \norm{u_i \grad (K_{ij} * u_j)}_{L^2(0,T; L^1(\mathbb{R}^d))} \leq C_K \ell m_j \sqrt{T} + \frac{C_3}{\log(\ell)} \norm{\grad u_j}_{L^2(0,T; L^1(\mathbb{R}^d))},
\end{align}
where $C_K := \max_{i,j} C_{K_{ij}}$, the largest constant coming from hypothesis \textup{\textbf{\ref{h2}}} across all components $K_{ij}$, and $C_3$ depends only on $\max_{i,j} \norm{K_{ij}}_{L^\infty(\mathbb{R}^d)}$ and $\max_{i} \norm{u_i \as{\log u_i}}_{L^\infty(0,T; L^1(\mathbb{R}^d))}$. Summing \eqref{est:isystem1.0009} over $i,j=1,\ldots,n$, combining with estimate \eqref{est:isystem1.0007} and choosing $\ell$ sufficiently large allows one to conclude that
\begin{align}
    \sum_{i=1}^n \norm{\grad u_i}_{L^2(0,T; L^1(\mathbb{R}^d))} \leq C_4 \sum_{i=1}^n m_i \norm{f_i}_{L^2(Q_T)} + C_K n^2 \ell \sqrt{T}  \max_{i} m_i 
\end{align}
and so estimate \eqref{est:arbmasssys1.3.0} follows.

As in estimate \eqref{eq:arbitrarymass2.3} for the scalar case, estimate \eqref{est:arbmasssys1.3.0} yields estimate \eqref{est:arbmass1.03}. Estimate \eqref{est:arbmass1.04} then follows from the $L^2(Q_T)$ estimates on $f_i$ and $\sqrt{u_i} \grad K_{ij} * u_j$ for each $i,j=1,\ldots,n$.    
\end{proof}

\subsection{Improved estimates and existence of weak solutions}

As in the scalar case, we improve these estimates to $u_i \in L^\infty (0,T; L^p ({\mathbb{R}^d})$ for any $1 \leq p < \infty$, and $\grad u_i$ in $L^2 (0,T; L^2 ({\mathbb{R}^d}))$ for each $i=1,\ldots,n$. 

\begin{lemma}[Improved estimates with no further assumptions for $n$-species case]\label{lem:system2}
    Assume the conditions of Lemma \ref{lem:isystem} (Lemma \ref{lem:isystemarbmass}) are satisfied. Then, for all $i=1,\ldots,n$, if $u_{i0}\in L^{p}(\Rd)$ for some $p \in (1,\infty)$, then
     \begin{equation}\label{est:system3}
        \norm{u_i}_{L^\infty(0,T; L^p({\mathbb{R}^d}))} \leq e^{C \max_{i,j}\norm{K_{ij}}_{L^\infty(\mathbb{R}^d)}^2 p / {D_i}} \norm{u_{i0}}_{L^p ({\mathbb{R}^d})},
    \end{equation}
    where $C$ is given in Lemma \ref{lem:isystem} (Lemma \ref{lem:isystemarbmass}). Moreover, for all $i=1,\ldots,n$, if $u_{i0} \in L^2(\Rd)$ then with $\tilde C = \tilde C(C)$ there holds
    \begin{equation}\label{est:system4}
        \norm{\grad u_i}_{L^2 (Q_T)} \leq   D_i^{-1} \tilde C \norm{u_{i0}}_{L^2 ({\mathbb{R}^d})}.
    \end{equation}
\end{lemma}
\begin{proof}
The proof is identical to Lemma \ref{lem:iscalar2} applied to each component $u_i$. 
\end{proof}
\begin{proof}[Proof of Theorems \ref{thm:existsmallmasssystem}-\ref{thm:existarbmasssystem}]
    The proof follows the blueprint of the scalar case. First, one can regularize each equation of system \eqref{maineq} for $\varepsilon>0$:
    $$
\frac{\p u_i ^\varepsilon}{\p t} = D_i \Delta u_i ^\varepsilon + \grad \cdot ( u_i ^\varepsilon \grad K_{ij} * \omega^\varepsilon * u_i ^\varepsilon )
    $$
    where $\omega^\varepsilon$ is the standard mollifier. Existence of a unique classical solution follows from, e.g., \cite{Giunta2021local}, or by generalizing the arguments used in \cite[Theorem 2.2]{Carrillo2020LongTime} from the scalar case to the system case. As in the scalar case, if $c_{i,n} > 0$ is satisfied for each $i=1,\ldots,n$ where $c_{i,n}$ is defined in \eqref{const:ci}, then it is also satisfied by $K_{ij} * \omega^\varepsilon$ and we remain in the small mass regime.

    By either Lemma \ref{lem:isystem} or Lemma \ref{lem:isystemarbmass}, we have the following estimates for all $i,j=1,\ldots,n$:
    \begin{enumerate}
    \item[(A)]$\{ \grad \sqrt{u_i^\varepsilon} \}_{\varepsilon>0}$ is bounded in $L^2 (Q_T)$; 
    \item[(B)] $\{ \sqrt{u_i^\varepsilon} \grad K_{ij} * {\omega}^\varepsilon * u_j^\varepsilon \}_{\varepsilon>0}$ is bounded in $L^2 (Q_T)$; 
    \item[(C)] $\{\sup_{t\in(0,T)} \int_{\mathbb{R}^d} u^\varepsilon_i \as{x}^2 \diff x\}_{\varepsilon>0}$ is  bounded.
\end{enumerate}

Notice that the estimates we have are exactly the same as in the scalar case. Following identical steps, one can finalize the proof of Theorems \ref{thm:existsmallmasssystem}-\ref{thm:existarbmasssystem}. 
\end{proof}

\subsection{Higher order estimates and uniqueness}

We now obtain improved estimates as in Lemmas \ref{lem:scalaresthigherorder}-\ref{lem:scalarLpgradient} for the $n$-species case under the additional hypotheses \textup{\textbf{\ref{h4}}}.

\begin{lemma}[Higher order estimates for $n$-species]\label{lem:systemesthigherorder}
Assume that each component of the initial data satisfies $u_{i0}\in L^{\infty}(\Rd)$, $\grad u_{i0} \in L^2(\mathbb{R}^d)$ and $K*u_{i0} \in W^{2,p}(\mathbb{R}^d)$ for some $p>2$. Suppose hypothesis \textup{\textbf{\ref{h4}}} holds in addition to the conditions of Lemma \ref{lem:isystem} (Lemma \ref{lem:isystemarbmass}). Then, there holds 
\begin{align}
 \norm{\grad u_i }_{L^\infty (0,T; L^2 ({\mathbb{R}^d}))} &\leq \hat C ; \label{est:system5} \\
    \norm{\Delta u_i}_{L^2 (Q_T)} &\leq \hat C; \label{est:system6} \\
    \norm{u_i }_{L^\infty (Q_T)} &\leq \hat C,\label{est:systemunif}
\end{align}
for each $i=1,\ldots,n$, where $\hat C = \hat C( D_i, T, C, \tilde C,\max_{i,j} \norm{\grad (\tilde K_{ji} * K_{ji})}_{L^2(\mathbb{R}^d)})$ and $C$, $\tilde C$ are given in Lemmas \ref{lem:isystem} and \ref{lem:system2} (Lemmas \ref{lem:isystemarbmass} and \ref{lem:system2}).
 \end{lemma}
\begin{proof}
    Again, the proof is very similar to the proof of Lemma \ref{lem:scalaresthigherorder} so we provide only key details. First, testing the equation for $u_i$ against $\tilde K_{ji} * K_{ji} * \Delta u_i$ we find
    \begin{align*}
        \frac{1}{2} \frac{{\rm d}}{{\rm d}t} \int_{\mathbb{R}^d} \magg{\grad (K_{ji} *  u_i)} \diff x\,+&\, D_i \int_{\mathbb{R}^d} \magg{\Delta (K_{ji} * u_i)} \diff x= \\
        & \sum_{l = 1}^n \int_{\mathbb{R}^d} \Delta (K_{ji} * u_i ) K_{ji} * (\grad u_i \cdot \grad (K_{il} * u_l)) \diff x\\
        +& \sum_{l=1}^n \int_{\mathbb{R}^d} u_i \, \Delta (K_{il} * u_l) \tilde K_{ji} * \Delta (K_{ji} * u_i)  \diff x=: I_1 ^i + I_2 ^i .
    \end{align*}
    Cauchy's inequality and Young's convolution inequality yields for each $i$
    \begin{align*}
        \as{I_1 ^i} \leq \frac{D_i}{4} \norm{\Delta (K_{ji} * u_i)}^2_{L^2(\mathbb{R}^d)} + \underline{D}^{-1} n a^2(t) \sum_{l=1}^n \norm{ \grad (K_{il} * u_l) (t,\cdot)}_{L^2(\mathbb{R}^d)}^2 ,
    \end{align*}
    where $\underline D = \min_i D_i$, $a(t) = \max_{i,j} \norm{K_{ij}}_{L^2(\mathbb{R}^d)} \norm{ \grad u_i (t,\cdot)}_{L^2(\mathbb{R}^d)}$ and $a^2(t) \in L^1 (0,T)$ by Lemma \ref{lem:system2}. Similarly, H{\"o}lder's inequality and Cauchy's inequality yields for each $i$
    \begin{align*}
        \as{I_2 ^i} \leq \sum_{l=1}^n \frac{D_l}{4 n} \norm{\Delta (K_{il} * u_l)}_{L^2(\mathbb{R}^d)}^2 + \underline{D}^{-1} b^2(t),
    \end{align*}
    where $b(t) = n \max_{i,j} \norm{\grad (\tilde K_{ji} * K_{ji})}_{L^2(\mathbb{R}^d)} \norm{u_i}_{L^\infty(0,T; L^2(\mathbb{R}^d)} \norm{\grad u_i (t,\cdot)}_{L^2(\mathbb{R}^d)}$, and in particular, $b^2(t) \in L^1(0,T)$. Hence, we have for each $i$ 
    \begin{align*}
        \frac{1}{2} \frac{{\rm d}}{{\rm d}t} \int_{\mathbb{R}^d} &\magg{\grad (K_{ji} * u_i)} \diff x+ \frac{3 D_i}{4} \int_{\mathbb{R}^d} \magg{\Delta ( K_{ji} *  u_i )} \diff x\leq \\
        &\underline{D}^{-1} n a^2(t) \sum_{l=1}^n \norm{ \grad (K_{il} * u_l) (t,\cdot)}_{L^2(\mathbb{R}^d)}^2 
        + \sum_{l=1}^n \frac{D_l}{4 n} \norm{\Delta (K_{il} * u_l)}_{L^2(\mathbb{R}^d)}^2 + \underline{D}^{-1} b^2(t),
    \end{align*}
    and so if we sum both sides over $i,j=1,\ldots,n$ and reindex the sums on the right hand side we obtain
    \begin{align}
        \sum_{i,j=1}^n \left( \frac{1}{2} \frac{{\rm d}}{{\rm d}t} \right.&\left.\int_{\mathbb{R}^d} \magg{\grad (K_{ji} * u_i)} \diff x+ \frac{D_i}{2} \int_{\mathbb{R}^d} \magg{\Delta ( K_{ji} *  u_i )} \diff x\right) \nonumber \\
        &\leq \underline{D}^{-1} n^2 \left( a^2(t) \sum_{i,j=1}^n \norm{\grad (K_{ji} * u_i ) (t,\cdot)}_{L^2(\mathbb{R}^d)}^2 + b^2(t)  \right)
    \end{align}
    Gr{\"o}nwall's lemma gives that $K_{ij} * u_j \in L^\infty(0,T; L^2(\mathbb{R}^d))$ and $K_{ij} * \Delta u_j \in L^2(Q_T)$ for all $i,j = 1,\ldots,n$.

    By an identical argument to that made in the proof of Lemma \ref{lem:scalaresthigherorder} using maximal regularity of the heat equation, we conclude that estimates \eqref{est:system5}-\eqref{est:system6} hold, completing the first part of the proof.

    The uniform estimates on each component $u_i$ follows in an identical fashion to the scalar case, using the previously obtained estimates and the quantity $\max_{i,j} \norm{K_{ij} * \grad u_j}_{L^\infty (Q_T)}$ replacing the constant $C_0 = \norm{K * \grad u}_{L^\infty (Q_T)}$ in the proof of Lemma \ref{lem:scalaresthigherorder}.

\end{proof}

As in the scalar case, estimates obtained in Lemma \ref{lem:systemesthigherorder} give rise to improved integrability of the gradient of each component of the solution. The proof is identical to the proof of Lemma~\ref{lem:systemesthigherorder}, working with the sum across all components, and so we omit the details.
 \begin{lemma}[Improved integrability of gradient]\label{lem:systemLpgradient}
     Suppose each component of the initial data satisfies $\grad u_{i0} \in L^p(\mathbb{R}^d)$ for some $p \in [2,\infty)$. Under the conditions of Lemma \ref{lem:systemesthigherorder} there exists a constant $\tilde C_0$ such that there holds
     \begin{align}\label{est:graduLpsys}
         \norm{\grad u_i}_{L^\infty (0,T; L^p(\mathbb{R}^d))} \leq \hat C_0,
     \end{align}
     for each $i=1,\ldots,d$, where $\hat C_0 = \hat C_0 (\hat C, p, \norm{\grad u_{i0}}_{L^p (\mathbb{R}^d)})$, and $\hat C$ is given in Lemma \ref{lem:systemesthigherorder}.
 \end{lemma}

\begin{theorem}\label{thm:uniquesystem}
    Suppose \textup{\textbf{\ref{h4}}} holds. Then any nonnegative solution $\mathbf{u} = (u_1, \ldots, u_n)$ solving problem \eqref{maineq} belonging to the class
    \begin{align*}
                 u_i &\in L^\infty (Q_T) \cap L^\infty (0,T; L^1 ({\mathbb{R}^d})); \\
         \grad u_i &\in L^\infty (0,T; L^2 ({\mathbb{R}^d}));\\
         \Delta u_i &\in L^2 (Q_T),
    \end{align*}
    for each $i=1,\ldots, n$, is unique.
\end{theorem}
\begin{proof}
    Set $\underline{D} := \min_{i} D_i > 0$. We begin as in the proof of Theorem \ref{thm:scalarunique}: assume that there are two solutions, $\mathbf{u}$ and $\mathbf{v}$. Denote by $w_i := u_i - v_i$. Then, for each $i=1,\ldots,n$, $w_i$ solves
    \begin{align}\label{eq:sysunique0.0}
        (w_i)_t - D_i \Delta w_i = \grad \cdot\left( w_i \sum_{j=1}^n \grad (K_{ij} * u_j ) \right) + \grad \cdot \left( v_i \sum_{j=1}^n \grad (K_{ij} * w_j) \right).
    \end{align}
    Testing against $w_i$ for each $i=1,\ldots,n$, summing over $i=1,\ldots,n$ and following an identical procedure to that of the proof of Theorem \ref{thm:scalarunique} yields
    \begin{align}\label{eq:sysunique1.0}
        \frac{1}{2} \sum_{i=1}^n \left( \frac{{\rm d} }{{\rm d}t} \int_{\mathbb{R}^d} w_i ^2 \diff x+ D_i \int_{\mathbb{R}^d} \magg{\grad w_i} \diff x \right) \leq &\, C_1(t) \sum_{i=1}^n \int_{\mathbb{R}^d} w_i ^2 \diff x\nonumber \\
        &+ C_2 \sum_{i,j=1}^n \int_{\mathbb{R}^d} \magg{\grad (K_{ij} * w_j)} \diff x,
    \end{align}
    where $C_1(t) = \max_{i,j} \norm{K_{ij} * \Delta u_j (t,\cdot)}_{L^\infty({\mathbb{R}^d})} \in L^1(0,T)$ and $C_2 = \max_{i} \norm{v_i}_{L^\infty(Q_T)} / \underline{D}$. 

    Now we seek to estimate the second term on the right hand side of \eqref{eq:sysunique1.0}. To this end, we consider equation \eqref{eq:sysunique0.0} by first reindexing $j \mapsto l$ followed by $i \mapsto j$. We then test with $\tilde K_{ij} * K_{ij} * w_j$ and sum the result over $j=1,\ldots,n$ to obtain
    \begin{align}\label{eq:sysunique1.2}
       \sum_{j=1}^n \left( \frac{1}{2} \frac{{\rm d}}{{\rm d}t}  \int_{\mathbb{R}^d} \magg{K_{ij} * w_j} \diff x\right.&\left. + \, D_j \int_{\mathbb{R}^d} \magg{\grad (K_{ij} * w_j)} \diff x \right) \nonumber \\
        =&\, - \sum_{j,l=1}^n \int_{\mathbb{R}^d} \grad ( \tilde K_{ij} * K_{ij} * w_j) \cdot w_j \grad (K_{jl} * u_l )) \diff x\nonumber \\
        &\, - \sum_{j,l=1}^n \int_{\mathbb{R}^d} \grad (\tilde K_{ij} * K_{ij} * w_j) \cdot v_j \grad (K_{jl} * w_l ) \diff x.
    \end{align}
    We may easily estimate the first term on the right hand side of \eqref{eq:sysunique1.2} as done in the proof of Theorem \ref{thm:scalarunique} to obtain
    \begin{align}\label{eq:sysunique1.3}
        \as{- \sum_{j,l=1}^n \int_{\mathbb{R}^d} \grad (\tilde K_{ij} * K_{ij} * w_j) \cdot w_j \grad (K_{jl} * u_l )) \dx} \leq C_3 n \sum_{j=1}^n \norm{w_j (t,\cdot)}_{L^2({\mathbb{R}^d})}^2 ,
    \end{align}
    where $C_3$ depends on $\norm{\grad (\tilde K_{ij} * K_{ij})}_{L^2(\mathbb{R}^d)}$ and $\norm{\grad (K_{jl} * u_l))}_{L^\infty(0,T; L^2(\mathbb{R}^d))}$, finite by hypothesis \textup{\textbf{\ref{h4}}} and by the assumed regularity of each component $u_i$ respectively, for each $i,j,l = 1,\ldots,n$.
    
    The second term on the right hand side of \eqref{eq:sysunique1.2} requires a bit more care. First, for fixed $i$, $j$, $l$, we can estimate as in the proof of Theorem \ref{thm:scalarunique} to deduce
    \begin{align}\label{eq:sysunique1.4}
        \Big|\int_{\mathbb{R}^d} \grad (\tilde K_{ij} * K_{ij} * w_j) \cdot v_j & \grad (K_{jl} * w_l ) \dx\Big| \nonumber\\ &\leq \frac{C_5}{2} \left( \varepsilon_l \norm{\grad K_{jl} * w_l (t,\cdot)}_{L^2({\mathbb{R}^d})}^2 + \varepsilon_l ^{-1} \norm{w_j (t,\cdot)}_{L^2({\mathbb{R}^d})}^2 \right),
    \end{align}
    where $C_5$ depends on $\norm{\grad (\tilde K_{ij} * K_{ij} ) }_{L^2({\mathbb{R}^d})}$, $\norm{v_j}_{L^\infty(Q_T)}$, and  $\varepsilon_l > 0$ is to be determined. Notice carefully that the right hand side does not depend on $i$ as this dependence is absorbed into the condition $\grad (\tilde K_{ij} * K_{ij} ) \in L^2({\mathbb{R}^d})$. Summing \eqref{eq:sysunique1.4} over $j,l = 1,\ldots,n$ then gives
    \begin{align}\label{eq:sysunique1.5}
       \Big|\sum_{j,l=1}^n \int_{\mathbb{R}^d} \grad (\tilde K_{ij} \,*&\, K_{ij} * w_j) \cdot v_j  \grad (K_{jl} * w_l ) \dx\Big| \nonumber \\
        &\leq \frac{C_5}{2} \sum_{j,l=1}^n \left( \varepsilon_l \norm{\grad (K_{jl} * w_l) (t,\cdot)}_{L^2({\mathbb{R}^d})} + \varepsilon_l ^{-1} \norm{w_j (t,\cdot)}_{L^2({\mathbb{R}^d})}^2 \right).
    \end{align}
Combining estimates \eqref{eq:sysunique1.3} and \eqref{eq:sysunique1.5} we arrive at
\begin{align}
    \sum_{j=1}^n \left( \frac{1}{2} \frac{{\rm d}}{{\rm d}t} \int_{\mathbb{R}^d} \right.&\left. \magg{K_{ij} * w_j} \diff x+ D_j \int_{\mathbb{R}^d} \magg{\grad (K_{ij} * w_j)} \diff x \right) \leq C_3 n \sum_{j=1}^n \norm{w_j (t,\cdot)}_{L^2({\mathbb{R}^d})}^2 \nonumber \\
    &+ \frac{C_5}{2} \sum_{j,l=1}^n \left( \varepsilon_l \norm{\grad (K_{jl} * w_l) (t,\cdot)}_{L^2({\mathbb{R}^d})} + \varepsilon_l ^{-1} \norm{w_j (t,\cdot)}_{L^2({\mathbb{R}^d})}^2 \right).
\end{align}
Notice again that the right hand side does not depend on $i$. Hence, we choose $\varepsilon_l = \frac{D_l}{C_5 n}$ for each $l = 1,\ldots,n$, sum both sides from $i=1,\ldots,n$, and then reindex the right hand side as $j \mapsto i$ followed by $l \mapsto j$ to conclude that
\begin{align}\label{eq:sysunique1.6}
    \frac{1}{2} \sum_{i,j=1}^n \left( \frac{{\rm d}}{{\rm d}t} \int_{\mathbb{R}^d} \magg{K_{ij} * w_j} \diff x+ D_j \int_{\mathbb{R}^d} \magg{\grad (K_{ij} * w_j)} \diff x\right) \leq C_6 (t) \sum_{i=1}^n \norm{w_j (t,\cdot)}_{L^2 ({\mathbb{R}^d})}^2 ,
\end{align}
for some $C_6 (t) \in L^1(0,T)$. 

To conclude, we multiply both sides of estimate \eqref{eq:sysunique1.0} by $\frac{\underline{D}}{4C_2}$ and combine the result with estimate \eqref{eq:sysunique1.6} to obtain
\begin{align}
    &\frac{\underline{D}}{8 C_2} \sum_{i=1}^n \frac{{\rm d} }{{\rm d}t} \int_{\mathbb{R}^d} w_i ^2 \diff x+ \frac{1}{2} \sum_{i,j=1}^n \frac{{\rm d}}{{\rm d}t} \int_{\mathbb{R}^d} \magg{K_{ij} * w_j} \diff x\leq C_7(t) \sum_{i=1}^n \int_{\mathbb{R}^d} w_i ^2 \diff x,
\end{align}
where $C_7(t) \in L^1(0,T)$. The conclusion of the theorem follows from Gr{\"o}nwall's lemma.
\end{proof}

\subsection{Existence of unique, global classical solutions}

We conclude with the proof of Theorems \ref{thm:systemfinal}-\ref{thm:systemfinal2}.
\begin{proof}[Proof of Theorem \ref{thm:systemfinal}]
    By Lemma \ref{lem:systemesthigherorder}, the global weak solution $\mathbf{u}$ obtained in either Theorem \ref{thm:existarbmasssystem}, or Theorem \ref{thm:existsmallmasssystem} in addition to hypothesis \textup{\textbf{\ref{h4}}}, satisfies for each component $u_i$:
        \begin{align*}
                 u_i &\in L^\infty (Q_T) \cap L^\infty (0,T; L^1 ({\mathbb{R}^d})); \\
         \grad u_i &\in L^\infty (0,T; L^2 ({\mathbb{R}^d}));\\
         \Delta u_i &\in L^2 (Q_T).
    \end{align*}
    Arguing as in the scalar case, each component $u_i$ satisfies $\norm{u_i (t,\cdot) - u_{i0}}_{L^2(\mathbb{R}^d)} \to 0$ as $t \to 0^+$, and so $\mathbf{u} = (u_1,\ldots,u_n)$ is a strong solution solving problem \ref{maineq} in the sense of Definition \ref{def:strongsolution}.
\end{proof}

We conclude with the proof of Theorem \ref{thm:systemfinal2}. The ingredients are essentially identical to the proof of Theorem \ref{thm:scalarclassicalsolution} and so we omit most details.
\begin{proof}[Proof of Theorem \ref{thm:systemfinal2}]
    The argument is identical to that of Theorem \ref{thm:scalarclassicalsolution}. First, the improved integrability of the gradient of each component of the solution $u_i$ allows one to conclude by $L^p$-theory of parabolic equations and the Sobolev embedding that $u_i \in C^{\sigma, 1+\sigma}_{loc} (Q_T)$ for some $\sigma \in (0,1)$, since $p > d+2$. Working with each component $w_{i,j} := \frac{\p u_i}{\p x_j}$, we find that under hypotheses \textbf{\ref{h2}} and \textbf{\ref{h6}}, $w_{i,j} \in C^{\sigma, 1+\sigma}_{loc}  (Q_T)$ for some $\sigma \in (0,1)$, for each $i,j = 1,\ldots,d$. Consequently, each component $u_i \in C^{\sigma, 2+\sigma}_{loc} (Q_T)$ and Schauder estimates imply that $u_i \in C^{1 + \sigma^\prime/2,2+\sigma^\prime} _{loc} (Q_T)$ for each $i=1,\ldots,n$, for any $\sigma^\prime < \sigma$. This additional regularity paired with \textbf{\ref{h6}} and the strong maximum principle yields the strict positivity of each component $u_i$.
\end{proof}

\section{Numerical Simulations}\label{sec:apps}

In this section, we present some illustrative numerical examples in one spatial dimension and discuss how they relate to the theory developed in the previous sections. A primary motivation of this exploration is that the long-time asymptotics for general kernels is not entirely clear. In particular, whether nontrivial stationary states solving the scalar equation \eqref{eq:mainscalar} exist on the whole space remains a challenging question. Consider the following: in \cite[Theorem 1.3]{MR4544662}, the authors establish some minimal requirements on the kernel $K$ such that solutions to \eqref{eq:mainscalar} satisfy $\norm{u(t, \cdot) - G(t,\cdot)}_{L^1(\mathbb{R}^d)} \to 0$ in the long-time limit, where $G(t,x)$ is the heat kernel on $\mathbb{R}^d$. More precisely, in addition to some other technical requirements, they assume that $K \in W^{1,\infty}(\mathbb{R}^d)$ with $\grad W \in L^{d - \varepsilon} (\mathbb{R}^d)$ for some $\varepsilon > 0$. On the other hand, as noted in \cite{MR4544662}, when the kernel $K$ is sufficiently singular (e.g., $K(x) = -\alpha \ln \as{x}$, $\alpha >0$, as in the case of Keller-Segel, see \cite{MR2568716}), nontrivial solution profiles may exist in the long-time limit. This leaves a gap in the regularity assumptions on $K$ that is occupied by those kernels of bounded variation: they do not have a well-defined element $\grad K$ belonging to any $L^p$-space, nor are they singular. Relatedly, in \cite[Theorem 6.1]{MR3950320}, it is shown that for linear diffusion, global minimisers do not exist for kernels satisfying $\limsup_{\as{x} \to \infty} \grad K \cdot x < 2 d D$; when $K$ is only of bounded variation, it is not immediately clear whether those techniques used could produce the same result for lower regularity kernels. 

We therefore demonstrate differences in dynamical behaviour for the prototypical top-hat kernel defined in \eqref{detectionkernelT}, which satisfies \textbf{\ref{h2}} but otherwise falls outside of the regularity regimes of currently known results. We consider several aspects: a. the magnitude of the strength of attraction to the potential; b. whether the potential is attractive/repulsive; c. differences between the single- and multi-species cases. 

In our numerical simulations, it is difficult to determine whether a true (numerical) steady state has been reached, or whether \textit{metastability} is present for this particular kernel \cite{carrillo2019aggregation}. In some cases (e.g., Figure \ref{fig:scalar1}), solutions appear to decay as in the smooth-kernel cases; in other cases (e.g., Figures \ref{fig:scalar2}-\ref{fig:scalar3}), it is less clear that solutions will continue to decay. This leaves the following possibilities for general kernels:
\begin{enumerate}
    \item Case I: there are no nontrivial stationary states for problem \eqref{eq:mainscalar} on the whole space; if so, the observed ``stationary'' profiles may be metastable states, or otherwise some nontrivial state that appears due to an approximation of the whole space via a bounded domain with no-flux boundary conditions. This suggests that, as the numerical domain length $L$ increases, the observed state would decrease.
    \item Case II: there exist nontrivial stationary states for problem \eqref{eq:mainscalar} on the whole space; if so, we may then ask several further questions, including how many exist, and their stability. 
\end{enumerate}
In either case, it is an interesting and challenging question to ask: what are the minimal regularity requirements on the kernel $K$ so that nontrivial stationary states solving problem \eqref{eq:mainscalar} exist on the whole space, or even on a bounded domain with zero-flux boundary conditions? 

Once such a criterion has been established, we may then ask: are these same conditions shared by the $n$-species system \eqref{maineq}? In addition to the implications that kernel regularity may have on the existence/non-existence of stationary states to \eqref{maineq} on the whole space, there are other features not covered by our previous analysis. In particular, it is unclear whether solutions exist in the arbitrary mass regime when the detailed balance condition \textbf{\ref{h5}} is false; this has several implications, such as the lack of a free-energy functional, which is essential to proving the non-existence of local/global minimisers as achieved in \cite{MR3950320}. This lack of energy structure may yield long-term dynamics that are significantly different from the scalar cases. We hope that these simulations, as complements to the rigorous analysis, provide some possibilities for future study of cases falling outside our technical criteria used.

Using a finite-volume method proposed in \cite{MR3372289, MR4605931}, we solve the problem in a fixed one-dimensional domain with no-flux boundary conditions for $1$- and $2$-species cases as an approximation to the problem on the whole space. We describe the method briefly here, directing readers to \cite{Carrillo2015FiniteVolume} for further details and discussion. The approximate domain $(-L,L)$ is divided into uniform finite-volume cells $C_j = [ x_{j-\tfrac{1}{2}}, x_{j + \tfrac{1}{2}} ]$ of size $\Delta x$, where $x_j = j \Delta x$, for a fixed number of cells $j \in \{ - M, \ldots, M\}$. We approximate the cell averages of the solution $u$, $\overline{u}_j (t) := \frac{1}{\Delta x} \int_{C_j} u(t,x) \dx$, via the system of ODEs for $\overline{u}_j (t)$:
\begin{align*}
    \frac{{\rm d} \overline{u}_j}{{\rm d}t} = - \left( \frac{ F_{j+1/2} - F_{j-1/2}}{\Delta x} \right),
\end{align*}
which is obtained by integrating \eqref{eq:mainscalar} each cell $C_j$. The numerical flux $F_{j + 1/2} (t)$ is an approximation of the continuous flux 
\begin{align}\label{eq:cont_flux}
\xi(t,x) := D \log (u) +  K * u
\end{align}
at the cell interface $x_{j+1/2}$ at time $t$. To compute the flux, we use a standard upwind scheme with a generalised minmod limiter, see \cite[Eq. (2.2)-(2.7)]{Carrillo2015FiniteVolume}. The scheme is second order accurate in space, and we use 100 cells per spatial unit, i.e., given a domain of length $2L$, $200^*L$ cells are used to solve the problem over $(-L,L)$. Time stepping is performed using a third order strong stability preserving Runge-Kutta method \cite{MR1854647}. 

\subsection{The Scalar Equation}

Several models studied in mathematical ecology feature a spatial convolution with a tophat kernel directing movement towards or away from some stimulus (e.g., a resource gradient \cite{Fagan2017perceptual}, scent marks on a landscape \cite{pottslewis2019}, remembered locations of previous interactions \cite{pottslewis2016,potts2016territorial}, or detection of the population density itself \cite{pottslewis2019, wangsalmaniw2022}). We consider the following prototypical scenario with tophat kernel:
\begin{align}\label{application:scalartophat}
    \begin{cases}
        \frac{\p u}{\p t} = \frac{\p}{\p x} \left( D \frac{\p u}{\p x} + u \frac{\p  K_{\text{tophat}} * u }{\p x} \right) \quad \text{ in } (0,T) \times \mathbb{R}, \\
        u(0,x) = \chi_{(-\ell,\ell)} (x), \quad\quad\quad\quad\quad\quad\quad \text{ in } \mathbb{R},
\end{cases}
\end{align}
where $D>0$ and $\chi_{(-\ell,\ell)}(x)$ is the characteristic function of the symmetric interval of radius $\ell>0$, normalized such that $\norm{\chi_{(-\ell,\ell)}}_{L^1(\mathbb{R})} = 1$. Here, $K_{\text{tophat}}$ denotes the top-hat kernel with a strength parameter $\alpha$ and perceptual radius $R$ introduced in \eqref{detectionkernelT}. Members of population $u$ are able to detect a nonlocal average density $R$ units about the location of detection, biasing the movement towards/away from areas of high/low density regions depending on the sign and magnitude of the parameter $\alpha$. The convention in \eqref{detectionkernelT} is such that it is an \textit{attractive} kernel, hence $\alpha>0$ corresponds to attraction to areas of high density, whereas $\alpha<0$ corresponds to repulsion from high density areas.

As previously noted, the top-hat kernel satisfies hypotheses \textup{\textbf{\ref{h1}}}-\textup{\textbf{\ref{h3}}} and \textbf{\ref{h6}}, but does not have a gradient in any $L^p$ space. In this example, it is informative to note that in one dimension the problem becomes \textit{local} in the sense that the nonlocal interaction term can be written as the difference of Diracs:
$$
\frac{\p  K_{\text{tophat}} * u }{\p x} (t,x) = -\alpha \left( \frac{u(t,x+R) - u(t,x-R)}{2 R} \right).
$$

Therefore, for any given smooth $\phi$ there holds
$$
\norm{\frac{\p}{\p x} ( K_{\text{tophat}} * \phi) }_{L^1(\mathbb{R})} \leq \frac{\as{\alpha}}{R} \norm{\phi}_{L^1(\mathbb{R})},
$$
and so \textup{\textbf{\ref{h2}}} is satisfied with $C_K = \as{\alpha}/ R$. By the comments made in the introduction, this is sufficient to conclude that hypothesis \textup{\textbf{\ref{h4}}} is also satisfied. Hypothesis \textup{\textbf{\ref{h3}}} and \textup{\textbf{\ref{h6}}} are obviously satisfied.

By Theorem \ref{thm:scalarclassicalsolution}, there exists a unique, global classical solution solving problem \eqref{application:scalartophat}. Strictly speaking, the initial data does not satisfy the hypotheses of the Theorem; in this sense, the solution is only classical on the open time interval $(0,T)$ and the initial data is satisfied in the sense of a weak solution.

The size of the simulated domain $L$ is chosen large enough to approximate the problem on the whole $\mathbb{R}$ if the long-time behaviour leads to a fast-decaying at infinity steady state. If the diffusion is dominating, then the no-flux boundary condition will change the long-time asymptotic behaviour of the equation. Our numerical experiments indicate that if there are steady solutions of \eqref{application:scalartophat}, they are supported on the whole domain as one expects from the effect of the linear diffusion term.

In Figures \ref{fig:scalar1}-\ref{fig:scalar4}, we present simulations for four exemplary cases of the dynamical behaviour. In all cases, we fix
\begin{align}\label{eq:numerics_fixed_values}
    D = 0.25, \quad R=1.0, \quad \ell = 4.0,
\end{align}
and $L$ is indicated in each of the figure captions. In each figure, we change only the strength parameter $\alpha$. In the upper panel (A), we display a contour plot of the solution profile and plot the free energy with respect to time; in the lower panel (B), we display cross-sections of the solution profile at select times. The solution profile is the solid black line, while the quantity $\as{\xi}$ is the dashed black line, where $\xi$ is given by \eqref{eq:cont_flux} with $K = K_{\textup{tophat}}$, which should be constant on the support of the solution $u$ when a (numerical) steady-state is achieved. Studying this quantity provides evidence that we are indeed at a (numerical) steady-state. Of course, by Theorem \ref{thm:scalarclassicalsolution} we know that the solution is strictly positive everywhere; therefore, the grey columns indicate regions for which the solution $u$ is smaller than $10^{-4}$.

In Figure \ref{fig:scalar1}, we have a small amount of self-attraction ($\alpha = 2$). In this case, the diffusive forces clearly dominate and the solution settles down to the constant state (i.e., $u \equiv 1/2L$). The free energy decreases very slowly. On a $\log$-scale (not depicted here), the free energy is still decreasing. We observe similar behaviour for small amounts of self-repulsion. 

In Figure \ref{fig:scalar2}, we have a large amount of self-attraction ($\alpha = 30$). In contrast to Figure \ref{fig:scalar1}, the attractive forces appear to dominate and the solution concentrates a majority of its mass within a single region about $x=0$. Within the grey column, the ``numerical support'' of $u$, we observe at time $t=100$ that $\xi$ is constant. For longer times, $\xi$ appears to approach a constant value over the entire domain (as in, e.g., Figure \ref{fig:scalar3}). On a $\log$-scale, different from the weak attraction case, the free energy does not appear to be changing. By the numerical method used here, it remains difficult to determine whether this is an effect of the no-flux boundary condition alone.

In contrast to the single peak observed for strong attraction, in Figure \ref{fig:scalar3} we observe two distinct peaks when the amount of self-attraction is intermediate between the strengths used for Figure \ref{fig:scalar1} and Figure \ref{fig:scalar2} ($\alpha = 20$). In very short times, we see that $\xi$ is constant over the regions of concentration; running until time $t=200$, we observe that $\xi$ is constant over the entire domain. The solution, during the time spanned from $t=2.7$ to $t=200$, is slowly concentrating around the two peaks dissolving the mass that is still located at $x=0$ and $t=2.7$. We have checked this behavior by plotting the densities in log-scale. We again remark that it is difficult to determine whether this behaviour is due to the no-flux boundary condition used.

Finally, in Figure \ref{fig:scalar4} we have a large amount of self-repulsion ($\alpha = -20$). In this case, the repulsive forces produce a patterned state, at least for small times; however, for larger times we begin to observe a very slow decay to a constant steady state. Similar to Figure \ref{fig:scalar1}, we observe the (numerical) solution to approach $1/2L$. In such cases, we do not believe a no-flux boundary condition will appropriately approximate the solution on the whole space for fixed numerical domain lengths.

These examples, modifying only the strength of nonlocal interaction, highlight some intuitive behaviour. First, when the interaction strength is too small, patterned states cannot form; this is known for smooth potentials in a bounded domain, see \cite{Carrillo2020LongTime}. On the other hand, patterned states can form for either attractive or repulsive kernels on a bounded domain; it may be the case that without confinement by the domain itself, repulsive kernels cannot produce a stationary state on the whole space. As discussed at the beginning of this section, it is unclear which attractive kernels of sufficient irregularity can produce a nontrivial patterned state on the whole space. We leave these questions for future investigation.

\begin{figure}
    \centering
    \subfloat[]{\includegraphics[width=0.85\textwidth]{"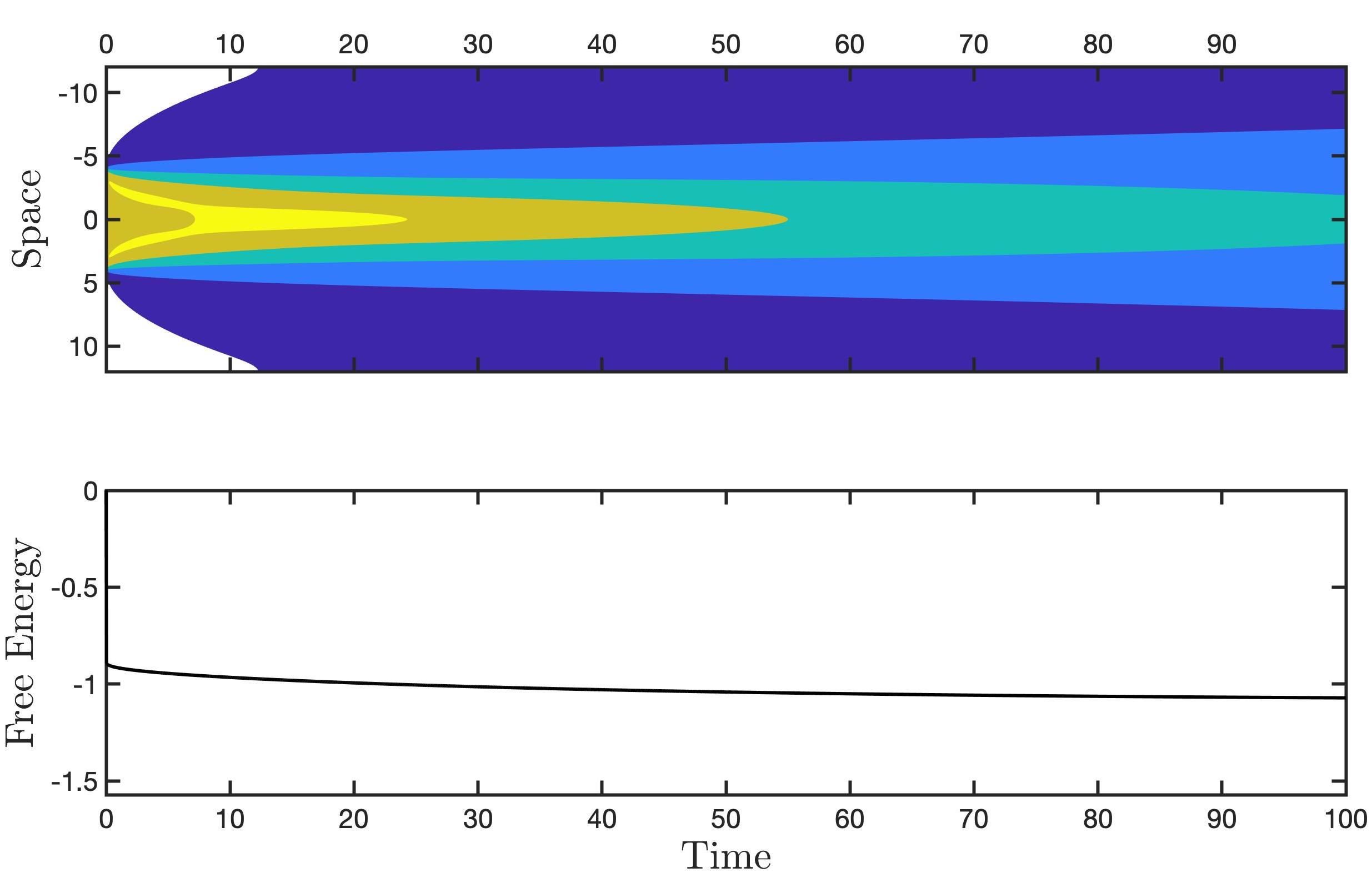"}\label{fig:ScalarContourWeakAttraction}}\\
    \subfloat[]{\includegraphics[width=0.85\textwidth]{"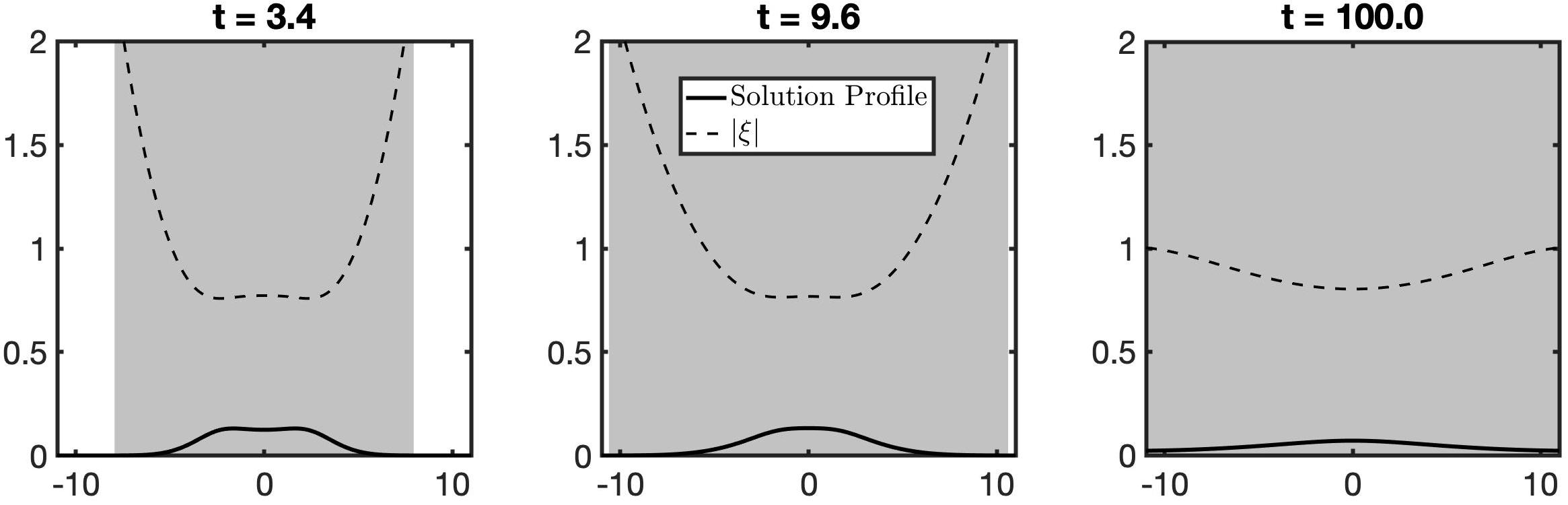"}\label{fig:ScalarFourPanelWeakAttraction}}
    \caption{Simulation of problem \eqref{application:scalartophat} with $\alpha = 2.0$ and numerical domain size of $24$ ($L=12$). A supplementary video of this simulation is hosted on figshare found here: \url{https://doi.org/10.6084/m9.figshare.25942969.v1}}
    \label{fig:scalar1}
\end{figure}

\begin{figure}
    \centering
    \subfloat[]{\includegraphics[width=0.85\textwidth]{"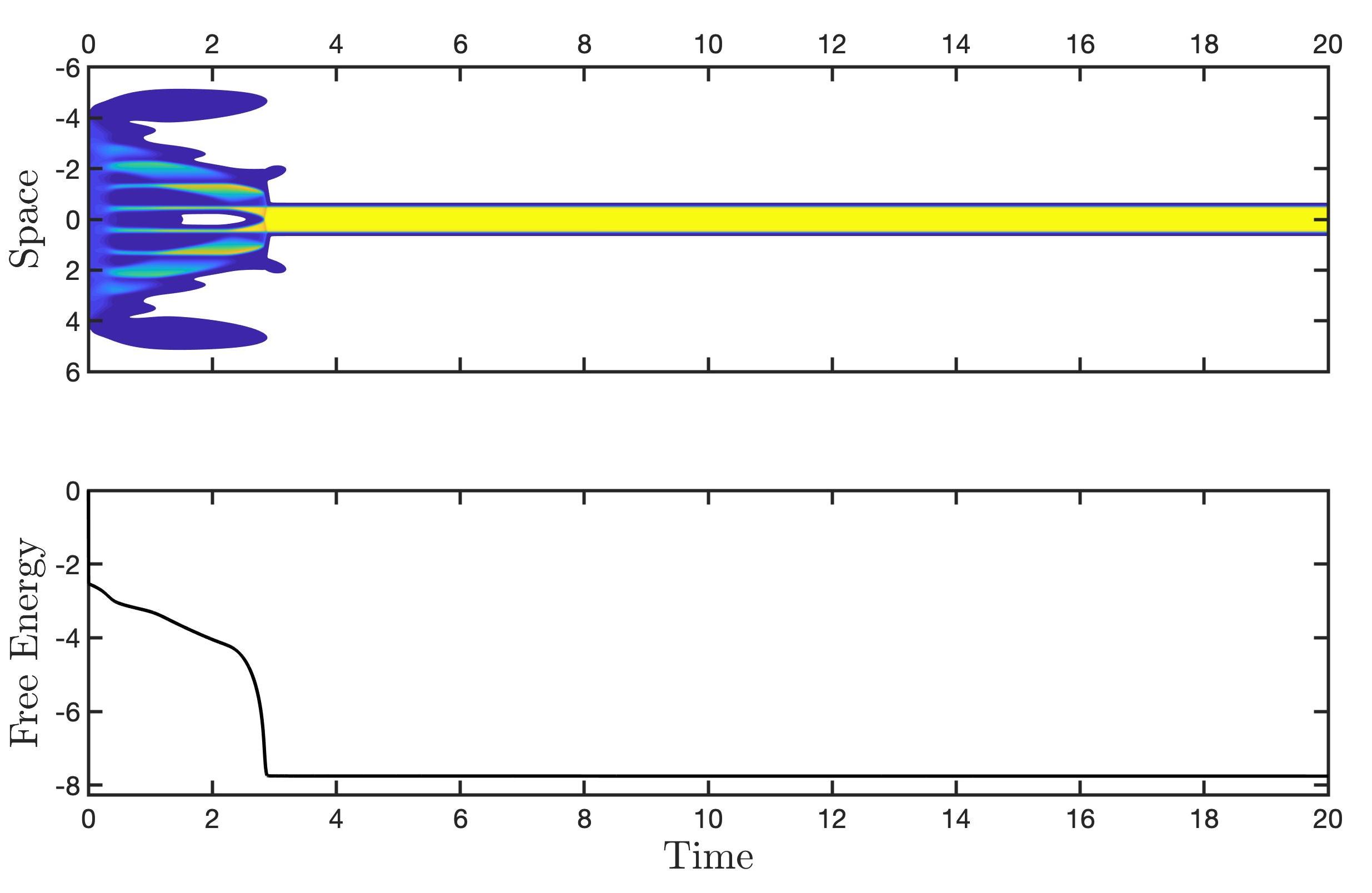"}\label{fig:ScalarContourStrongAttraction1}}\\
    \subfloat[]{\includegraphics[width=0.85\textwidth]{"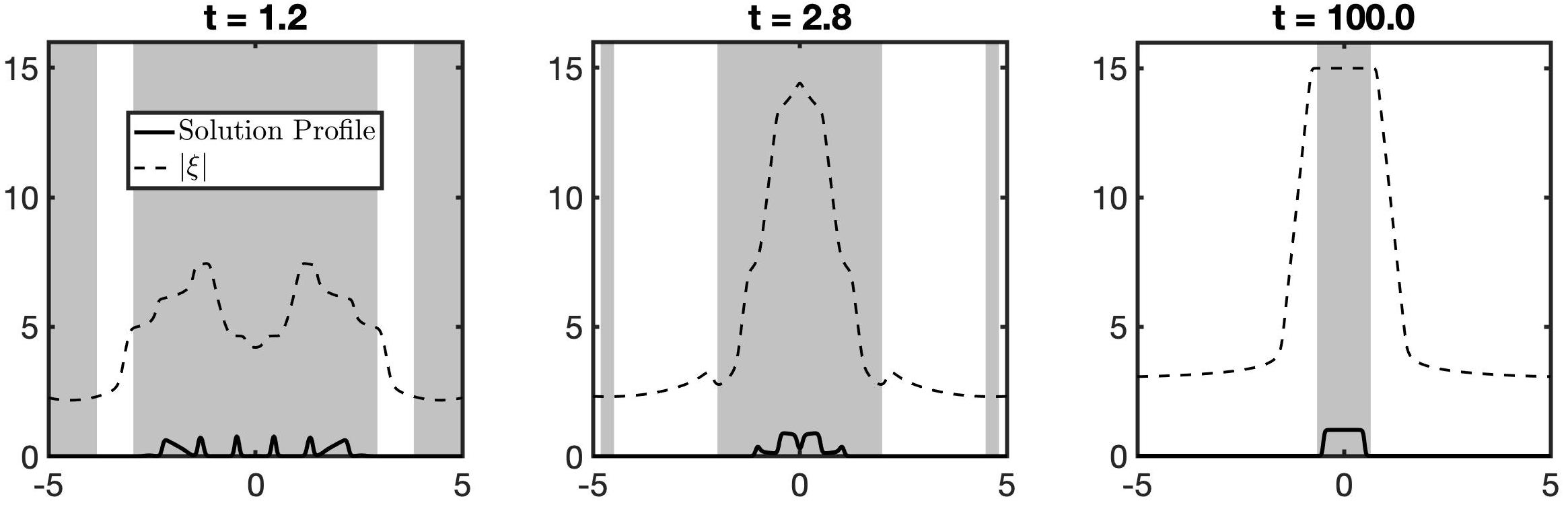"}\label{fig:ScalarFourPanelStrongAttraction1}}
    \caption{Simulation of problem \eqref{application:scalartophat} with $\alpha = 30$ and numerical domain size of $12$ ($L=6$). A supplementary video of this simulation is hosted on figshare found here: \url{https://doi.org/10.6084/m9.figshare.25934431.v1}}
    \label{fig:scalar2}
\end{figure}

\begin{figure}
    \centering
    \subfloat[]{\includegraphics[width=0.85\textwidth]{"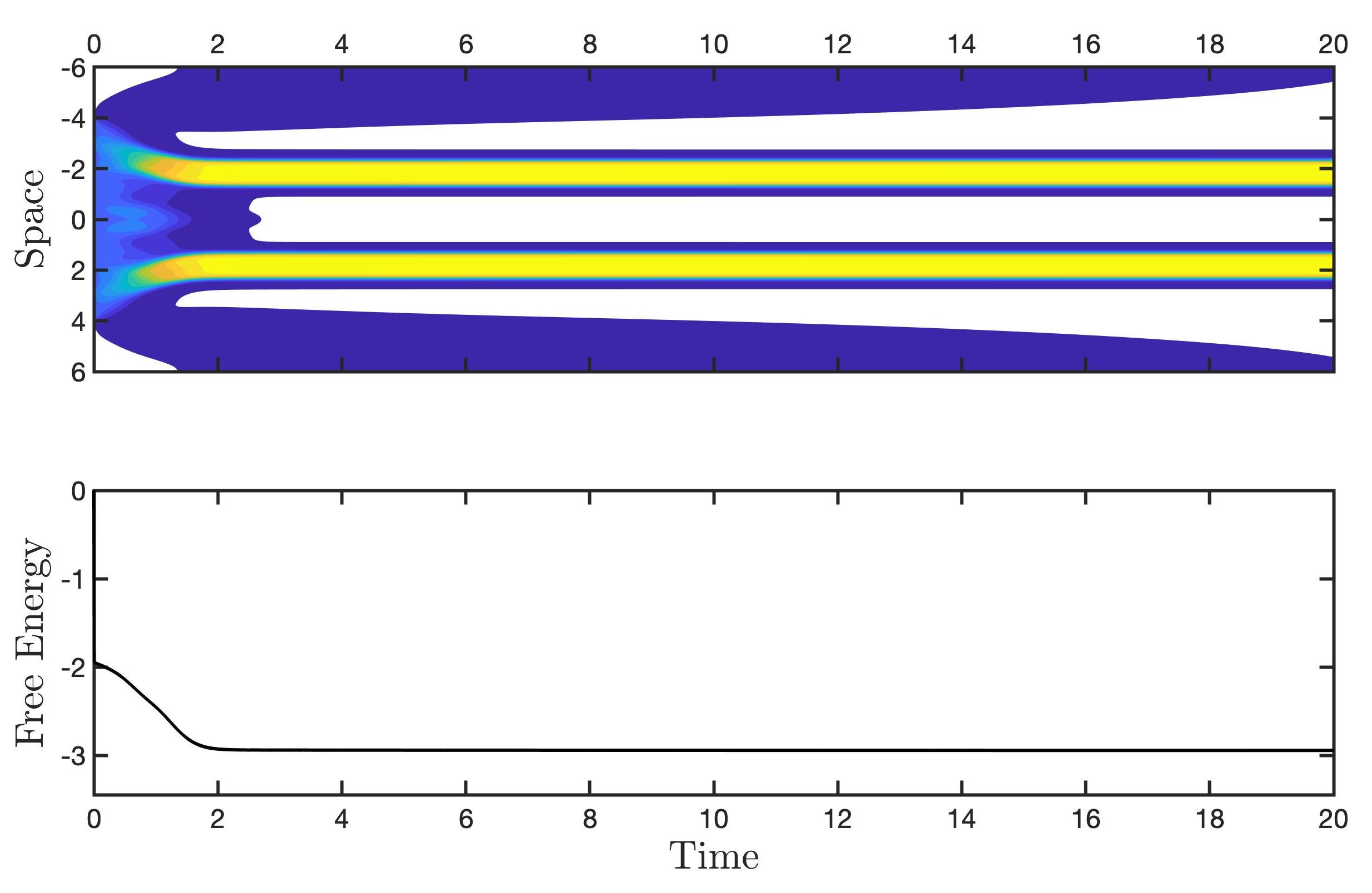"}\label{fig:ScalarContourStrongAttraction2}}\\
    \subfloat[]{\includegraphics[width=0.85\textwidth]{"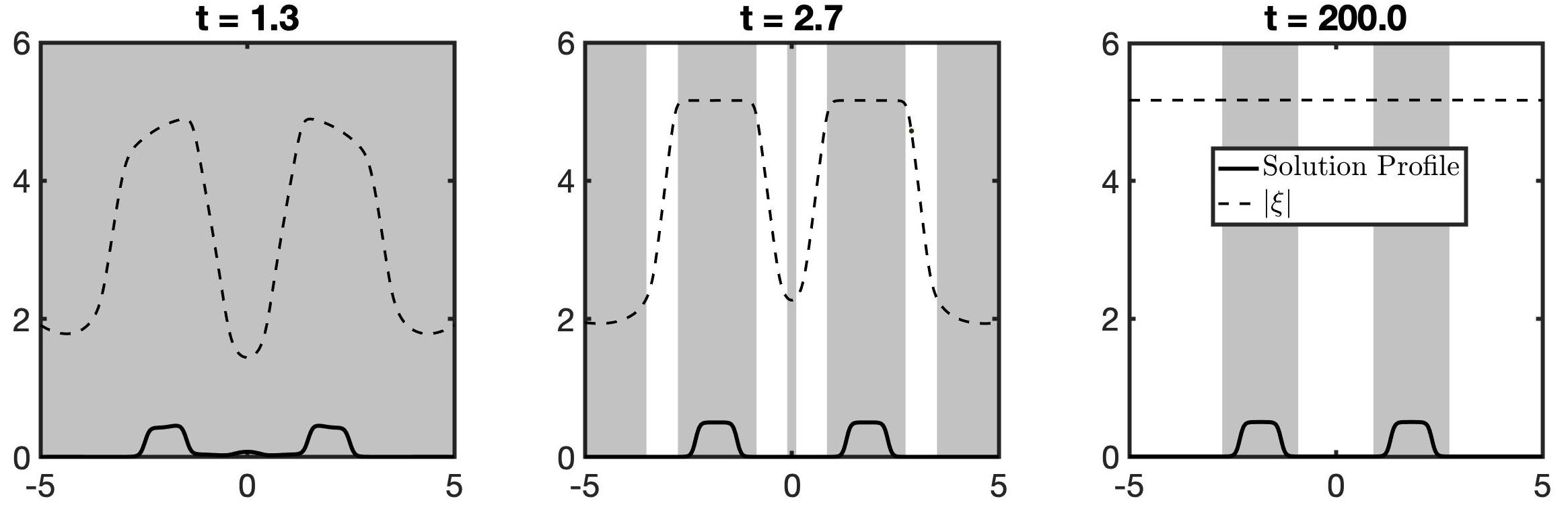"}\label{fig:ScalarFourPanelStrongAttraction2}}
    \caption{Simulation of problem \eqref{application:scalartophat} with $\alpha = 20$ and numerical domain size of $24$ ($L=12$). A supplementary video of this simulation is hosted on figshare found here: \url{https://doi.org/10.6084/m9.figshare.25934425.v1}}
    \label{fig:scalar3}
\end{figure}

\begin{figure}
    \centering
    \subfloat[]{\includegraphics[width=0.85\textwidth]{"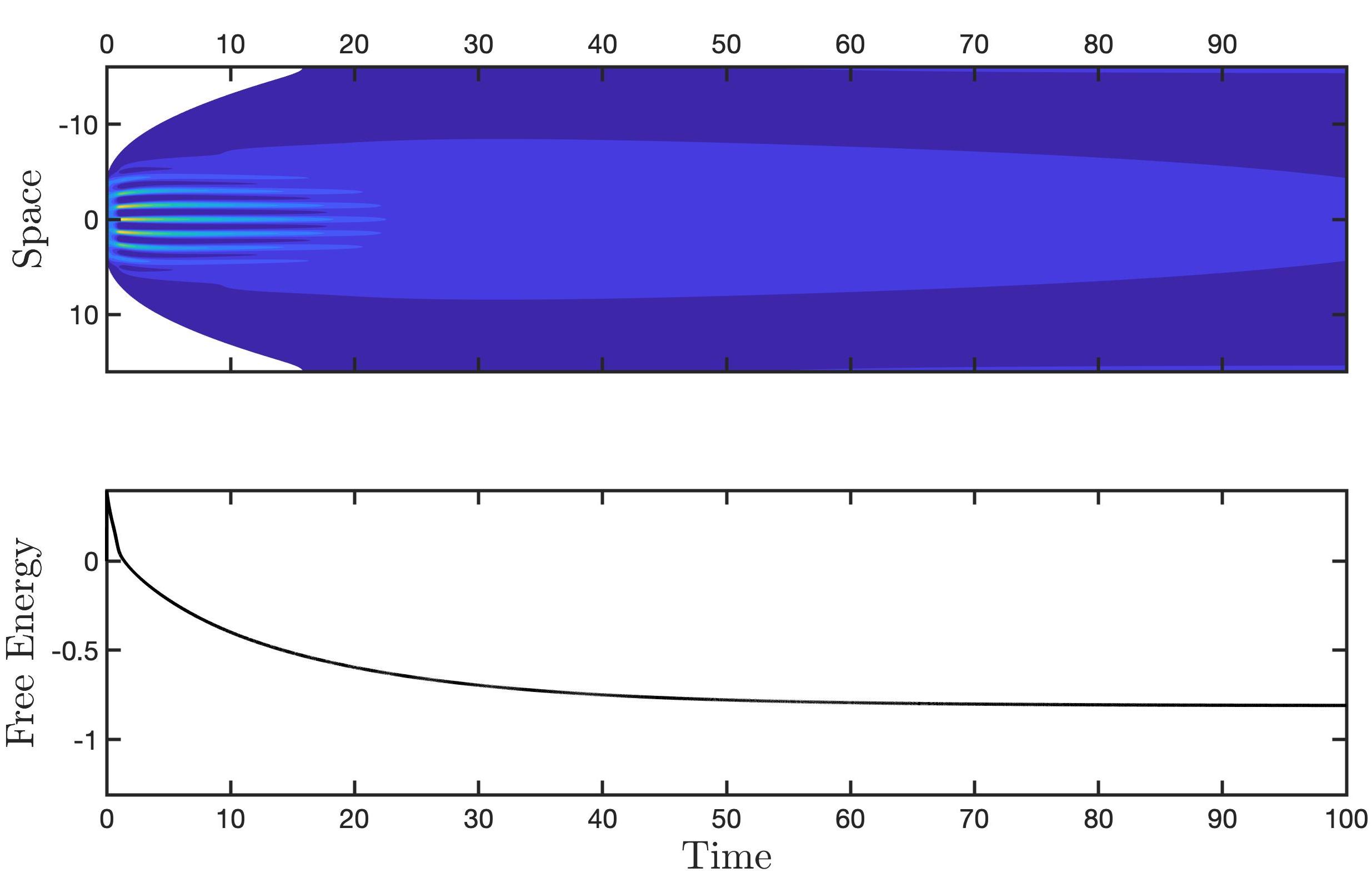"}\label{fig:ScalarContourStrongRepulsion}}\\
    \subfloat[]{\includegraphics[width=0.85\textwidth]{"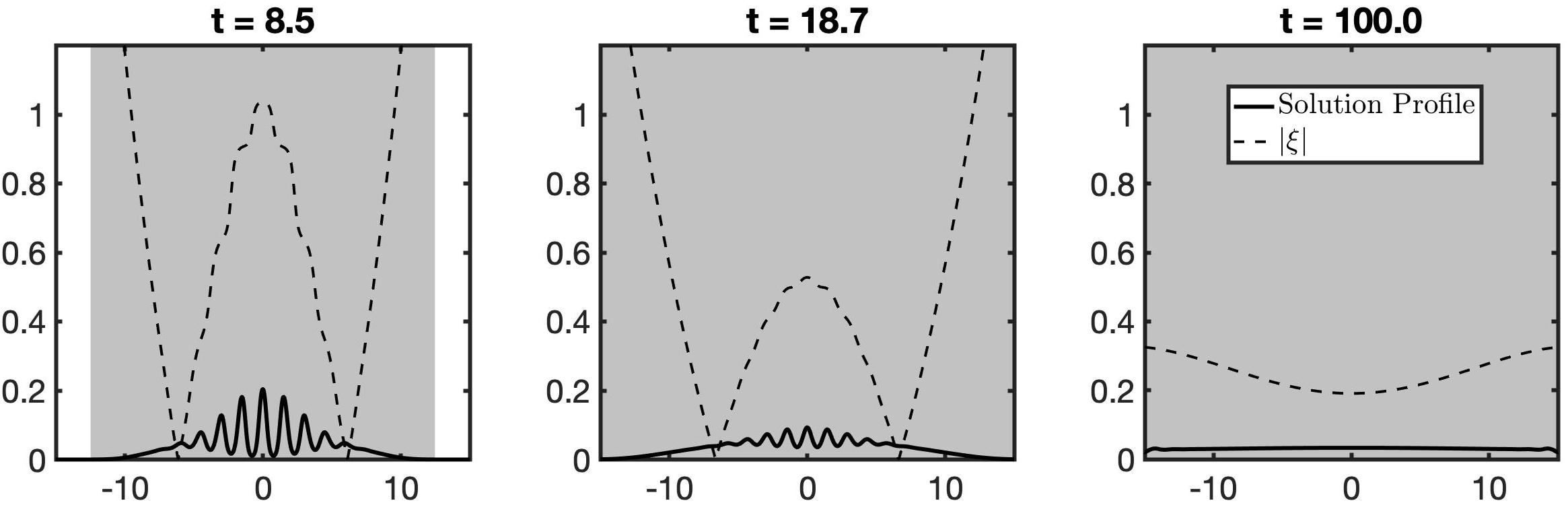"}\label{fig:ScalarFourPanelStrongRepulsion}}
    \caption{Simulation of problem \eqref{application:scalartophat} with $\alpha =-20$ and numerical domain size $32$ ($L=16$). A supplementary video of this simulation is hosted on figshare found here: \url{https://doi.org/10.6084/m9.figshare.25934428.v1}}
    \label{fig:scalar4}
\end{figure}

\subsection{The \texorpdfstring{$n$}{}-Species System}

Different from the scalar equation case, the $n$-species system leaves more room to explore the differences between Theorems \ref{thm:existsmallmasssystem} and \ref{thm:existarbmasssystem} due to the increased complexity of the detailed balance hypothesis \textup{\textbf{\ref{h5}}} for several interacting populations. For simplicity, we first consider the following general $2$-species aggregation-diffusion system
\begin{align}\label{application:systemtophat}
    \begin{cases}
        \frac{\p u}{\p t} = \frac{\p}{\p x} \left( D \frac{\p u}{\p x} + u \frac{\p}{\p x} ( K_{11} * u + K_{12} * v) \right), \quad \text{ in } (0,T) \times \mathbb{R}, \\
        \frac{\p v}{\p t} = \frac{\p}{\p x} \left( D \frac{\p v}{\p x} + v \frac{\p}{\p x} ( K_{21} * u + K_{22} * v) \right), \quad \text{ in } (0,T) \times \mathbb{R} , \\
        u (0,x) = v(0,x) = \chi_{(-\ell,\ell)}(x), \quad\quad\quad\quad\quad\quad\quad\quad\quad\quad\quad \text{ in } \mathbb{R}.
\end{cases}
\end{align}
In these simulations we again fix $D$, $R$ and $\ell$ as in \eqref{eq:numerics_fixed_values} with the initial data normalized so that $m_i =1$ for $i=1,2$. We again use no-flux boundary conditions on a domain of length $2L > 0$ specified in each figure caption.

Under hypothesis \textup{\textbf{\ref{h1}}}, Theorem \ref{thm:existsmallmasssystem} guarantees that there exists a global weak solution $(u,v)$ solving problem \eqref{application:systemtophat} so long as there holds
$$
\frac{1}{2} \left( \norm{K_{12}}_{L^\infty(\mathbb{R}^d)} + \norm{K_{21}}_{L^\infty(\mathbb{R}^d)} \right)  < \min_{i=1,2} \left\{ D - \norm{K_{ii}}_{L^\infty(\mathbb{R}^d)} \right\}.
$$
Notice that if the small mass condition \eqref{const:c1} from the scalar case is violated for either population $u$ or $v$, we cannot guarantee the existence of a solution from Theorem \ref{thm:existsmallmasssystem}; when the small mass condition is satisfied, a global weak solution exists for the $2$-species system so long as the cross-interaction between the two populations (described by $K_{12}$ and $K_{21}$) is not too strong. It is interesting to note that this requires no relation between the kernels of different populations: the kernel $K_{12}$ may be a completely different form than $K_{21}$, such as featuring different detection radii $R$ between populations. We do not explore such possibilities here, however. We instead consider as a concrete example the typical case where
\begin{align}\label{kernel:tophatsystem}
    K_{ij} (\cdot) := \alpha_{ij} K_{\text{tophat}} (\cdot),
\end{align}
for some coefficients $\alpha_{ij} \in \mathbb{R}$, $i,j=1,2$. In this way, both populations have identical detection abilities, differing only in their relative rates of attraction/repulsion. With our chosen parameters, a necessary condition for the existence of a weak solution via Theorem \ref{thm:existsmallmasssystem} is
$$
\as{\alpha_{ii}} < 2 D R = \frac{1}{2},\quad i=1,2.
$$
Then, if the cross interaction is such that $\as{\alpha_{ij}} \ll 1$, $i\neq j$, Theorem \ref{thm:existsmallmasssystem} ensures the existence of a global weak solution; since \textup{\textbf{\ref{h4}}} is also satisfied, Theorem \ref{thm:systemfinal} ensures that this solution is in fact the unique, global classical solution for a rather large class of interaction kernels with no symmetry requirements of any kind. In such cases, numerical simulation reveals behaviour similar to that of a single equation: the cross-interaction cannot be too strong, and so the dynamics are primarily governed by diffusion and self-interaction (i.e., $\alpha_{ii}$).\\

On the other hand, while Theorem \ref{thm:existarbmasssystem} allows for strong cross-interactions, the addition of hypothesis \textup{\textbf{\ref{h5}}} requires some additional structure on the coefficients $\alpha_{ij}$. We explore this briefly below.

For $n=2$ interacting populations, it is not difficult to verify that the condition
$\pi_i \alpha_{ij} = \pi_j \alpha_{ji}$
results in a consistent system so long as the cross-interaction terms $\alpha_{ij}$, $i \neq j$, are of the same sign. Therefore, Theorem \ref{thm:existarbmasssystem} ensures the existence of a global weak solution to problem \eqref{application:systemtophat} for any initial mass $m_i>0$, any $D_i>0$ fixed, and any kernels of the form given in \eqref{kernel:tophatsystem} so long as there holds
$$
\sign \{ \alpha_{ij} \} = \sign \{ \alpha_{ji} \},
$$
for each $i,j=1,\ldots,n$. Moreover, hypothesis \textup{\textbf{\ref{h4}}} is again satisfied by the top-hat detection function and so the solution obtained is in fact the unique, classical solution. 

In Figure \ref{fig:system1} we present the solution profiles for the two-species system with repulsive symmetric cross-interaction, i.e. $\alpha_{12} = \alpha_{21} < 0$. We choose $\alpha_{11} = 20$ so that the expected solution profile for $u$ is that of Figure \ref{fig:scalar3} in the absence of $v$; we then choose $\alpha_{22} = 2$ (an order of magnitude less than $\alpha_{11}$), so that the expected solution profile of $v$ is that of Figure \ref{fig:scalar1} in the absence of $u$. It is interesting to note that the solution behaviour is roughly as expected for each single equation with the no-flux boundary conditions, where the steady-state for $u$ concentrates into two peaks as indicated by the flatness of $\xi_u$ on the essential support of $u$ and the solution profile for $v$ stabilizing towards a steady state as much as $u$ by noticing the constant behavior of $\xi_v$ on its essential support. We observe the segregation of the mass of both species as well as the decay of the free energy functional in time. We remark that we expect the long time behavior of the solutions in the whole space completely different from the no-flux boundary conditions. Clarifying this point would need further numerical exploration.

In Figure \ref{fig:system2}, we run a simulation for a scenario that is not covered by the theory presented here. We fix $\alpha_{ii} = 20$, $i=1,2$, so that a patterned state of the form in Figure \ref{fig:scalar3} for each population in the absence of the other is expected. We then choose $\alpha_{12} = -10 < \alpha_{21} = 5$ so that the detailed balance condition of Hypothesis \textbf{\ref{h5}} is violated. Over short times scales some complex rearrangement of the two population densities occurs, but eventually the solution profile settles down to a form found in Figures \ref{fig:scalar2}-\ref{fig:scalar3}. The flatness of $\xi_u$, $\xi_v$ in the grey regions suggests this is indeed the steady-state, and this continues to flatten over a longer simulation time. Of course, this case no longer features a gradient flow structure, and so we observe non-monotone behaviour of the ``free energy'' of this simulation. It is interesting to note that the behaviour of the solution does not appear fundamentally different than that of cases falling within the theory of well-posedness obtained in the present work. Therefore, we believe that the detailed balance condition is a technical requirement for the techniques of the scalar equation to be applied to the general $n$-species system, and the system should remain well-posed with no further condition on the sign or magnitude of the coefficients $\alpha_{ij}$. 

\begin{figure}
    \centering
    \subfloat[]{\includegraphics[width=0.85\textwidth]{"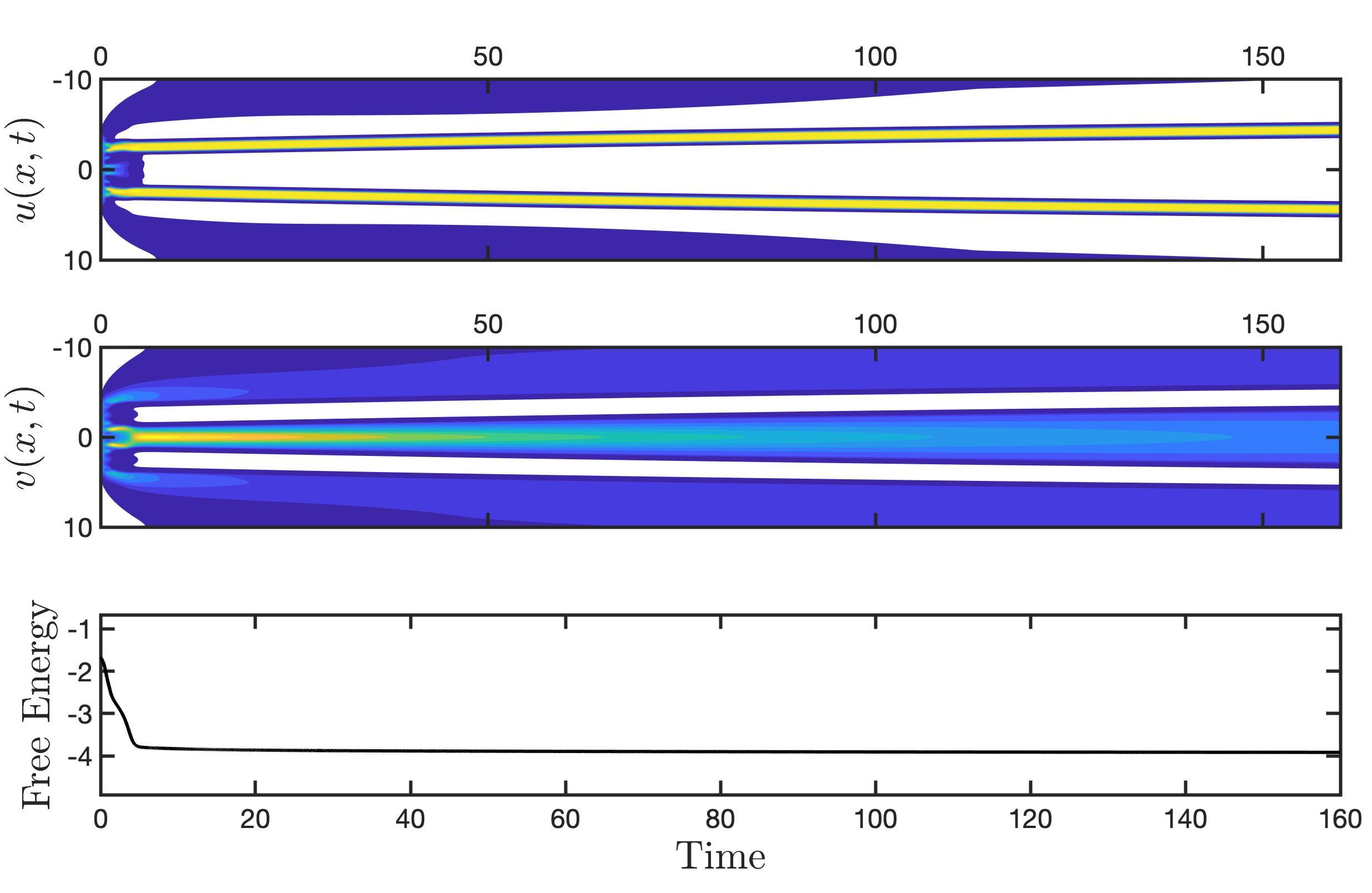"}\label{fig:2SpeciesContourStrongRepulsion}}\\
    \subfloat[]{\includegraphics[width=0.85\textwidth]{"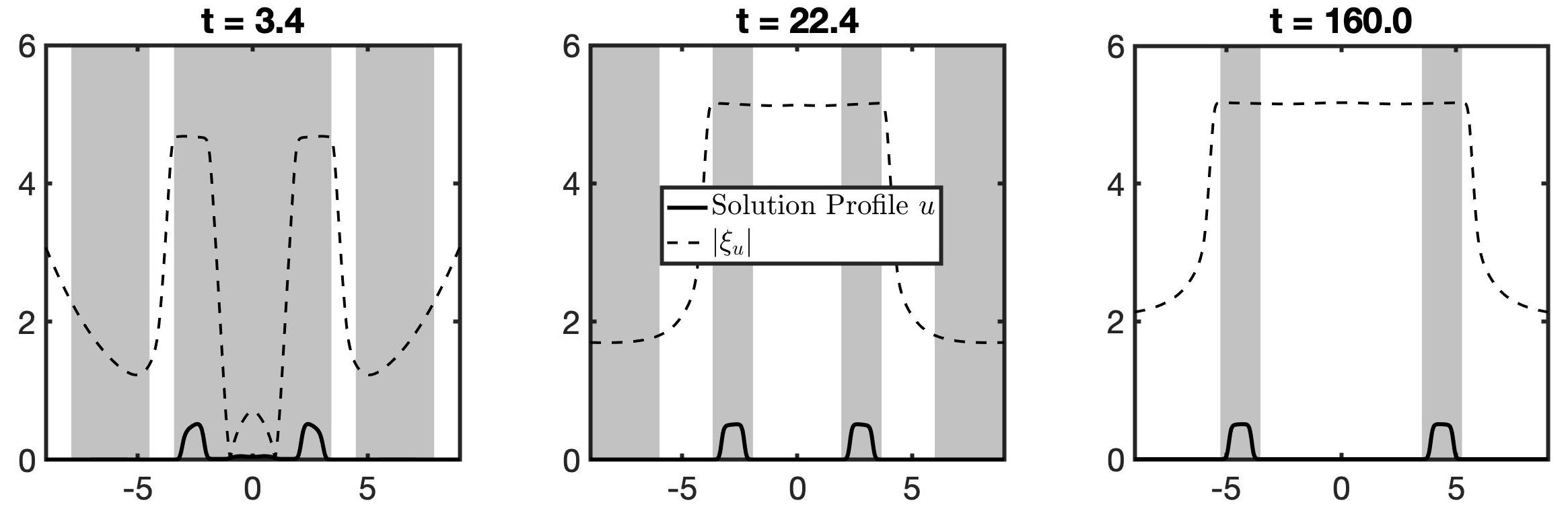"}\label{fig:2SpeciesThreePanelModerateRepulsionU}}\\
    \subfloat[]{\includegraphics[width=0.85\textwidth]{"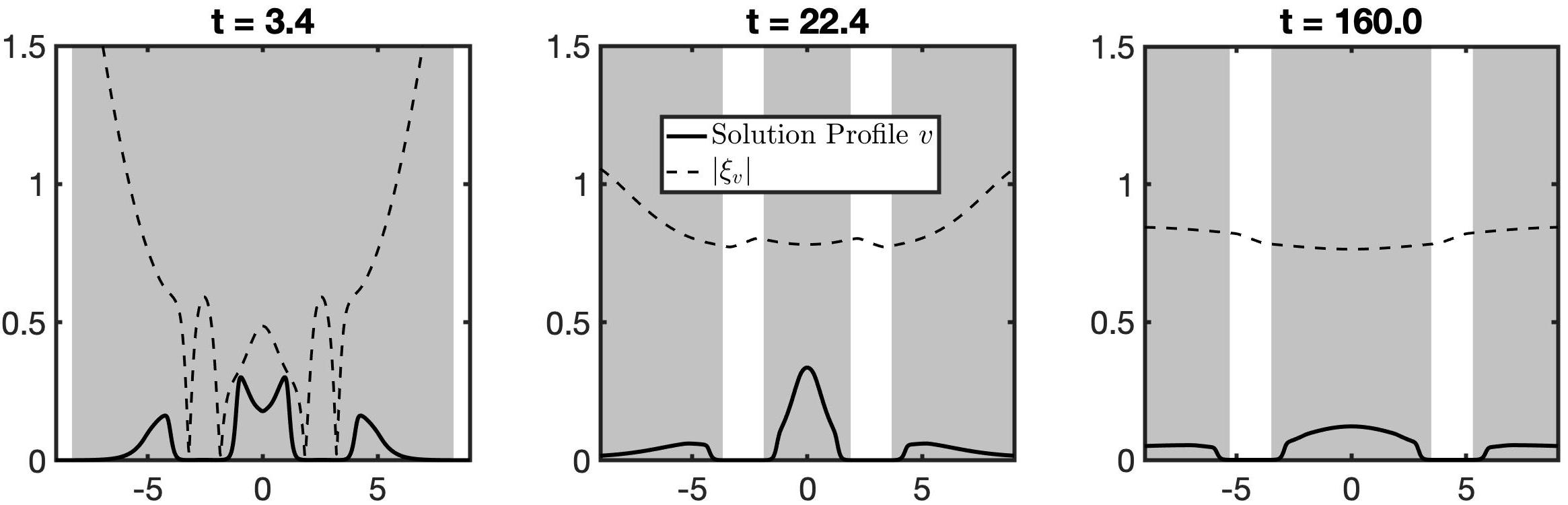"}\label{fig:2SpeciesThreePanelModerateRepulsionV}}
    \caption{Simulation of problem \eqref{application:systemtophat} with $\alpha_{11} =20$, $\alpha_{12} = \alpha_{21} = -10$, and $\alpha_{22} = 2$, with numerical domain length $20$ ($L=10$). A supplementary video of this simulation is hosted on figshare found here: \url{https://doi.org/10.6084/m9.figshare.25943101.v1}}
    \label{fig:system1}
\end{figure}

\begin{figure}
    \centering
    \subfloat[]{\includegraphics[width=0.85\textwidth]{"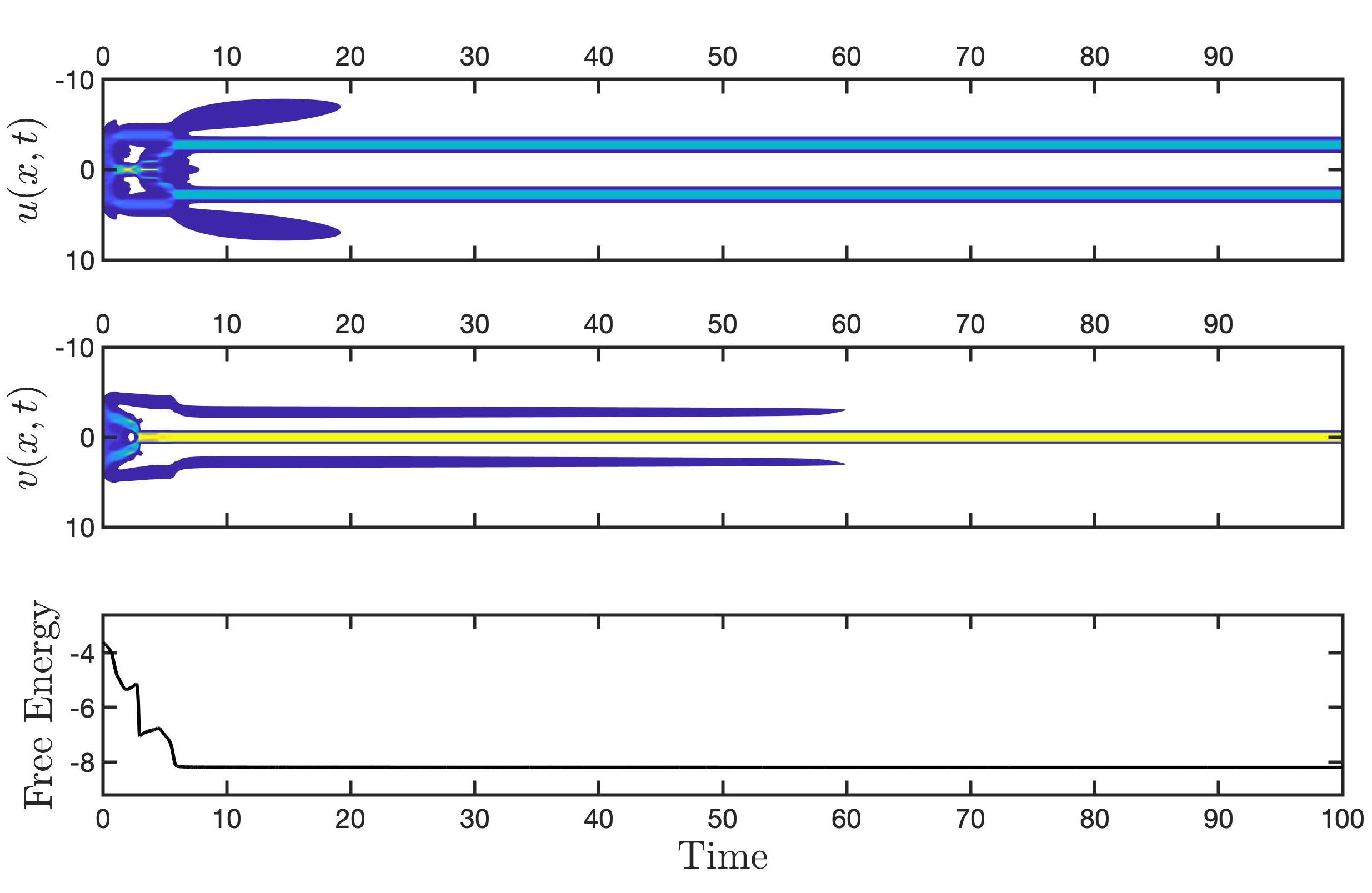"}\label{fig:2SpeciesContourRunChase}}\\
    \subfloat[]{\includegraphics[width=0.85\textwidth]{"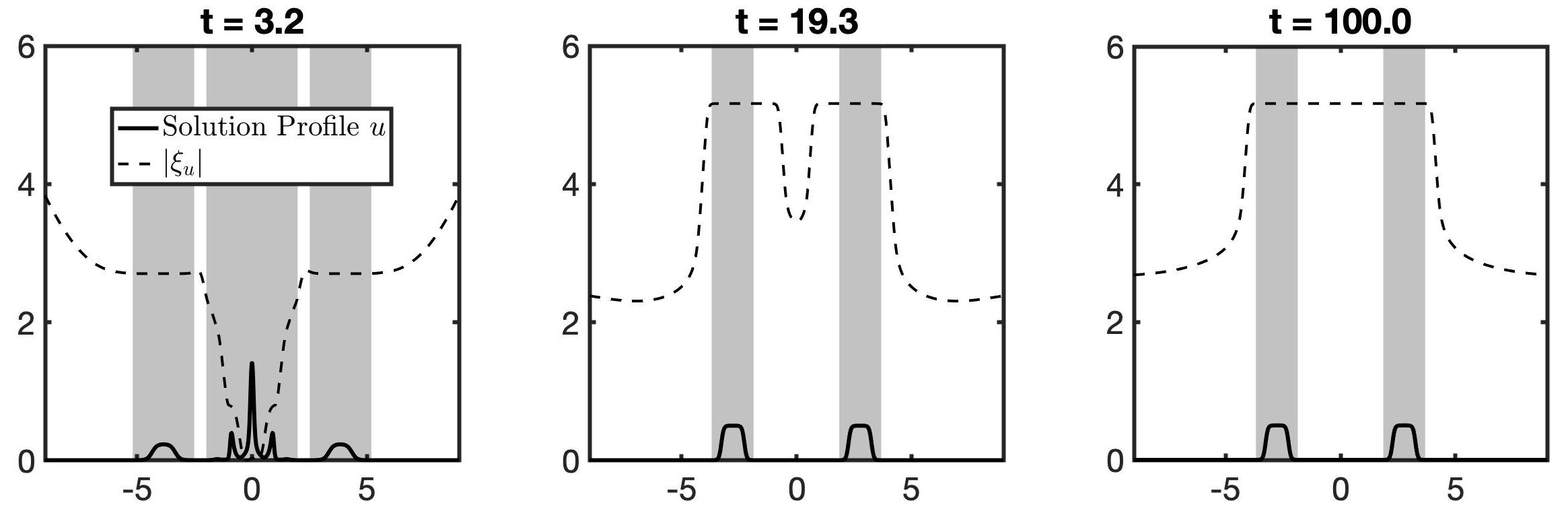"}\label{fig:2SpeciesThreePanelRunChaseU}}\\
    \subfloat[]{\includegraphics[width=0.85\textwidth]{"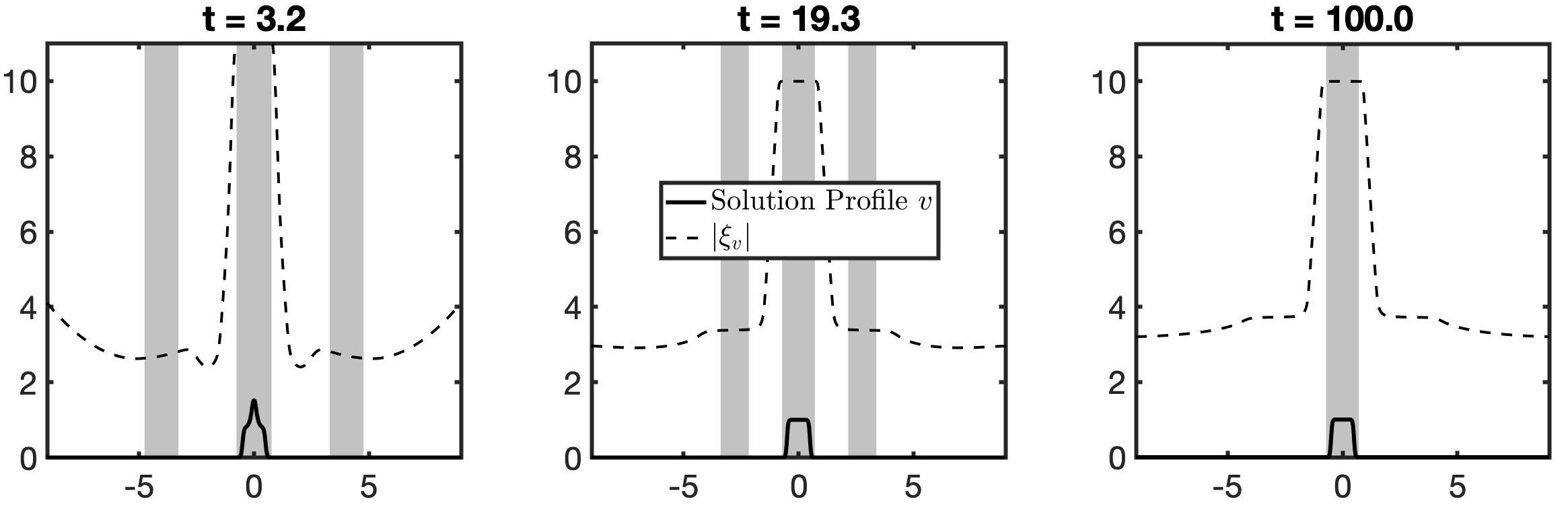"}\label{fig:2SpeciesThreePanelRunChaseV}}
    \caption{Simulation of problem \eqref{application:systemtophat} with $\alpha_{11} = \alpha_{22} = 20$, $\alpha_{12} = -10$, and  $\alpha_{21} = 5$, with numerical domain length $20$ ($L=10$). A supplementary video of this simulation is hosted on figshare found here: \url{https://doi.org/10.6084/m9.figshare.25943098.v1}}
    \label{fig:system2}
\end{figure}

When $n\geq3$, even cross-interactions of the same sign do not guarantee the existence of a solution via Theorem \ref{thm:existarbmasssystem}. To see this, we note that generically the condition $\pi_i \alpha_{ij} = \pi_j \alpha_{ji}$ gives at most $n(n-1)/2$ equations with $n(n-1)$ unknowns (disregarding the diagonal elements). Hence the detailed balance condition hypothesis \textup{\textbf{\ref{h5}}} becomes significantly under-determined as the number of interacting species grows. This highlights an interesting consequence of the gradient flow structure for many interacting populations, leaving open the question of well-posedness for these systems without structural requirements between the interaction kernels for more than $4$ interacting populations. 

Pairing our analytical insights with our numerical observations, we propose the following questions for future interest.
\begin{enumerate}
    \item For the scalar equation \eqref{eq:mainscalar}, when the kernel $K \in L^\infty(\mathbb{R}^d)$ satisfies \textbf{\ref{h2}}, but otherwise features no additional regularity, do solutions decay to $0$ in the long-time limit? More generally, what are minimal regularity requirements on $K$, which do not require $\grad K$ to belong to some $L^p$ space, that yield no nontrivial stationary states to the scalar equation on the whole space?
    \item For the $n$-species system \eqref{maineq}, when all kernels satisfy \textbf{\ref{h2}} but the detailed balance condition \textbf{\ref{h5}} is violated, does a unique strong solution exist? Our numerical simulations suggest that the answer is \textit{yes}, but it is not immediately clear how to extend the current approach to such cases.
    \item Relating the scalar and system cases, given minimal conditions on $K$ such that the solution to the scalar equation tends to $0$ in the long-time limit, do solutions to the $n$-species system enjoy the same property? It appears that for cases with an energy functional (i.e., detailed balance is satisfied), this should be the case; however, it is not at all obvious whether this remains true when the energy structure is removed.
\end{enumerate}

\appendix
\section{Extension of Theorem \ref{thm:existsmallmassscalar} to the positive-definite and Keller-Segel cases}\label{app:extension_general_kernel}
We present here a result which shows that the crucial estimate in Lemma \ref{lem:iscalar} can be combined with the positive-definite case in \cite{jungel2022nonlocal} and the Keller-Segel case in \cite{MR2103197} (the latter only in dimension $d=2$) to consider even more general kernels $K$.
\begin{lemma}
Let $u$ be the smooth and strictly positive solution of \eqref{eq:mainscalar} with $u_0 \in L^1(\Rd) \cap C^{\infty}(\Rd)$, $u_0>0$, and with the kernel $K = K_1 + K_2 + K_3$ where $K_1 \in L^{\infty}(\Rd)$, $\Delta K_2 = \mu \in \mathcal{M}(\Rd)$ is a Radon measure with bounded positive part $\mu^+$ satisfying $\|\mu^+\|_{\text{TV}}<\infty$ while $K_3 \in L^{\infty}(\Rd)$ is positive-definite in the sense that for all $\xi:\Rd \to \mathbb{R}$ with $\xi \in L^{1}(\Rd)$, we have $\int_{\Rd} \xi(x)\, K_3\ast \xi(x) \diff x \geq 0$. Moreover, if $\mu^+ \neq 0$, we assume $d=2$ and we let $C_{GN}$ to be the best constant in the Gagliardo-Nirenberg inequality $\|f\|_{L^4(\mathbb{R}^2)} \leq C_{GN} \|\nabla f\|^{1/2}_{L^2(\mathbb{R}^2)}\, \|f\|^{1/2}_{L^2(\mathbb{R}^2)}$ for all $f \in H^1(\mathbb{R}^2)$. Let $m = \int_{\Rd} u_0(x)\diff x$. Then, if
$$
D-m\, \|K_1\|_{L^{\infty}(\Rd)}-\frac{C_{\text{GN}}^4\, m\,\|\mu^+\|_{\text{TV}}}{4}>0,
$$
we have the estimate
$$
\partial_t \int_{\Rd} u\log u \diff x + \left(D-m\, \|K_1\|_{L^{\infty}(\Rd)}-\frac{C_{\text{GN}}^4\, m\,\|\mu^+\|_{\text{TV}}}{4}\right)\,\int_{\Rd} \frac{|\nabla u|^2}{u} \diff x \leq 0.
$$
\end{lemma}
\begin{proof}
We multiply \eqref{eq:mainscalar} by $\log u$ and we integrate by parts to obtain
\begin{multline*}
\partial_t \int_{\Rd} u\log u \diff x + D\,\int_{\Rd} \frac{|\nabla u|^2}{u} \diff x = -\int_{\Rd} K\ast \nabla u \, \nabla u \diff x = \\ =
-\int_{\Rd} K_1\ast \nabla u \, \nabla u \diff x -\int_{\Rd} K_2\ast \nabla u \, \nabla u \diff x
-\int_{\Rd} K_3\ast \nabla u \, \nabla u \diff x.
\end{multline*}
The first term is controlled as in Lemma \ref{lem:iscalar} to obtain
\begin{equation}\label{eq:estimate_K1_small_mass_dec_K1_K2_K3}
-\int_{\Rd} K_1\ast \nabla u \, \nabla u \diff x \leq \|K_1\|_{L^{\infty}(\Rd)} \, \|\nabla u\|^2_{L^1(\Rd)} \leq m\, \|K_1\|_{L^{\infty}(\Rd)} \, \int_{\Rd} \frac{|\nabla u|^2}{u} \diff x.
\end{equation}
The second term is controlled as in \cite{MR2103197} 
$$
-\int_{\Rd} K_2\ast \nabla u \, \nabla u \diff x = \int_{\Rd} \Delta K_2\ast u \,  u \diff x \leq  \int_{\Rd} \mu^+ \ast u \,  u \diff x \leq \|\mu^+\|_{\text{TV}} \int_{\Rd} u^2 \diff x,
$$
where we used the H{\"o}lder inequality and $\|\mu^+\ast u\|_{L^2(\Rd)} \leq \|\mu^+\|_{\text{TV}}\, \|u\|_{L^2(\Rd)}$. Then, we use the Gagliardo-Nirenberg inequality from the statement with $f = \sqrt{u}$ to obtain
\begin{equation}\label{eq:estimate_K2_Keller_Segel_using_GN}
-\int_{\Rd} K_2\ast \nabla u \, \nabla u \diff x \leq \frac{C_{\text{GN}}^4\, m\,\|\mu^+\|_{\text{TV}}}{4}\int_{\Rd} \frac{|\nabla u|^2}{u} \diff x.
\end{equation}
We stress that the manipulation is valid only in dimension $d=2$ if $\mu^+ \neq 0$ (see Remark \ref{rem:appendix_keller_segel_comment} below). Finally, $-\int_{\Rd} K_3\ast \nabla u \, \nabla u \diff x\leq 0$ and we conclude the proof. 
\end{proof}

\begin{remark}
The regularity of $K_3$ is mainly needed to make the integral $\int_{\Rd} K_3\ast \nabla u \, \nabla u \diff x$ well-defined and to pass to the limit in this term in the approximating scheme when proving the existence of solution. The basic estimate on $\int_{\Rd} \frac{|\nabla u|^2}{u} \diff x$ implies that $\nabla u \in L^2(0,T;L^1(\Rd))$ so $K_3 \in L^{\infty}(\Rd)$ is sufficient. On the other hand, one could potentially bootstrap the estimates on $u$ and $\nabla u$ in the spirit of Lemma \ref{lem:iscalar2} to weaken the required assumption on $K_3$.
\end{remark}

\begin{remark}\label{rem:appendix_keller_segel_comment}
If $\mu^+ \neq 0$, we have to restrict to $d=2$. First, the inequality fails in dimensions $d>4$ since the Sobolev exponent satisfies $\frac{2d}{d-2} < 4$ for $d>4$. Second, even for $d \in \{3,4\}$, the inequality reads $\|f\|_{L^4(\mathbb{R}^d)} \leq C_{GN} \|\nabla f\|^{d/4}_{L^d(\mathbb{R}^2)}\, \|f\|^{(4-d)/4}_{L^2(\mathbb{R}^d)}$ so the computation in \eqref{eq:estimate_K2_Keller_Segel_using_GN} yields a bound on $-\int_{\Rd} K_2\ast \nabla u \, \nabla u \diff x$ in terms of $\left(\int_{\Rd} \frac{|\nabla u|^2}{u} \diff x\right)^{d}$ which cannot be compensated by the (LHS). Another approach would be to study $\partial_t \int_{\Rd} u^{d/2} \diff x$, which is an important quantity for the Keller-Segel system in an arbitrary dimension $d$ \cite{MR2099126}; however, mimicking \eqref{eq:estimate_K1_small_mass_dec_K1_K2_K3}, one has to estimate $\int_{\Rd} \nabla u^{\frac{d}{2}} \, \nabla K_1\ast u \diff x$ in terms of $\int_{\Rd} |\nabla u^{\frac{d}{4}}|^2 \diff x$ which does not seem easy.
\end{remark}

\begin{remark}
It is not clear if we can extend in the similar way Lemma \ref{lem:iscalararbmass} which is based on the fact that $f:= \sqrt{u}\, \grad (D \log u + K  * u)$ is bounded in $L^2(Q_T)$. Decomposing $K =K_1 + K_2 + K_3$, it is not clear how to use the information on $\Delta K_2$ to control $\sqrt{u}\, \grad K_2\ast u$ or the positive-definiteness of $K_3$ to control $\sqrt{u}\, \grad K_3\ast u$. 
\end{remark}

\subsection*{Acknowledgements}
JAC and JS were supported by the Advanced Grant Nonlocal-CPD (Nonlocal PDEs for Complex Particle Dynamics: Phase Transitions, Patterns and Synchronization) of the European Research Council Executive Agency (ERC) under the European Union’s Horizon 2020 research and innovation programme (grant agreement No. 883363). JAC was also partially supported by the EPSRC grant numbers EP/T022132/1 and EP/V051121/1 and by the “Maria de Maeztu” Excellence Unit IMAG, reference CEX2020-001105-M, funded by MCIN/AEI/10.13039/501100011033/. YS is supported by Natural Sciences and Engineering Research Council of Canada (NSERC Grant PDF-578181-2023). We also wish to acknowledge the helpful remarks of Alessandro Cucinotta, Charles Elbar and Filippo Santambrogio which greatly simplified the paper. Finally, we wish to acknowledge comments of the Anonymous Reviewers which greatly improved the paper.

\bibliographystyle{abbrv}
\bibliography{references}

\end{document}